\documentclass[11pt,a4paper]{scrartcl}
\usepackage[toc,page]{appendix}
\usepackage[T1]{fontenc}
\usepackage[utf8]{inputenc}
 
\usepackage{epsf}
\usepackage{geometry} 
\usepackage{listings} 
\usepackage{algorithmic,algorithm}
\usepackage{multirow}
\usepackage{enumerate}
\usepackage{booktabs}

\usepackage[pdftex]{graphicx}
\usepackage{amscd}
\usepackage[pdftex]{color} 
\usepackage{lscape}
\usepackage{indentfirst}
\usepackage{latexsym}
\usepackage{amsmath, amsfonts, amssymb,mathrsfs}
\usepackage{accents}
\usepackage[pdftex,colorlinks]{hyperref}
\usepackage{verbatim}
\usepackage{bm}

\usepackage{todonotes,soul}

\usepackage{amsthm}
\usepackage{textcomp}
\usepackage{mathtools}
\usepackage{empheq}

\definecolor{orange2}{RGB/HSB}{255,108,0/120,145,225}

\usepackage{tabularx}
\usepackage{arydshln}

\newcommand{\vertiii}[1]{{\left\vert\kern-0.25ex\left\vert\kern-0.25ex\left\vert #1 
    \right\vert\kern-0.25ex\right\vert\kern-0.25ex\right\vert}}

\hypersetup{
    bookmarks=true,         
    unicode=false,          
    pdftoolbar=true,        
    pdfmenubar=true,        
    pdffitwindow=true,      
    pdftitle={},    
    pdfauthor={WG Kraus},     
    pdfsubject={},   
    pdfnewwindow=true,      
    pdfkeywords={}, 
    colorlinks=true,       
    linkcolor=red,          
    citecolor=blue,        
    filecolor=magenta,      
    urlcolor=cyan           
}

\everymath{\displaystyle}
\newtheorem{theorem}{Theorem}
\newtheorem{lemma}[theorem]{Lemma}
\newtheorem{corollary}[theorem]{Corollary}
\newtheorem{proposition}[theorem]{Proposition}
\newtheorem{problem}[theorem]{Problem}

\theoremstyle{remark}

\theoremstyle{definition}

\numberwithin{equation}{section} \numberwithin{table}{section}
\numberwithin{theorem}{section}

\newcommand{\triplenorm}[1]{\ensuremath{|\!|\!| #1 |\!|\!|}}

\def\bf{\boldsymbol f}

\newcommand{\laplace}{\Delta}
\newcommand{\boundary}{\partial}

\newcommand{\restr}[1]{\big|_{#1}} 
\newcommand{\closure}[1]{\overline{#1}}

\renewcommand{\vec}[1]{\mathbf{#1}}
\newcommand{\mat}[1]{\mathbf{#1}}

\newcommand{\abs}[1]{\lvert #1 \rvert}
\newcommand{\norm}[1]{\lVert #1 \rVert}
\newcommand{\scalar}[2]{\left( #1, #2 \right)}
\newcommand{\multiscalar}[2]{( #1, #2 )}
\newcommand{\dual}[2]{\langle #1, #2 \rangle}
\newcommand{\set}[2]{\left\lbrace #1 : #2 \right\rbrace}

\newcommand{\Tau}{\mathcal{T}}
\newcommand{\field}[1]{\mathbb{#1}}
\newcommand{\N}{\field{N}}
\newcommand{\R}{\field{R}}

\newcommand{\dx}{\, \mathrm{d}\vec{x}}
\newcommand{\ds}{\, \mathrm{d}s}
\newcommand{\dt}{\, \mathrm{d}t}

\usepackage{siunitx}
\newcommand*\samethanks[1][\value{footnote}]{\footnotemark[#1]}

\begin{document}

\title{$C^1$-conforming variational discretization of the biharmonic wave equation}

\author{Markus Bause \thanks{Helmut Schmidt University, Faculty of Mechanical Engineering, Holstenhofweg 85, 22043 Hamburg, Germany} 
\and Maria Lymbery \thanks{University of Duisburg-Essen, Faculty of Mathematics, Thea-Leymann-Str. 9, 45127 Essen, Germany} \thanks{Corresponding author. E-mail address: \href{mailto:maria.lymbery@uni-due.de}{maria.lymbery@uni-due.de}} 
\and Kevin Osthues \samethanks[2]
 }

\date{\today}


\maketitle

\begin{abstract}
Biharmonic wave equations are of importance to various 
applications including thin plate analyses. In this work, the numerical approximation of their solutions by a $C^1$-conforming in space and time finite element approach is proposed and analyzed. Therein, the smoothness properties of solutions to the continuous evolution problem is embodied.
High potential of the presented approach for more sophisticated multi-physics and multi-scale systems is expected. 
Time discretization is based on a combined Galerkin and collocation technique. For 
space discretization the Bogner--Fox--Schmit element is applied. Optimal order error estimates are proven. 
The convergence and performance properties are illustrated 
with numerical experiments.      	
\end{abstract}



\section{Introduction}\label{sec:intro}

In this work we propose and analyze a space-time finite element approximation by $C^1$-conforming in space and time discrete functions of the initial-boundary value problem for the biharmonic wave equation,  
\begin{subequations}
    \label{eq:instationary_plate}
    \begin{align}
        \partial_{tt} u(\vec{x}, t) + \Delta^2 u(\vec{x}, t) & =  f(\vec{x}, t), & & \hspace*{-3cm} \text{in} \ \Omega \times \left(0, T\right], \label{eq:instationary_plate_a} \\
        u(\vec{x}, 0) & = u_{0}(\vec{x}), & & \hspace*{-3cm} \text{in} \ \Omega, \label{eq:instationary_plate_b} \\
        \partial_{t} u(\vec{x}, 0) & = u_{1}(\vec{x}), & & \hspace*{-3cm} \text{in} \ \Omega, \label{eq:instationary_plate_c} \\
        u(\vec{x}, t) & = 0, & & \hspace*{-3cm} \text{on} \ \partial \Omega \times \left(0, T\right], \label{eq:instationary_plate_d} \\
        \partial_{\vec{n}} u(\vec{x}, t) & = 0, & & \hspace*{-3cm} \text{on} \ \partial \Omega \times \left(0, T\right], \label{eq:instationary_plate_e}
    \end{align}
\end{subequations}
for a bounded domain $\Omega \subset \R^2$. This model is encountered in the modeling of various physical phenomena, such as plate bending and thin plate elasticity. The dynamic theory of thin Kirchhoff--Love plates investigates the propagation of waves in the plates as well as standing waves and vibration modes. Moreover, the system \eqref{eq:instationary_plate} can be studied as a prototype model for more sophisticated Kirchhoff-type equations, such as the Euler--Bernoulli equation describing the deflection of viscoelastic plates.     

The finite element discretization of fourth order differential operators in space has been subject to intensive research in the literature. The Bogner--Fox--Schmit (BFS) element \cite{BFS65} is a classical $C^1$-conforming  thin plate element obtained by taking the tensor products of cubic Hermite splines.
The discrete solutions are continuously differentiable on tensor product (rectangular) elements, which can be a serious drawback since it limits the applicability of the resulting finite element method. However, for geometries allowing tensor product discretization it is considered to be one of the most efficient elements for plate analysis, cf.\ \cite[p.\ 153]{ZT00}.
It is also a reasonably low order element for plates which is very simple to implement, in contrast with triangular elements which either use higher order polynomials, such as the Argyris element \cite{AFS68}, or macro element techniques, such as the Clough--Tocher element \cite{CT66}. Due to the appreciable advantages of the BFS element and our target to propose a $C^1$-conforming in space and time finite element approach for \eqref{eq:instationary_plate}, the BFS element is applied here.

We note that the finite element approximation of the biharmonic operator continues to be an active field of research, for 
recent contributions, 
see, e.g.~\cite{Burman2020cut,CN21,DE21}. In particular, discretization methods that support polyhedral meshes (the mesh cells can be polyhedra 
or have a simple shape but contain hanging nodes) and hinge on the primal formulation of the biharmonic equation leading to a symmetric positive definite system matrix are currently focused.
These methods can be classified into three groups, depending on the dimension of the smallest geometric object to which discrete unknowns are attached. This criterion influences the stencil of the method. Furthermore, it has an impact on the level of conformity that can be achieved for the discrete solution.
The methods in the first group were developed for the case where $\Omega \subset \R^2$. 
They attach discrete unknowns to the mesh vertices, edges, and cells and can achieve $C^1$-conformity. Salient examples are the $C^1$-conforming virtual element methods (VEM) from \cite{BM13,CM16} and the $C^0$-conforming VEM from \cite{ZCZ16}. Another example 
is the nonconforming 
VEM from \cite{BS05,ZZCM18}. 
The methods in the second group attach discrete unknowns only to the mesh faces and cells for $\Omega\subset \R^d$, with $d\geq 2$. They admit static condensation, and 
provide a nonconforming approximation to the solution. The two salient examples are the weak Galerkin methods from \cite{MWY14,ZZ15,YZZ20} and 
the hybrid high-order method from \cite{BPGK18}.
Finally, the methods in the third group attach discrete unknowns only to the mesh cells and belong to the class of interior penalty discontinuous Galerkin methods. 
These are also nonconforming methods; cf.\ \cite{MS03,SM07,GH09}. Important examples of nonconforming finite elements on simplicial meshes are the Morley element \cite{M68,WX06} and the Hsieh--Clough--Tocher element (cf., e.g., \cite[Chap.~6]{ZZCM18}). 
The spatial discretization of wave problems by discontinuous Galerkin methods has been focused further in the literature, cf., e.g.,\cite{grote2006,anton2016}.

Among the most attractive methods for time discretization of second-order differential equations in time are the so-called \textit{continuous Galerkin or Galerkin--Petrov} 
(cf., e.g., \cite{BL94,FP96,KB14}) and the \textit{discontinuous Galerkin} (cf., e.g., \cite{J93,KB14}) schemes. 
For lowest order elements, these methods can be identified with certain well-known difference schemes, e.g.\ with the classical trapezoidal Newmark scheme (cf., e.g., \cite{W84,W90,H00}), the backward Euler scheme and the Crank--Nicolson scheme.
Strong relations and equivalences between variational time discretizations, collocation methods and Runge--Kutta schemes have been observed. In the literature, the relations are exploited in the formulation and analysis of the schemes. For this we refer to, e.g., \cite{AMN09,AMN11}.
Recently, variational time discretizations of higher order regularity in time \cite{Anselmann2020Galerkin,Bause2020post-processed} have been devised for the second-order hyperbolic wave equations and analyzed carefully. In particular, optimal order error estimates are proved in \cite{Anselmann2020Galerkin,Bause2020post-processed}. In \cite{Bause2020post-processed}, a $C^1$-conforming in time family of space-time finite element approximation that is based on a post-processing of the continuous in time Galerkin approximation is introduced.
Concepts that are developed in \cite{Ern2016discontinuous} for first-order hyperbolic problems are transferred to the wave equation written as a first order system in time. In  \cite{Bause2020post-processed}, a family of Galerkin--collocation approximation schemes with $C^1$- and $C^2$-regular in time discrete solutions are proposed and investigated by an optimal order error analysis and computational experiments.
The conceptual basis of the families of approximations to the wave equation is the establishment of a connection between the Galerkin method for the time discretization and the classical collocation methods, with the perspective of achieving the accuracy of the former with reduced computational costs provided by the latter in terms of less complex algebraic systems.
Further numerical studies for the wave equation can be found in \cite{AB19,AB20_2}. For the application of the Galerkin--collocation to mathematical models of fluid flow and systems of ordinary differential equations we refer to \cite{AB21,BM19,BMW17}.
In the numerical experiments, the Galerkin--collocation schemes have proved their superiority over lower-order and standard difference schemes. In particular, energy conservation is ensured which is an essential feature for discretization schemes to second-order hyperbolic problems since the physics of solutions to the continuous problem are preserved. 

As a logical consequence, for the biharmonic wave problem \eqref{eq:instationary_plate} it appears to be promising to combine the Galerkin--collocation time discretization with the BFS finite element discretization of the spatial variables to a $C^1$-conforming approximation in space and time. This is done here. 
We expect that the uniform variational approximation and higher order regularity will be advantageous for future applications in  multi-physics systems based on \eqref{eq:instationary_plate} as a subproblem, the development of multi-scale approaches (in space and time) for \eqref{eq:instationary_plate} and the application of space-time adaptive methods. For the latter one, we refer to \cite{BGR10,BBK20,KBB15} for parabolic problems.
In this work, we present the combined Galerkin--collocation and BFS finite element approximation of \eqref{eq:instationary_plate}.
Key ingredients of the construction of the Galerkin--collocation approach are the application of a special quadrature formula, proposed in~\cite{JB09}, and the definition of a related interpolation operator for the right-hand side term of the variational equation.
Both of them use derivatives of the given function. The Galerkin--collocation scheme relies in an essential way on the perfectly matching set of the polynomial spaces (trial and test space), quadrature formula, and interpolation operator.
Then, a numerical error analysis is performed, optimal order error estimates are proved. Here, we restrict ourselves to presenting and stressing the differences to the wave equation for the Laplacian considered in \cite{Anselmann2020Galerkin}.
Finally, a numerical study of the proposed discretization scheme is presented in order to illustrate the analyses. 

This paper is organized as follows. In Section~\ref{sec:preliminary}, we introduce our notation and formulate problem \eqref{eq:instationary_plate} as a first-order system in time. In Section~\ref{sec:discretizations}, the Galerkin--collocation method is considered for time discretization. Some beneficial results for the error analysis are summarized in Section~\ref{sec:error_analysis}. In Section~\ref{sec:error_estimates}, we prove error estimates for the introduced Galerkin--collocation method for the plate vibration problem~\eqref{eq:instationary_plate}. Finally, in Section~\ref{sec:num} we present a numerical study confirming the error estimates and perform a comparative study with only continuous in time approximations.

\section{Preliminaries and notation}
\label{sec:preliminary}
\subsection{Evolution form}
Throughout this paper, standard notation is used for Sobolev and Bochner spaces. 
By $B$ we denote a Banach space. 
We use  
$\scalar{ \cdot}{\cdot}$ for the 
$L^2(\Omega)$ inner product inducing the norm 
$$
\Vert \cdot \Vert = \Vert \cdot\Vert_{L^2(\Omega)} 
$$
and $\dual{\cdot}{ \cdot}$ for the duality pairing between a Hilbert space and its dual space.
For the Sobolev norms we adopt the notation 
$$
\Vert \cdot \Vert_m = \Vert \cdot\Vert_{H^m(\Omega)}\quad \text{for}\,\, m\in \mathbb{N},\, m\ge 1
$$
and further define the spaces 
$$
H=L^2(\Omega) \quad \text{and} \quad V = H^2_0(\Omega).
$$ 
Let $V'$ be the dual space of $V$. We introduce the operator $A:V \to V'$ 
which for any given $u\in V$ is uniquely defined by
$$\dual{Au}{v}= \scalar{\Delta u}{\Delta v} \qquad \forall v\in V$$
and also the operator $\mathcal{L}:V\times H\to H\times V'$ given by
$$
\mathcal{L}=\begin{pmatrix}
0 & -I \\
A & 0
\end{pmatrix}.
$$
Here $I$ is the identity operator that acts on $H$. For the error analysis, we define the energy norm on $ H_0^2(\Omega) \times L^2(\Omega) $ by $ \triplenorm{(w_0, w_1)}^2 = \norm{\laplace w_0}^2 + \norm{w_1}^2 $.

In order to formulate problem~\eqref{eq:instationary_plate} in 
an evolutionary form we further introduce the space
$$
X:=L^2(0,T;V)\times L^2(0,T;H)
$$
and 
set
$$
u^0=u \qquad \text{and}\qquad u^1=\partial_t u.
$$

With this notation then problem~\eqref{eq:instationary_plate} can be equivalently stated as: 
Find $U=(u^0,u^1)\in X$ satisfying
\begin{subequations}\label{evolution_form}
\begin{align}
\partial_t U+\mathcal{L}U&=F \qquad \text{in}\quad (0,T)\\
U(0)&=U_0
\end{align}
\end{subequations}
where $f\in L^2(0,T;H)$ is given, $F=(0,f)$ and $U_0=(u_0,u_1)$.

The existence and uniqueness of solutions to \eqref{evolution_form} is a classical result, 
cf.~\cite[p. 273, Thm. 1.1]{Lions1971optimal}, and \cite[p. 275, Thm. 8.2]{Lions1972non-homogeneous}. Further, we have the following regularity result $ H^2(\Omega) \subset \subset C(\overline{\Omega})$ for $\Omega \subset \R^2$, cf. \cite{adams:2003}.

\subsection{Time discretization}
Our aim is to replace the time interval with a discrete time mesh and subsequently to iteratively compute 
the solution of~\eqref{eq:instationary_plate} 
in the time nodes. For this reason, we split $I = \left(0, T\right]$ into $N\in \mathbb{N}$ time subintervals 
\begin{align*}
I_{n} & = \left( t_{n-1}, t_{n} \right], \qquad  n = 1, \ldots, N,
\end{align*}
where $ 0 = t_{0} < t_{1} < \cdots < t_{N} = T $ and introduce the time step parameter $ \tau = \max_{n = 1, \dots, N} \tau_{n} $, where $ \tau_{n} = t_{n} - t_{n-1} $. 
The set $\mathcal{M}_\tau := \{I_1, \ldots , I_N\}$ of time intervals represents the
time mesh. For simplicity, we use $ I_0 = \{t_0\}$.

We denote the space of all $B$-valued polynomials in time of order $k\in \N_0$ over 
a given interval $I_n$ by 
\begin{equation*}
\mathbb{P}_k(I_n;B)=\left\{w_\tau:I_n \to B: w_\tau(t)=\sum_{j=0}^k 
W^j t^j,\; \forall t\in I_n,\; W^j \in B \; \forall j \right\}.
\end{equation*}

Moreover, for an integer $k \in \mathbb{N}$ we introduce the space of globally continuous functions in time, $X_{\tau}^k(B)$, and the space of global $L^2$-functions 
in time, $Y_{\tau}^k(B)$, as follows 
\begin{align*}
X_{\tau}^k(B)&:=\left\{w_\tau\in C(\bar{I};B): w_\tau\restr{I_n}\in \mathbb{P}_k(I_n;B)\;\, \forall I_n \in \mathcal{M}_{\tau} \right\}, \\
Y_{\tau}^k(B)&:=\left\{w_\tau\in L^2(I;B): w_\tau\restr{I_n}\in \mathbb{P}_k(I_n;B)\;\,\forall I_n \in \mathcal{M}_{\tau} \right\}.
\end{align*}

We designate 
\begin{equation*}
\partial_t^s w(t_n^+):=\lim_{t\rightarrow t_n^{+0}} \partial_t^s w(t) \quad \text{and}\quad 
\partial_t^s w(t_n^-):=\lim_{t\rightarrow t_n^{-0}} \partial_t^s w(t)
\end{equation*}
to be the one-sided limits of the $s$-th derivative of a piecewise sufficiently smooth with respect to the time mesh $\mathcal{M}_{\tau}$ function 
$w:I \to B$ where $s\in \mathbb{N}_0$.

\section{Discretizations of space and time}
\label{sec:discretizations}
\subsection{Space discretization}


Let $\mathcal{T}_h$ be a shape-regular mesh of the spatial domain $\Omega$ with $h>0$ denoting the mesh size and 
let 
\begin{equation*}V_{h} = 
\left\lbrace
v_h \in C^{1}(\Omega):v_h\restr{T} \in \mathbb{Q}_{3}(T),v_h\restr{\boundary \Omega} = 0, \, \partial_{\vec{n}}v_h\restr{\boundary \Omega} = 0 \ \forall T \in \Tau_h
\right\rbrace
\end{equation*}
%
be the finite element space built on the mesh using the Bogner--Fox--Schmit element. Here  $\mathbb{Q}_3(T)$ denotes the set of all polynomials with 
maximum degree $3$ in each variable. 
%

%
We denote the $L^2$-orthogonal projection onto $V_h$ by $P_h$, i.e.,
\begin{equation*}
(P_h w,v_h)= (w,v_h) \qquad \forall v_h \in V_h
\end{equation*}
and define the elliptic operator $R_h:V \to V_h$ via 
\begin{equation}
\label{R_h}
(\Delta R_h w,\Delta v_h)=(\Delta w,\Delta v_h)\qquad \forall v_h\in V_h.
\end{equation}
 
For $w\in H^s \cap H_0^2$ we have the estimates 
\begin{equation}\label{ell_proj_err2}
\Vert w - R_h w\Vert_m \le Ch^{4-m} \Vert w \Vert_4, \qquad 0\le m\le 3
\end{equation}
and
\begin{equation*}
\Vert \Delta(w-R_h w)\Vert_m \le C h^{2-m}\Vert w\Vert_4, \qquad 0\le m\le 1
\end{equation*}
which follow directly from the interpolation error estimates given in  \cite{Burman2020cut} along with Cea's lemma and the Aubin--Nitsche trick.

Further, we introduce the $L^2$-projection $\mathcal{P}_h:H\times H \to V_h \times V_h$ and 
the elliptic projection $\mathcal{R}_h:V\times V\to V_h\times V_h$ both of which are onto the product space $V_h\times V_h$
and also the discrete operator $A_h:V\to V_h$ for which it holds
\begin{equation}\label{A_h}
\scalar{A_h w}{v_h}=\scalar{\Delta w}{\Delta v_h}\qquad \forall v_h \in V_h.
\end{equation}  
Therefore, if $w\in V\cap H^4(\Omega)$, we have 
$$
\scalar{A_h w}{v_h}=\scalar{\Delta w}{\Delta v_h}=\dual{A w}{v_h}\qquad \forall v_h \in V_h
$$
or
$$
A_h w = P_h Aw \qquad \text{for }w\in V\cap H^4(\Omega).
$$
Moreover, for the operator $\mathcal{L}_h:V\times H\to V_h\times V_h$ defined as 
\begin{equation}\label{A_h_mathcal}
\mathcal{L}_h=\begin{pmatrix}
0 & -P_h \\ A_h & 0
\end{pmatrix}
\end{equation}
the following relation holds
$$
\scalar{\mathcal{L}_h W}{\Phi_h}=\scalar{-w^1}{\varphi_h^0}+\scalar{\Delta w^0}{\Delta\varphi_h^1}
=  \scalar{-w^1}{\varphi_h^0}+\dual{A w^0}{\varphi_h^1}
=\dual{\mathcal{L} W}{\Phi_h}
$$
for $W=(w^0,w^1)\in (V\cap H^4(\Omega))\times H$ and for all $\Phi_h=
(\varphi_h^0,\varphi_h^1)\in V_h\times V_h$ which demonstrates the consistency of 
$\mathcal{L}_h$ on $(V\cap H^4(\Omega))\times H$, i.e., 
\begin{equation}\label{consistency}
\mathcal{L}_h W=\mathcal{P}_h \mathcal{L}W.
\end{equation}

Finally, an appropriate approximation in $V_h\times V_h$ of the initial value $U_0\in V\times H$ 
is denoted by $U_{0,h}$.

\subsection{Numerical integration}
The following makes 
use of the Hermite-type, Gauss and Gauss-Lobatto quadrature formulas 
which for a sufficiently regular function $g$ on the interval $\bar{I}_n=[t_{n-1},t_n]$ read as 
\begin{subequations}\label{quadrature_formulas}
\begin{align}
Q_n^H(g)&=\left(\frac{\tau_n}{2} \right)^2 \hat{w}_L^H \partial_t g(t_{n-1}^+)
+\frac{\tau_n}{2}\sum_{s=1}^{k-1} \hat{w}_s^H g(t_{n,s}^H)
+\left(\frac{\tau_n}{2}\right)^2 \hat{w}_R^H \partial_t g(t_n^{-}),\label{Hermite}\\
Q_n^G(g)&=\frac{\tau_n}{2}\sum_{s=1}^{k-1}\hat{w}_s^G g(t_{n,s}^G), 
\label{Gauss}\\
\quad Q_n^{GL}(g)&=\frac{\tau_n}{2}\sum_{s=1}^{k}\hat{w}_s^{GL} g(t_{n,s}^{GL}),
\label{Gauss-Lobatto}
\end{align}
\end{subequations}
respectively. Here, $t_{n,s}^H$, $t_{n,s}^G$ and $t_{n,s}^{GL}$
are the corresponding quadrature points on the interval while 
$\{\hat{w}_L^H,\,\hat{w}_R^H, \,\hat{w}_s^H\}$, $\hat{w}_s^{G}$ and 
$\hat{w}_s^{GL}$ denote the corresponding weights.


We also consider the global Hermite interpolation $I_{\tau}^H:C^1(\bar{I};B)\to X_{\tau}^k(B)$ defined as 
\begin{equation}\label{global_Hermite}
I_{\tau}^H w\restr{I_n}:= I_n^H(w\restr{I_n})
\end{equation}
for all $n=1,\ldots,N$ where $I_n^H:C^1(\bar{I}_n;B)\to \mathbb{P}_k(\bar{I}_n;B)$ denotes the local Hermite interpolation operator 
with respect to point values and first derivatives on the interval $\bar{I}_n$.

\subsection{Space-time discretizations}

In this subsection we introduce the discretization of the biharmonic wave problem~\eqref{eq:instationary_plate} by a space-time finite element approach 
utilizing 
a Galerkin--collocation approximation (cf.\ \cite{Anselmann2020Galerkin}) of the time variable along with BFS element for the approximation in space. The time discretization combines Galerkin and collocation techniques.  Moreover, for comparative studies and in order to analyze the impact of the discrete solution's higher regularity in time on the accuracy of the numerical results, the standard continuous in time Galerkin--Petrov approach (cf., e.g., \cite{FP96,Bause2020post-processed}) is presented here 
briefly. Within the latter familiy of schemes, the Crank--Nicolson method is recovered for piecewise linear approximations.  

\subsubsection{The Galerkin--collocation method cGP-\textnormal{C}\texorpdfstring{$^1(k)$}{\textasciicircum 1(k)}}
The variational time discretization for the plate vibration problem~\eqref{eq:instationary_plate} is derived following 
the idea in~\cite{Bause2017error, Anselmann2020Galerkin} and reads as follows:
\begin{problem}\label{problem_cgp}
Let $U_{\tau,h}(t_{n-1}^{-})$ for $n>1$ and $U_{\tau,h}(t_0^{-})=U_{0,h}$ for $n=1$ 
be given. Find $U_{\tau,h}\restr{I_n}\in (\mathbb{P}_k(I_n;V_h))^2$ satisfying 
\begin{subequations}\label{cgp-c1_problem}
\begin{align}
U_{\tau,h}(t_{n-1}^{+}) & = U_{\tau,h}(t_{n-1}^{-}), \label{cgp-c1_problem_a}\\
\partial_t U_{\tau,h}(t_{n-1}^{+}) & =-\mathcal{L}_h U_{\tau,h}(t_{n-1}^{+}) +
\mathcal{P}_h F(t_{n-1}^{+}),\label{cgp-c1_problem_b} \\
\partial_t U_{\tau,h}(t_{n}^{-}) & = -\mathcal{L}_h U_{\tau,h}(t_{n}^{-}) + \mathcal{P}_h 
F(t_n^{-}), \label{cgp-c1_problem_c}\\
Q_n^H(\scalar{\partial_t U_{\tau,h}}{V_{\tau,h}}
+ \scalar{\mathcal{L}_h U_{\tau,h}}{V_{\tau,h}})
& = Q_n^H(\scalar{F}{V_{\tau,h}}) \label{cgp-c1_problem_d}
\end{align}
\end{subequations}
\end{problem}
for all $V_{\tau,h}\in(\mathbb{P}_{k-3}(I_n;V_h))^2$.

The existence and uniqueness of a solution to Problem~\ref{problem_cgp} 
has been discussed in~\cite{Anselmann2020Galerkin}, see also~\cite{Bause2017error}.

From the definition of the scheme it also follows that $U_{\tau,h}\in (C^1(\bar{I};V_h))^2$ and 
\eqref{cgp-c1_problem_b} can be written as 
\begin{equation*}
\partial_t U_{\tau,h}(t_{n-1}^+)=\partial_t U_{\tau,h}(t_{n-1}^-),
\end{equation*}
where $\partial_t U_{\tau,h}(t_{0}^-)=-\mathcal{L}_hU_{0,h}+\mathcal{P}_h F(0) $
and, therefore, Problem~\ref{problem_cgp} can be equivalently written as:

\begin{problem}\label{cgp-c1_problem_2}
    Let $ k \geq 3 $ be fixed and 
    be given the values 
    $ \left(u^0_{\tau,h}\restr{I_{n-1}}(t_{n-1}), u^1_{\tau, h}\restr{I_{n-1}}(t_{n-1})\right) \in V_{h}^{2} $ for $1< n \le N $ 
    and 
    $ \left(u^0_{\tau,h}\restr{I_{0}}(t_{0}), u^1_{\tau, h}\restr{I_{0}}(t_{0})\right) = (u_{0, h}, u_{1, h}) $ 
    for 
    $ n = 1 $. Then the Galerkin--collocation for $ \left(u^0_{\tau,h}\restr{I_{n}}, u^1_{\tau, h}\restr{I_{n}}\right) 
    \in \mathbb{P}_{k}(I_{n}; V_{h})^{2} $ is defined as 
    \begin{subequations}
        \label{eq:collocation}
        \begin{align}
        \partial_{t}^{s} u^0_{\tau,h}\restr{I_{n}}(t_{n-1}) & = \partial_{t}^{s} u^0_{\tau,h}\restr{I_{n-1}}(t_{n-1}), \quad s \in \{0, 1\}, \label{eq:gcu} \\
        \partial_{t}^{s} u^1_{\tau, h}\restr{I_{n}}(t_{n-1}) & = \partial_{t}^{s} u^1_{\tau, h}\restr{I_{n-1}}(t_{n-1}), \quad s \in \{0, 1\}, \label{eq:gcv} \\
        \partial_{t} u^0_{\tau,h}\restr{I_{n}}(t_{n}) - u^1_{\tau, h}\restr{I_{n}}(t_{n}) & = 0, \label{eq:gcvelot} \\
        \partial_{t} u^1_{\tau, h}\restr{I_{n}}(t_{n}) + {A}_{h} u^0_{\tau,h}\restr{I_{n}}(t_{n}) & = f(t_{n}), \label{eq:gcpdet}
        \end{align}
        and
        \begin{align}
        \int_{I_{n}} \int_{\Omega} \partial_{t} u^0_{\tau,h} \, \varphi^0_{\tau, h} - u^1_{\tau, h} \, \varphi^0_{\tau, h} \dx \dt & = 0, \label{eq:gcvelo} \\
        \int_{I_{n}} \int_{\Omega} \partial_{t} u^1_{\tau, h} \, \varphi^1_{\tau, h} + {A}_{h} u^0_{\tau,h} \, \varphi^1_{\tau, h} \dx \dt & = \int_{I_{n}} \int_{\Omega} f \, \varphi^1_{\tau, h} \dx \dt, \label{eq:gcpde}
        \end{align}
    \end{subequations}
    for all $ \left(\varphi^0_{\tau, h}, \varphi^1_{\tau, h}\right) \in (\mathbb{P}_{k-3}(I_{n}; V_{h}))^{2} $.
\end{problem}

The discrete initial values $ (u_{0,h}, u_{1, h}) \in V_{h}^{2} $ are determined from the interpolation of the functions 
$ (u_{0}, u_{1}) $. 
We use the interpolant $ u_{1, h} $ for the value $ \partial_{t} u^0_{\tau, h}(0) $ from \eqref{eq:gcu}. 
In order to obtain an appropriate value for $ \partial_{t} u^1_{\tau, h}(0) $ in \eqref{eq:gcv} 
we consider~\eqref{eq:instationary_plate_a} at time $ t = 0 $ and interpolate the function 
\begin{align*}
    \partial_{t} u^1(\vec{x}, 0) = f(\vec{x}, 0) - \laplace^{2} u(\vec{x}, 0).
\end{align*}
The collocation conditions ensure a reduction in the size of the test space which results in a smaller linear system of equations. 

Next, we summarize a result for scheme~\eqref{cgp-c1_problem} 
used in the error analysis, 
cf.~\cite{Anselmann2020Galerkin}.

\begin{proposition}\label{prop1}
Consider the solution $U_{\tau,h}\in (X_{\tau}^k(V_h))^2$ of Problem~\ref{problem_cgp}. 
It holds that 
\begin{equation*}
B_{n}^{GL}(U_{\tau, h}, V_{\tau, h})
= Q_n^{GL} (\scalar{I_{\tau}^H F}{V_{\tau,h}})
\end{equation*}
for all $V_{\tau,h}\in (\mathbb{P}_{k-2}(I_n;V_h))^2$ and for $n=1,\ldots, N$, where
\begin{equation*}
    B_{n}^{GL}(U_{\tau, h}, V_{\tau, h}) = Q_n^{GL}(\scalar{\partial_t U_{\tau,h}}{V_{\tau,h}}+ \scalar{\mathcal{L}_h U_{\tau,h}}{V_{\tau,h}}).
\end{equation*}
\end{proposition}

\subsubsection{The cGP(\texorpdfstring{$k$}{k})-method}
\label{sec:crank_nicolson}
The second time discretization method for the dynamic plate vibration problem 
considered is the Crank--Nicolson method~\cite{farago}. 
The differential equation is solved iteratively by determining the solution at the time interval points $t_{n}$. 
To derive the Crank--Nicolson method, we use  
$ \left(u^0_{\tau, h}\restr{I_{n}}, u^1_{\tau, h}\restr{I_{n}}\right) \in \mathbb{P}_{k}(\closure{I}_{n}; V_{h})^{2} $ and 
$ \left(\varphi^0_{\tau, h}\restr{I_{n}}, \varphi^1_{\tau, h}\restr{I_{n}}\right) \in \mathbb{P}_{k - 1}(\closure{I}_{n}; V_{h})^{2} $ 
as ansatz~\cite{Anselmann2020Galerkin}. 
Since the solution space differs from the test space, this is referred to as a continuous Galerkin--Petrov method, or cGP($k$) for short. 
The discrete solution functions are globally continuous and
use piecewise polynomials of degree $k$ for the time discretization.
The cGP(1)-method corresponds to the Crank--Nicolson method. 

In contrast to the Galerkin--collocation from the previous subsection, 
the Crank--Nicolson method only provides a solution that is continuous in time, but not a continuously differentiable solution.



\section{Error analysis}
\label{sec:error_analysis}

Let us first note that the results from~\cite{Karakashian2005convergence} for semilinear second order hyperbolic wave equations can 
be carried over to the plate vibration problem when $f=f(u)$, for more details see the appendix. 

Next, we present several definitions required for the error analysis. Let $B\subset H$ and $l \in \mathbb{N}$. 
The local $L^2$-projections $\Pi_n^{l}: L^2(I_n;B)\to \mathbb{P}_{l}(I_n;B)$ are defined by 
\begin{equation*}
\int_{I_n} \scalar{\Pi_n^l w}{q}\dt = \int_{I_n} \scalar{w}{q} \dt \qquad \forall q\in \mathbb{P}_l(I_n;B).
\end{equation*}

We consider the Hermite interpolant in time $I_{\tau}^{k+1}: C^1(\bar{I};B)\to C^1(\bar{I};B)\cap X_{\tau}^{k+1}(B)$ 
studied in~\cite{Bause2020post-processed,Ern2016discontinuous}. For this operator it is fulfilled that 
$$
I_{\tau}^{k+1} u(t_n) = u(t_n),\quad \partial_t I_{\tau}^{k+1} u(t_n) = \partial_t u(t_n), \quad n=0,\ldots,N,
$$
and 
$$
I_{\tau}^{k+1} u(t_{n,\mu}^{GL}) = u(t_{n,\mu}^{GL}), \quad n=1,\ldots, N, \, \mu =2,\ldots, k-1
$$ 
and for a smooth function $u$, the following error estimates hold true on each interval $I_n$
\begin{align*}
\Vert \partial_t u-\partial_t I_{\tau}^{k+1} u\Vert_{C^0 (\bar{I}_n;B)} &\le C \tau_n^{k+1} \Vert u \Vert_{C^{k+2}(\bar{I}_n;B)},\\ 
\Vert \partial_t^2 u-\partial_t^2 I_{\tau}^{k+1} u\Vert_{C^0 (\bar{I}_n;B)} &\le C \tau_n^{k} \Vert u \Vert_{C^{k+2}(\bar{I}_n;B)}.
\end{align*}

Moreover, we define the operator $R_{\tau}^k u\restr{I_n}\in \mathbb{P}_{k}(I_n;B)$ for $n=1,\ldots,N$ 
via the $(k+1)$ conditions
\begin{align*}
R_{\tau}^k u\restr{I_n}(t_{n-1})&=I_{\tau}^{k+1}u(t_{n-1}), \\
\partial_t R_{\tau}^{k+1} u\restr{I_n}(t_{n,s})&=\partial_t I_{\tau}^{k+1} u(t_{n,s}), \quad s =0,\ldots,k
\end{align*}
and we set $R_{\tau}^k u(0):=u(0)$.

Here, we briefly summarize some of the properties of $R_{\tau}^k$ that are important for the analysis. 
Their proofs can be found in~\cite{Bause2020post-processed,Ern2016discontinuous}.

\begin{lemma}\label{lemma_R_tauk}
    Let $k\ge 3$. For $ n = 1, \ldots, N $ and $ u \in C^{k+1}(\closure{I}_n; B) $ the estimate
    \begin{align*}
        \norm{u - R_{\tau}^k u}_{C^0(\closure{I}_n; B)} \leq C \tau_n^{k+1} \norm{u}_{C^{k+1}(\closure{I}_n; B)}
    \end{align*}
    holds.
\end{lemma}

A direct consequence of Lemma~\ref{lemma_R_tauk} is given in the following Corolllary.

\begin{corollary}
    For $ n = 1, \ldots, N $ and $ u \in C^{k+1}(\closure{I}_n; B) $ the estimate 
    \begin{align}\label{estimate_R_tauk}
        \norm{\partial_t u - \partial_t R_{\tau}^k u}_{C^0(\closure{I}_n; B)} \leq C \tau_n^k \norm{u}_{C^{k+1}(\closure{I}_n; B)}
    \end{align}
    is fulfilled.
\end{corollary}

Next we consider the global Hermite interpolation operator defined in~\eqref{global_Hermite}.

\begin{lemma}
For $I_{\tau}^H:C^1(\bar{I};H)\to X_{\tau}^k(H)$ the following estimates 
\begin{align*}
\Vert u - I_{\tau}^H u \Vert_{C^0 (\bar{I}_n;B)} & \le  C \tau_n^{k+1} \Vert u\Vert_{C^{k+1}(\bar{I}_n;B)}\\
\Vert \partial_t u -\partial_t I_{\tau}^H u \Vert_{C^0 (\bar{I}_n;B)}  & \le  C \tau_n^{k} \Vert u\Vert_{C^{k+1}(\bar{I}_n;B)}
\end{align*}
hold true for all $n=1,\ldots, N$ and all $u\in C^{k+1}(\bar{I}_n;B)$.
\end{lemma}

Another important result for our analysis that has been proved in~\cite{Bause2020post-processed} is 
presented as follows.

\begin{lemma}\label{Gauss_quadrature}
Let us consider the Gauss quadrature formula~\eqref{Gauss}. For all polynomials $p\in \mathbb{P}_{k-1}(I_n;B)$ 
and all $n=1,\ldots,N$ it is fulfilled that
\begin{align*}
\Pi_{\tau}^{k-2}p(t_{n,s}^G)=p(t_{n,s}^G),\quad s=1,\ldots,k-1.
\end{align*}
\end{lemma}

Finally, a useful norm bound, see~\cite{Bause2020post-processed,Ern2016discontinuous}, is presented.

\begin{lemma}
For any $u\in \mathbb{P}_k(I_n;H)$ the following inequality
\begin{equation*}
\int_{I_n} \Vert u \Vert^2 \dt \le C \tau_n \Vert u(t_{n-1})\Vert^2 + \tau_n^2 \int_{I_n} \Vert \partial_t u \Vert^2 \dt.
\end{equation*}
is fulfilled. 
\end{lemma}

\section{Error estimates}
\label{sec:error_estimates}

Our ultimate goal in this section is to prove estimates for the error
\begin{equation*}
    E(t) = U(t) - U_{\tau, h}(t)=(e^0(t),e^1(t)), 
\end{equation*}
where the Galerkin--collocation approximation $U_{\tau,h}$ is the solution of Problem~\ref{cgp-c1_problem_2} 
and $U(t)=(u^0(t),u^1(t))$. To achieve this, we start with estimations for $ \partial_t E $ and afterwards we estimate the error $E$. Note that $ E $ is continuously differentiable in time if $ U \in (C^1(\closure{I}; V))^2 $ holds for the exact solution.

\subsection{Error estimates for \texorpdfstring{$ \partial_t U_{\tau, h} $}{partial\_t U\_\{tau, h\}}}

First, we derive estimates for $ \partial_t U_{\tau, h} $, which will be used later.
\begin{theorem}
    Let $ U_{\tau, h} \in (X_{\tau}^k(V_h))^2 $ be the discrete solution of Problem~\ref{problem_cgp}. Then, the time derivative $ \partial_t U_{\tau, h} \in (X_{\tau}^{k-1}(V_h))^2 $ solves for $ n = 1, \ldots, N $ 
    \begin{align*}
        B_n^{GL}(\partial_t U_{\tau, h}, V_{\tau, h}) = Q_n^{GL}(\scalar{\partial_t I_{\tau}^H F}{V_{\tau, h}}) = \int_{I_n} \scalar{\partial_t I_{\tau}^H F}{V_{\tau, h}} \dt
    \end{align*}
    for all $ V_{\tau, h} \in (\mathbb{P}_{k-2}(I_n; V_h))^2 $.
\end{theorem}
\begin{proof}
    The proof differs from proof \cite[Theorem~5.1]{Anselmann2020Galerkin} only in the definition of the operator $ \mathcal{L}_h $.
\end{proof}

\begin{lemma}
    Let $ U_{0, h} = (R_h u_0, R_h u_1) $. Then the identity
    \begin{align*}
        \partial_t U_{\tau, h}(0) =
        \begin{pmatrix}
        R_h & 0 \\ 
        0 & P_h
        \end{pmatrix} 
        \partial_t U(0).
    \end{align*}
    holds.
\end{lemma}
\begin{proof}
    From $ U_{\tau, h}(0) = U_{0, h} = (R_h u_0, R_h u_1) $ together with \eqref{cgp-c1_problem_a} and \eqref{cgp-c1_problem_b} for $ n = 1 $ it follows
    \begin{align*}
        \partial_t U_{\tau, h}(0) & = - \mathcal{L}_h U_{\tau, h}(0) + \mathcal{P}_h F(0) \\
        & =
        - \begin{pmatrix}
        0 & -P_h \\ 
        A_h & 0
        \end{pmatrix} 
        \begin{pmatrix}
        R_h u_0 \\ 
        R_h u_1
        \end{pmatrix} 
        + \begin{pmatrix}
        0 \\ 
        P_h f(0)
        \end{pmatrix} 
        = \begin{pmatrix}
        P_h R_h u_1 \\ 
        - A_h R_h u_0 + P_h f(0)
        \end{pmatrix}.
    \end{align*}
    Using the definition of the operator $ R_h $, \eqref{R_h}, we obtain
    \begin{align*}
        \scalar{A_h R_h u_0}{v_h} = \scalar{\laplace R_h u_0}{\laplace v_h} = \scalar{\laplace u_0}{\laplace v_h} = \dual{A u_0}{v_h} = \scalar{P_h A u_0}{v_h}
    \end{align*}
    for all $ v_h \in V_h $. Thus, $ A_h R_h u_0 = P_h A u_0 $. In addition, we have
    \begin{align*}
    \scalar{P_h R_h u_1}{v_h} = \scalar{R_h u_1}{ v_h}
    \end{align*}
    for all $ v_h \in V_h $, from which we infer $ P_h R_h u_1 = R_h u_1 $. Thus, we have
    \begin{align*}
        \partial_t U_{\tau, h}(0) =
        \begin{pmatrix}
        R_h u_1 \\ 
        - P_h A u_0 + P_h f(0)
        \end{pmatrix} 
    \end{align*}
    Calculating the right-hand side gives
    \begin{align*}
        \begin{pmatrix}
        R_h & 0 \\
        0 & P_h
        \end{pmatrix}
        \partial_t U(0)
        & =
        \begin{pmatrix}
        R_h & 0 \\
        0 & P_h
        \end{pmatrix}
        (- \mathcal{L} U(0) + F(0)) \\
        & =
        \begin{pmatrix}
        R_h & 0 \\
        0 & P_h
        \end{pmatrix}
        \begin{pmatrix}
        u_1 \\
        - A u_0 + f(0)
        \end{pmatrix}
        =
        \begin{pmatrix}
        R_h u_1 \\
        - P_h A u_0 + P_h f(0)
        \end{pmatrix}.
    \end{align*}
    This proves the statement.
\end{proof}

\begin{theorem}
    Let $ \hat{u} $ be the solution of \eqref{eq:instationary_plate} with data $ (\hat{f}, \hat{u}_0, \hat{u}_1) $ instead of $ (f, u_0, u_1) $. Furthermore, let $ l \in \N $ and $ \hat{f}_\tau $ be an approximation of $ \hat{f} $ satisfying
    \begin{align*}
        \norm{\hat{f} - \hat{f}_\tau}_{C(\closure{I}_n; H)} \leq C_{\hat{f}} \tau_n^{l + 1}, \quad n = 1, \ldots, N,
    \end{align*}
    where $ C_{\hat{f}} $ is independent of $ n, N $ and $ \tau_n $. Let $ \hat{U}_{\tau, h} = (\hat{u}_{\tau, h}^0, \hat{u}_{\tau, h}^1) \in (X_{\tau}^l(V_h))^2 $ denote the solution of the local perturbed cGP($l$)-cG($r$) problem
    \begin{align*}
        \int_{I_n} \scalar{\partial_t \hat{U}_{\tau, h}}{V_{\tau, h}} + \scalar{\mathcal{L}_h \hat{U}_{\tau, h}}{V_{\tau, h}} \dt = \int_{I_n} \scalar{\hat{F}_\tau}{V_{\tau, h}} \dt
    \end{align*}
    for all test function $ V_{\tau, h} = (v_{\tau, h}^0, v_{\tau, h}^1) \in (\mathbb{P}_{l-1}(I_n; V_h))^2 $ with $ \hat{F}_{\tau} = (0, \hat{f}_{\tau}) $ and initial value $ \hat{U}_{\tau, h}(t_{n-1}^+) = \hat{U}_{\tau, h}(t_{n-1}^-) $ for $ n > 1 $ and $ \hat{U}_{\tau, h}(t_0) = \hat{U}_{0, h} = (R_h \hat{u}_0, P_h \hat{u}_1) $. Then a sufficiently smooth exact solution $ \hat{u} $ satisfies
    \begin{align}
    \norm{\hat{u}(t) - \hat{u}_{\tau, h}^0(t)} + \norm{\partial_t \hat{u}(t) - \hat{u}_{\tau, h}^1(t)} \leq C (\tau^{l + 1} C_t(\hat{u}) + h^{4} C_x(\hat{u})), \label{eq:perturbed_error1} \\
    \norm{\laplace (\hat{u}(t) - \hat{u}_{\tau, h}^0(t))} \leq C (\tau^{l + 1} C_t(\hat{u}) + h^2 C_x(\hat{u})) \label{eq:perturbed_error2}
    \end{align}
    for all $ t \in \closure{I} $ where $ C_t(\hat{u}) $ and $ C_x(\hat{u}) $ depend on various temporal and spatial derivatives of $ \hat{u} $.
\end{theorem}

\begin{proof}
The proof follows the 
lines of the proof of~\cite[Theorem~5.5]{Bause2020post-processed}.
\end{proof}

The main result in this section is 
\begin{theorem}
Let the exact solution $U=(u^0,u^1):=(u,\partial_t u)$ be sufficiently smooth and $U_{0,h}:=(R_h u_0, R_h u_1)$, then the following 
estimates hold true for all $t\in \bar{I}$
\begin{align}
\Vert \partial_t U(t)-\partial_t U_{\tau,h}(t) \Vert & \le C( \tau^k C_t(\partial_t u) + h^{4} C_x(\partial_t u)) 
\le C(\tau^k+h^{4}), \label{error_partial_t_1}
\\
\Vert \Delta (\partial_t u^0(t)-\partial_t u_{\tau,h}^0(t)) \Vert &\le C( \tau^k C_t(\partial_t u) + h^{2} C_x(\partial_t u)) 
\le C(\tau^k+h^{2}), \label{error_partial_t_2}
\end{align}
where the quantities $C_t(\partial_t u)$ and $C_x(\partial_t u)$ depend on various temporal and spatial derivatives of $\partial_t u$.
\end{theorem}
\begin{proof}
    This proof differs from proof \cite[Theorem~5.5]{Anselmann2020Galerkin} in the definition of the operators $ A $, 
    $ \mathcal{L}_h $ and $ R_h $ as well as the finite element space $ V_h $. Analogous to 
    this proof in~\cite{Anselmann2020Galerkin}, 
    estimates \eqref{error_partial_t_1} and \eqref{error_partial_t_2} follow from estimates \eqref{eq:perturbed_error1} and \eqref{eq:perturbed_error2}.
\end{proof}

\subsection{Error estimates for \texorpdfstring{$ U_{\tau, h} $}{U\_\{tau, h\}}}

In this subsection we wish to estimate $ U_{\tau, h} $. Therefore, we use the following splitting
\begin{gather*}
E(t) = \Theta(t) + E_{\tau, h}(t) \quad \text{with} \\
\Theta(t) = U(t) - \mathcal{R}_h R_\tau^k U(t) \quad \text{and} \quad E_{\tau, h}(t) = \mathcal{R}_h R_\tau^k U(t) - U_{\tau, h}(t),
\end{gather*}
where $ E_{\tau, h}(t) = (e_{\tau, h}^0(t), e_{\tau, h}^1(t)) $.

\begin{lemma}[Estimation of the interpolation error]
    Let $ m \in \{0,1\} $. Then, the error estimates
    \begin{alignat}{2}
        \norm{\Theta(t)}_m & \leq C (h^{4 - m} + \tau_n^{k+1}), & \quad & t \in \closure{I}_n, \label{est_int_error1}\\
        \norm{\partial_t \Theta(t)}_m & \leq C (h^{4 - m} + \tau_n^k), & \quad & t \in \closure{I}_n,
         \label{est_int_error2}
    \end{alignat}
    are valid for $ n = 1, \ldots, N $ where $ \norm{\cdot}_0 = \norm{\cdot} $.
\end{lemma}
\begin{proof}
First, we note that for the elliptic projection $R_h$ defined in~\eqref{R_h} it holds that 
$ \norm{R_h u} \leq C \norm{\laplace R_h u} \leq C \norm{\laplace u} $. 
    Then, the proof follows the same lines as the proof 
    of~\cite[Lemma~5.7]{Bause2020post-processed}.
    
 Using Lemma~\ref{lemma_R_tauk} and the approximation properties of the elliptic projection 
 $R_h$ we obtain
 \begin{align*}
        \norm{\Theta(t)}_m & = \norm{U(t) - \mathcal{R}_h R_{\tau}^{k} U(t)}_m \\
        & \leq \norm{U(t) - \mathcal{R}_h U(t)}_m + \norm{\mathcal{R}_h (U(t) - R_{\tau}^{k} U(t))}_m \\
        & \leq C h^{4 - m} \norm{U}_{C^0(\closure{I}; H^{4}(\Omega))} + \tau_n^{k+1} \norm{U}_{C^{k+1}(\closure{I}; H^1(\Omega))}
    \end{align*}
    which proves~\eqref{est_int_error1}.
  Applying estimate~ \eqref{estimate_R_tauk} and the fact that $\partial_t$ and $R_h$ commute yields 
    \begin{align*}
        \norm{\partial_t \Theta(t)}_m & \leq \norm{\partial_t U(t) - \mathcal{R}_h \partial_t U(t)}_m + \norm{\mathcal{R}_h (\partial_t U(t) - \partial_t R_{\tau}^{k+1} U(t))}_m \\
        & \leq C h^{4 - m} \norm{\partial_t U}_{C^0(\closure{I}; H^{4}(\Omega))} + C \tau_n^{k+1} \norm{U}_{C^{k+2}(\closure{I}; H^1(\Omega))}
    \end{align*}    
    which shows~\eqref{est_int_error2}.
\end{proof}

\begin{lemma}[Consistency error]
    Let $ U \in C^1(\closure{I}; V) \times C^1(\closure{I}; H) $. Then, we have
    \begin{align*}
    B_n^{GL}(E, V_{\tau, h}) = Q_n^{GL}(\scalar{I_\tau^{GL} F - I_\tau^H F}{V_{\tau, h}}) = Q_n^{GL}(\scalar{F - I_\tau^H F}{V_\tau, h})
    \end{align*}
    for all $ V_{\tau, h} \in (Y_{\tau, h}^{k-2})^2 $ and $ n = 1, \ldots, N $.
\end{lemma}
\begin{proof}
    Because of Proposition~\ref{prop1} and the consistency property \eqref{consistency}, the proof differs from the proof \cite[Lemma~5.7]{Anselmann2020Galerkin} only in the definition of the operators $ \mathcal{L} $ and $ \mathcal{L}_h $.
\end{proof}

For the proof of the following result we refer to \cite[Lemma~5.9]{Bause2020post-processed}.
\begin{lemma}
    \label{dtGLinterpolation}
    Let $ p \in \mathbb{P}_k(I_n) $ be an arbitrary polynomial of degree less than or equal to $ k $. Then, the identity
    \begin{align*}
    \partial_t p(t_{n, \mu}^G) = \partial_t I_{\tau}^{GL} p(t_{n, \mu}^G)
    \end{align*}
    is satisfied for all Gauss points $ t_{n, \mu}^G \in I_n, \mu = 1, \ldots, k-1 $.
\end{lemma}

\begin{lemma}[Stability]
    The following identity
    \begin{align*}
    \begin{aligned}
    & \quad B_n^{GL}((e_{\tau, h}^0, e_{\tau, h}^1), (\Pi_n^{k-2} A_h I_{\tau}^{GL} e_{\tau, h}^0, \Pi_n^{k-2} I_{\tau}^{GL} e_{\tau, h}^1)) \\
    & = \frac{1}{2}(\norm{\laplace e_{\tau, h}^0(t_n)}^2 - \norm{\laplace e_{\tau, h}^0(t_{n-1})}^2 + \norm{e_{\tau, h}^1(t_n)}^2 - \norm{e_{\tau, h}^1(t_{n-1})}^2 )
    \end{aligned}
    \end{align*}
    holds true for all $ n = 1, \ldots, N $.
\end{lemma}
\begin{proof}
    First, we note that $ \scalar{(e_{\tau, h}^0, e_{\tau, h}^1)}{(\Pi_n^{k-2} A_h I_{\tau}^{GL} e_{\tau, h}^{0}, \Pi_n^{k-2} I_{\tau}^{GL} e_{\tau, h}^{1})} \in \mathbb{P}_{2k - 2}(I_n; \R) $.
    Using the Gauss-Lobatto interpolation operator, we obtain $ I_{\tau}^{GL} e_{\tau, h}^1 \in \mathbb{P}_{k-1}(I_n; V_h) $ and $ A_h I_{\tau}^{GL} e_{\tau, h}^0 \in \mathbb{P}_{k-1}(I_n; V_h) $.
    Then, we have
    \begin{align}
    \begin{aligned}
    \label{stabilityT}
    & \quad B_n^{GL}((e_{\tau, h}^0, e_{\tau, h}^1), (\Pi_n^{k-2} A_h I_{\tau}^{GL} e_{\tau, h}^0, \Pi_n^{k-2} I_{\tau}^{GL} e_{\tau, h}^1)) \\
    & = Q_n^{GL}(\scalar{(\partial_t e_{\tau, h}^0, \partial_t e_{\tau, h}^1)}{(\Pi_n^{k-2} A_h I_{\tau}^{GL} e_{\tau, h}^0, \Pi_n^{k-2} I_{\tau}^{GL} e_{\tau, h}^1)}) \\
    & \quad + Q_n^{GL}(\scalar{( -P_h I_{\tau}^{GL} e_{\tau, h}^1, A_h I_{\tau}^{GL} e_{\tau, h}^0)}{(\Pi_n^{k-2} A_h I_{\tau}^{GL} e_{\tau, h}^0, \Pi_n^{k-2} I_{\tau}^{GL} e_{\tau, h}^1)}) \\
    & = \int_{I_{n}} \scalar{(\partial_t e_{\tau, h}^0, \partial_t e_{\tau, h}^1)}{(\Pi_n^{k-2} A_h I_{\tau}^{GL} e_{\tau, h}^0, \Pi_n^{k-2} I_{\tau}^{GL} e_{\tau, h}^1)} \dt \\
    & \quad + \int_{I_{n}} \scalar{(-P_h I_{\tau}^{GL} e_{\tau, h}^1, A_h I_{\tau}^{GL} e_{\tau, h}^0)}{(\Pi_n^{k-2} A_h I_{\tau}^{GL} e_{\tau, h}^0, \Pi_n^{k-2} I_{\tau}^{GL} e_{\tau, h}^1)} \dt \\
    & = T_1 + T_2.
    \end{aligned}
    \end{align}
    Using Lemma~\ref{Gauss_quadrature} 
    together with the exactness of the $(k-1)$-point Gauss quadrature formula for polynomials on $ \mathbb{P}_{2k-3}(I_n; \R) $ and subsequent application of Lemma~\ref{dtGLinterpolation}, we obtain
    \begin{align*}
    \begin{aligned}
    T_1 & = \int_{I_{n}} \scalar{(\Pi_n^{k-2} \partial_t e_{\tau, h}^0, \Pi_n^{k-2} \partial_t e_{\tau, h}^1)}{(\Pi_n^{k-2} A_h I_{\tau}^{GL} e_{\tau, h}^0, \Pi_n^{k-2} I_{\tau}^{GL} e_{\tau, h}^1)} \dt \\
    & = Q_{n}^G(\scalar{(\partial_t e_{\tau, h}^0, \partial_t e_{\tau, h}^1)}{(A_h I_{\tau}^{GL} e_{\tau, h}^0, I_{\tau}^{GL} e_{\tau, h}^1)}) \\
    & = Q_{n}^G(\scalar{(\partial_t I_{\tau}^{GL} e_{\tau, h}^0, \partial_t I_{\tau}^{GL} e_{\tau, h}^1)}{(A_h I_{\tau}^{GL} e_{\tau, h}^0, I_{\tau}^{GL} e_{\tau, h}^1)}) \\
    & = \frac{\tau_n}{2} \sum_{\mu = 1}^{k-1} \hat{w}_{\mu}^{G} \scalar{\partial_t I_{\tau}^{GL} e_{\tau, h}^0(t_{n, \mu}^{G})}{A_h I_{\tau}^{GL} e_{\tau, h}^0(t_{n, \mu}^{G})} \\
    & \quad + \frac{\tau_n}{2} \sum_{\mu = 1}^{k-1} \hat{w}_{\mu}^{G} \scalar{\partial_t I_{\tau}^{GL} e_{\tau, h}^1(t_{n, \mu}^{G})}{I_{\tau}^{GL} e_{\tau, h}^1(t_{n, \mu}^{G})} \\
    & = \frac{\tau_n}{2} \sum_{\mu = 1}^{k-1} \hat{w}_{\mu}^{G} \frac{1}{2} \mathrm{d}_t \norm{A_h^{1/2} I_{\tau}^{GL} e_{\tau, h}^0(t_{n, \mu}^{G})}^2
    + \frac{\tau_n}{2} \sum_{\mu = 1}^{k-1} \hat{w}_{\mu}^{G} \frac{1}{2} \mathrm{d}_t \norm{I_{\tau}^{GL} e_{\tau, h}^1(t_{n, \mu}^{G})}^2.
    \end{aligned}
    \end{align*}
    Using the exactness of the $(k-1)$-point Gauss quadrature formula on $ \mathbb{P}_{2k-3}(I_n; \R) $ again, we obtain
    \begin{align}
        \begin{aligned}
        \label{statilityT1}
        T_1 & = \int_{I_{n}} \left( \frac{1}{2} \mathrm{d}_t \norm{A_h^{1/2} I_{\tau}^{GL} e_{\tau, h}^0(t)}^2 + \frac{1}{2} \mathrm{d}_t \norm{I_{\tau}^{GL} e_{\tau, h}^1(t)}^2 \right) \dt \\
        & = \frac{1}{2} \left( \norm{A_h^{1/2} e_{\tau, h}^0(t_n)}^2 - \norm{A_h^{1/2} e_{\tau, h}^0(t_{n-1})}^2 + \norm{ e_{\tau, h}^1(t_n)}^2 - \norm{e_{\tau, h}^1(t_{n-1})}^2 \right).
        \end{aligned}
    \end{align}
    By using Lemma~\ref{Gauss_quadrature}, we have that
    \begin{align}
        \begin{aligned}
        \label{stabilityT2}
        T_2 & = \int_{I_{n}} \scalar{(-P_h I_{\tau}^{GL} e_{\tau, h}^1, A_h I_{\tau}^{GL} e_{\tau, h}^0)}{(\Pi_n^{k-2} A_h I_{\tau}^{GL} e_{\tau, h}^0, \Pi_n^{k-2} I_{\tau}^{GL} e_{\tau, h}^1)} \dt \\
        & = \int_{I_{n}} \scalar{(-\Pi_n^{k-2} I_{\tau}^{GL} e_{\tau, h}^1, \Pi_n^{k-2} A_h I_{\tau}^{GL} e_{\tau, h}^0)}{(\Pi_n^{k-2} A_h I_{\tau}^{GL} e_{\tau, h}^0, \Pi_n^{k-2} I_{\tau}^{GL} e_{\tau, h}^1)} \dt \\
        & = \int_{I_{n}} \scalar{- \Pi_n^{k-2}I_{\tau}^{GL} e_{\tau, h}^1}{\Pi_n^{k-2} A_h I_{\tau}^{GL} e_{\tau, h}^0} + \scalar{\Pi_n^{k-2} A_h I_{\tau}^{GL} e_{\tau, h}^0}{\Pi_n^{k-2} I_{\tau}^{GL} e_{\tau, h}^1} \dt \\
        & = 0.
        \end{aligned}
    \end{align}
    Inserting the equations \eqref{statilityT1} and \eqref{stabilityT2} into \eqref{stabilityT} along with the identity $ \norm{A_h^{1/2} v_h} = \norm{\laplace v_h} $ for $ v_h \in V_h $ finally yields the assertion.
\end{proof}

\begin{lemma}[Boundedness]
    Let $ V_{\tau, h} = (\Pi_n^{k-2} A_h I_\tau^{GL} e_{\tau, h}^0, \Pi_n^{k-2} I_\tau^{GL} e_{\tau, h}^1) $. Then the estimate
    \begin{align*}
    \abs{B_n^{GL}(\Theta, V_{\tau, h})} \leq C \tau_n^{1/2} (\tau_n^{k+1} + h^{4}) \{ \tau_n \norm{E_{\tau, h}}^2 + \tau_n^2 Q_n^G(\norm{\partial_t E_{\tau, h}}^2) \}^{1/2}
    \end{align*}
    is valid for all $ n = 1, \ldots, N$.
\end{lemma}

\begin{proof}
The proof of this lemma differs from that of~\cite[Lemma~5.11]{Bause2020post-processed} only in the definition of the elliptic projection operator 
and the applied projection error estimates.
\end{proof}

\begin{lemma}[Estimates on right-hand side term]
    Let $ V_{\tau, h} = (\Pi_n^{k-2} A_h I_\tau^{GL} e_{\tau, h}^0, \Pi_n^{k-2} I_\tau^{GL} e_{\tau, h}^1) $. Then, the estimation
    \begin{align*}
    Q_n^{GL}(\scalar{(0, f - I_\tau^H f)}{V_{\tau, h}}) \leq C \tau_n^{1/2} \tau_n^{k+1} \{\tau_n \norm{E_{\tau, h}(t_{n-1})}^2 + \tau_n^2 Q_n^G(\norm{\partial_{t} E_{\tau, h}}^2) \}^{1/2}
    \end{align*}
    is valid for $ n = 1, \ldots, N $.
\end{lemma}
\begin{proof}
    The proof differs from proof \cite[Lemma~5.11]{Anselmann2020Galerkin} only in the definition of the operator $ A_h $.
\end{proof}

\begin{lemma}[Estimates on $ E_{\tau, h} $]
    Let $ U_{0, h} = (R_h, u_0, R_h, u_1) $. Then
    \begin{align*}
    \norm{e_{\tau, h}^0(t_n)}_1^2 + \norm{e_{\tau, h}^1}^2 \leq C \left( \tau^{k+1} + h^{4} \right)^2
    \end{align*}
    holds true for all $ n = 1, \ldots, N $. Additionally,
    \begin{align*}
    \norm{\laplace e_{\tau, h}^0(t)} & \leq C (\tau^{k+1} + h^3) \\
    \norm{e_{\tau, h}^0(t)} + \norm{e_{\tau, h}^1(t)} & \leq C (\tau^{k+1} + h^{4})
    \end{align*}
    is statisfied for all $ t \in \closure{I} $.
\end{lemma}

\begin{proof}
The proof follows the lines of~\cite[Lemma~5.12]{Anselmann2020Galerkin}.
\end{proof}

\begin{theorem}[Error estimate for $ U_{\tau, h} $]\label{thm:estimat_U}
    Let $ U = (u, \partial_{t} u) $ be the solution of the problem \eqref{eq:instationary_plate} and $ U_{\tau, h} $ be the discrete solution of problem \eqref{eq:collocation} along with the initial condition $ U_{0, h} = (R_{h} u_{0}, R_{h} u_{1}) $. Then the following estimates for the error $ E(t) = (e^{0}(t), e^{1}(t)) = U(t) - U_{\tau, h}(t) $ apply for all $ t \in \closure{I} $:
    \begin{align*}
    \norm{e^0(t)} + \norm{e^1(t)} & \leq C (\tau^{k + 1} + h^{4}), \\
    \norm{\laplace e^{0}(t)} & \leq C (\tau^{k + 1} + h^{3}).
    \end{align*}
    Additionally, the estimates 
    \begin{align}
        \norm{e^0}_{L^2(I; H)} + \norm{e^1}_{L^2(I; H)} & \leq C (\tau^{k+1} + h^{4}), \label{eq:error_estimate_a} \\
        \norm{\laplace e^0}_{L^2(I; H)} & \leq C (\tau^{k + 1} + h^{3}) \nonumber
    \end{align}
    are satisfied.
\end{theorem}

\begin{proof}
The proof of this theorem differs from the proof of~\cite[Theorem~5.13]{Anselmann2020Galerkin} only in the definition of the operators.
\end{proof}

\section{Numerical experiments}\label{sec:num}

The aim of the numerical experiments included in this section is:
\begin{itemize}
\item[(i)] to compute the numerical convergence orders for the time discretizations discussed in the previous sections;
\item[(ii)] to compare the solutions obtained when using the cGP-C$^1(3)$-, the cGP($2$)- and the
Crank--Nicolson (cGP(1)-) method for time discretization.
\end{itemize}

The numerical experiments in the first part of this section are used to confirm the error estimates proven in Section~\ref{sec:error_estimates} and to show the 
faster convergence of the cGP-C$^1(3)$-method compared with the other algorithms investigated. In the second part of this section, 
we perform a comparative study in order to demonstrate the superiority of the cGP-C$^1(3)$-method over the considered only continuous in time approximation schemes.

We assume that the spatial domain $ \Omega \subset \R^2 $ is either the unit square $ (0, 1)^2 $ or the square $ (-1, 1)^2 $ which during the discretization process has been partitioned
as bisections of $N^2$ squares with mesh size $h = \sqrt{2}/N$ or $ h=2\sqrt{2}/N $, respectively. Furthermore, we use the Bogner--Fox--Schmit element throughout for spatial discretization.

All the numerical tests included in this section have been conducted in NGSolve, see \url{https://ngsolve.org}. 

\subsection{Numerical convergence study}
We will utilize the first example to provide numerical evidence for the error estimates proven in Section~\ref{sec:error_estimates}, on the one hand, and to present the better convergence properties of the cGP-C$^1(3)$-method, on the other. For this purpose, we compare the experimental order of convergence of different time discretization schemes.

Thereby, we expect to obtain an order of convergence of 2 for the Crank-Nicolson (cGP(1)-) method, an order of 3 for the cGP(2)-method, and an order of 4 for the cGP-C$^1(3)$-method.
In addition, as in \cite{Bause2020post-processed} for the wave equation, we expect to observe superconvergence for the cGP(2)-method in the discrete time points.
While we predict a convergence order of 4 only in the discrete points for the cGP(2)-method, we anticipate a convergence order of 4 in all time points for the cGP-C$^1(3)$-method.

In our first example, we solve system~\eqref{eq:instationary_plate} for the spatial domain $ \Omega = (0, 1)^{2} $ and the temporal domain $ I = (0, 1) $. 
For a right-hand side 
\begin{align*}
    f(\vec{x}, t) & = - 4 \pi^{2} \sin(2 \pi t) \cdot \sin^{2}(\pi x_{1}) \cdot \sin^{2}(\pi x_{2}) \\
    & \quad + \pi^{4} \sin(2 \pi t) \cdot \left[ 16 \sin^{2}(\pi x_{1}) - 8 \right] \cdot \sin^{2}(\pi x_{2}) \\
    & \quad + 2 \pi^{4} \sin(2 \pi t) \left[ 2 - 4 \sin^{2}(\pi x_{1}) \right] \cdot \left[ 2 - 4 \sin^{2}(\pi x_{2}) \right] \\
    & \quad + \pi^{4} \sin(2 \pi t) \cdot \sin^{2}(\pi x_{1}) \cdot \left[ 16 \sin^{2}(\pi x_{2}) - 8 \right]
\end{align*}
and initial values
\begin{align*}
u_{0}(\vec{x}) & = 0, \\
u_{1}(\vec{x}) & = 2 \pi \cdot \sin^{2}(\pi x_{1}) \cdot \sin^{2}(\pi x_{2}),
\end{align*}
the exact solution $u$ is given by 
\begin{align}
    \label{eq:fct2}
    u(\vec{x}, t) = \sin(2 \pi t) \cdot \sin^{2}(\pi x_{1}) \cdot \sin^{2}(\pi x_{2}).
\end{align} 
The corresponding numerical solution of Problem~\eqref{eq:instationary_plate} computed with 
the cGP-C$^1(3)$-method for 
$ \tau = 0.01, \, h = \frac{\sqrt{2}}{10} $ is shown in Figure~\ref{fig:solution2}.

\begin{figure}[htb]
    \centering
    \begin{tabular}{ccccc}
        \includegraphics[clip, trim=500 100 500 50, width=2.8cm]{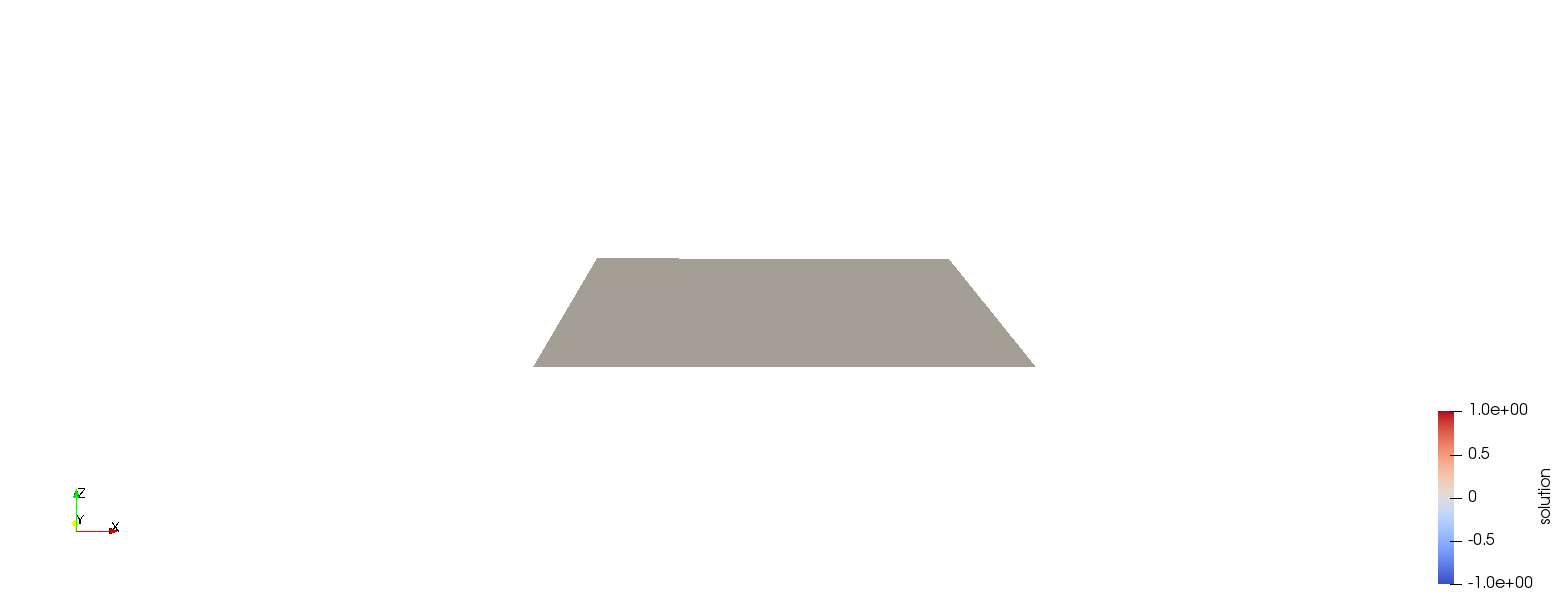} & \includegraphics[clip, trim=500 100 500 50, width=2.8cm]{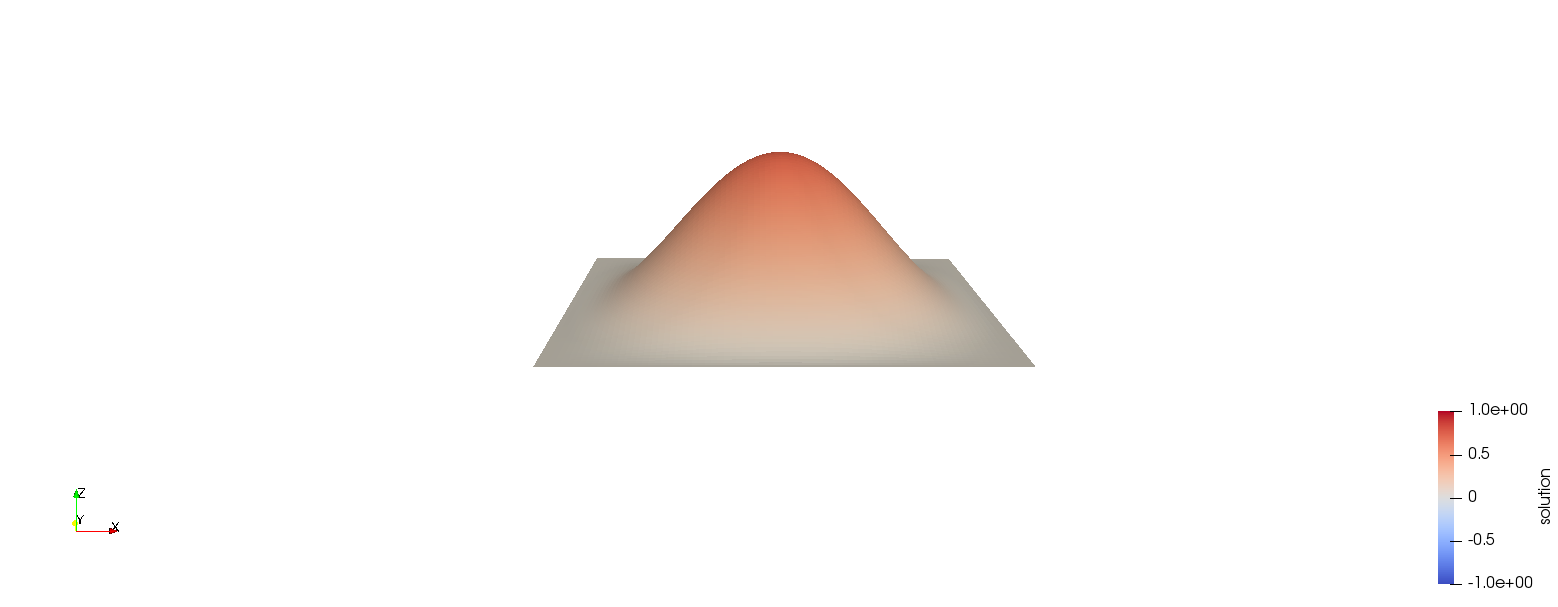} & \includegraphics[clip, trim=500 100 500 50, width=2.8cm]{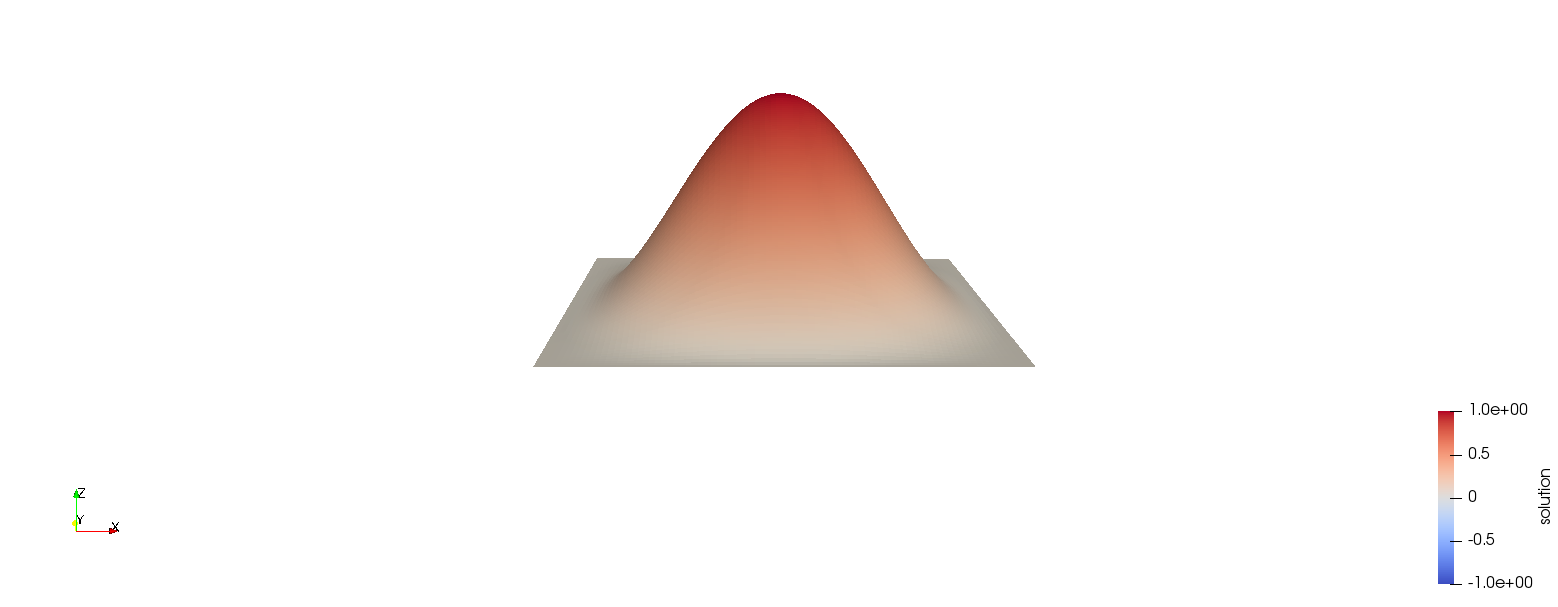} & \includegraphics[clip, trim=500 100 500 50, width=2.8cm]{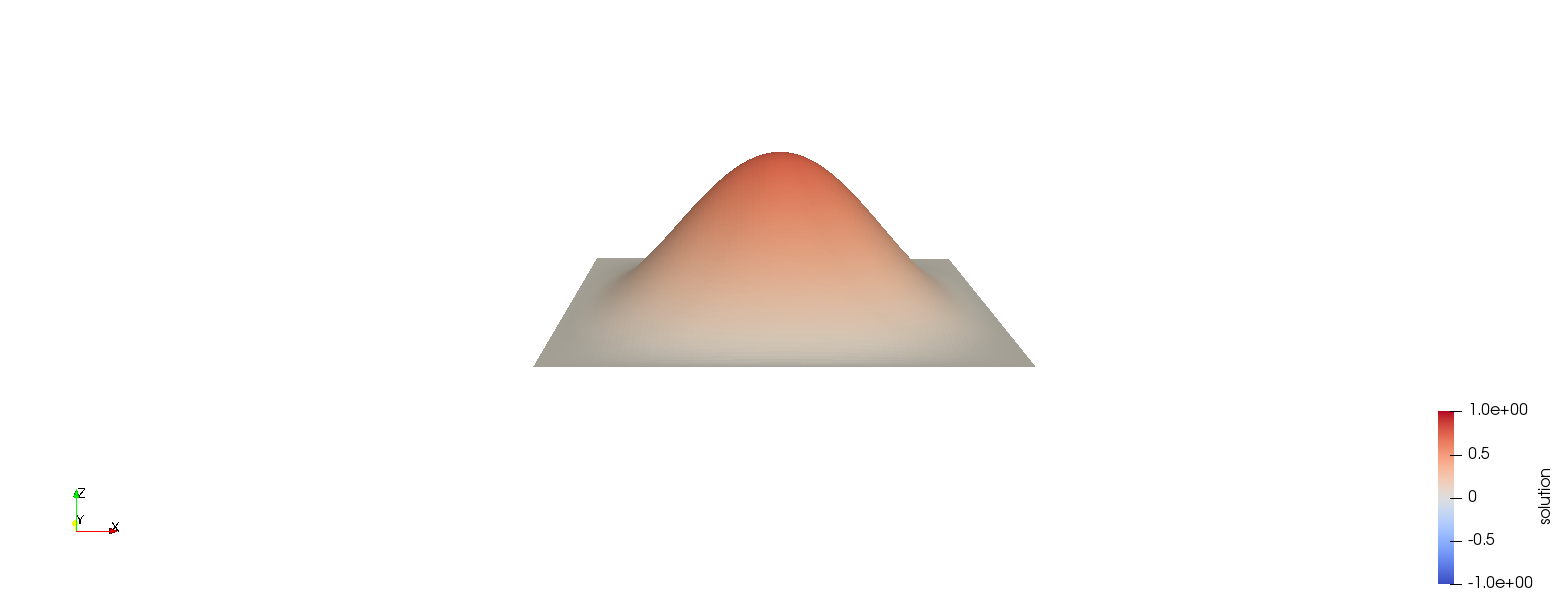} &  \\ 
        \includegraphics[clip, trim=500 50 500 100, width=2.8cm]{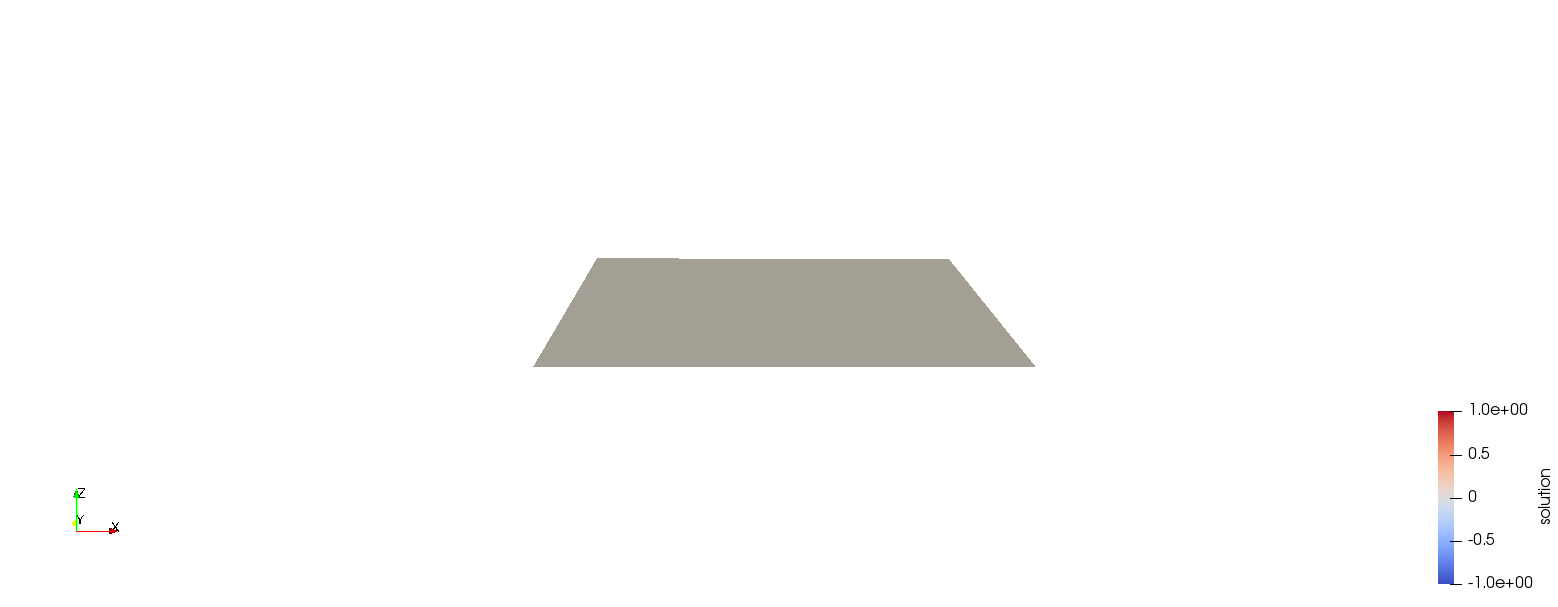} & \includegraphics[clip, trim=500 50 500 100, width=2.8cm]{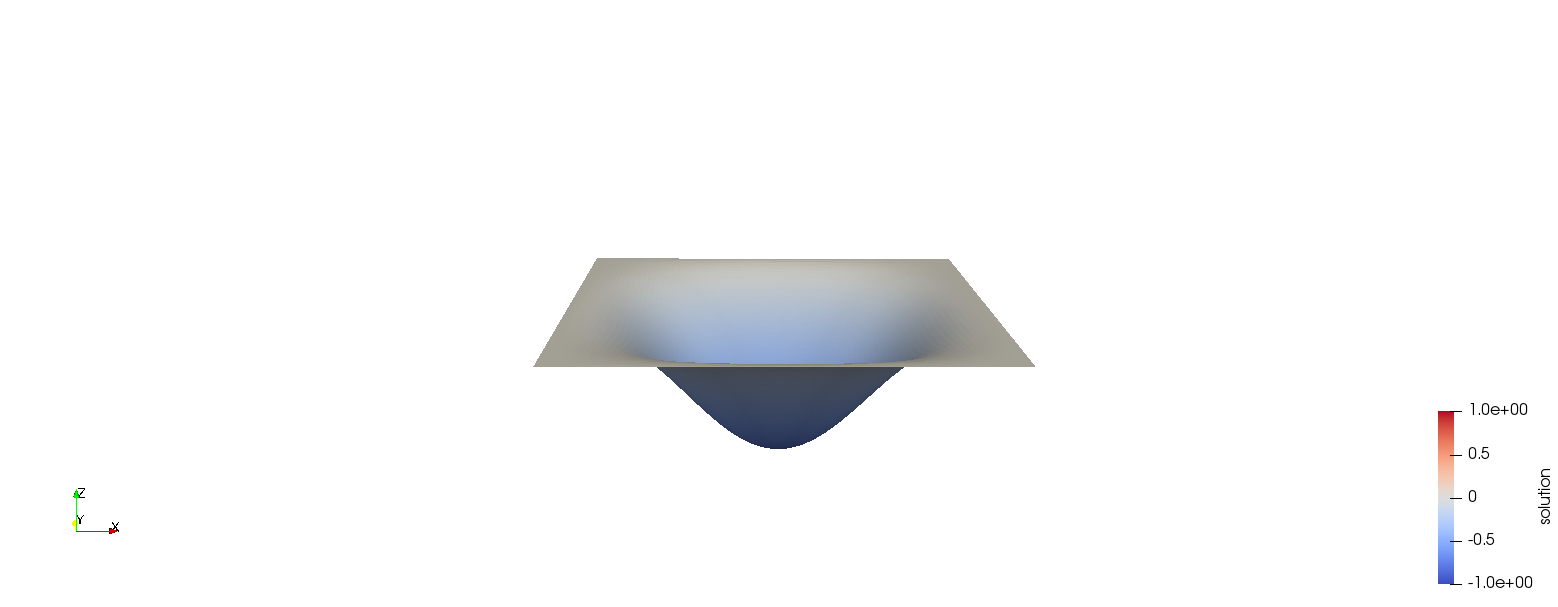} & \includegraphics[clip, trim=500 50 500 100, width=2.8cm]{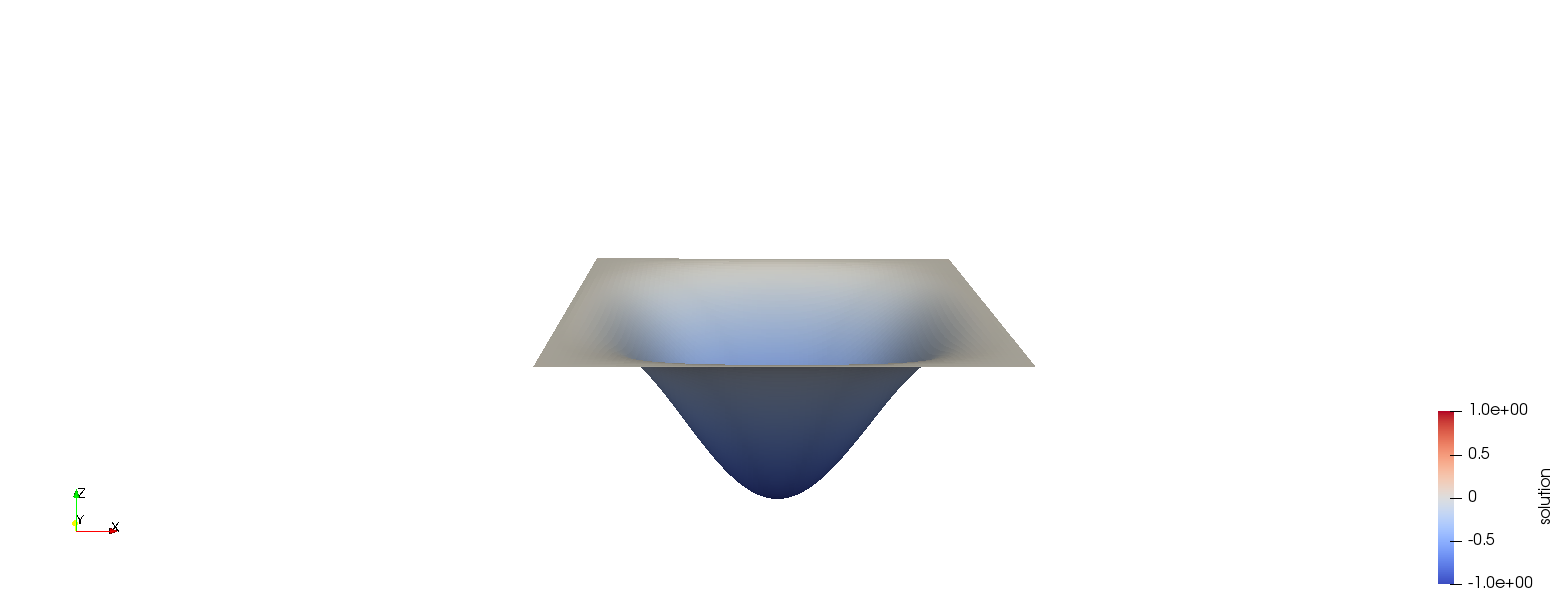} & \includegraphics[clip, trim=500 50 500 100, width=2.8cm]{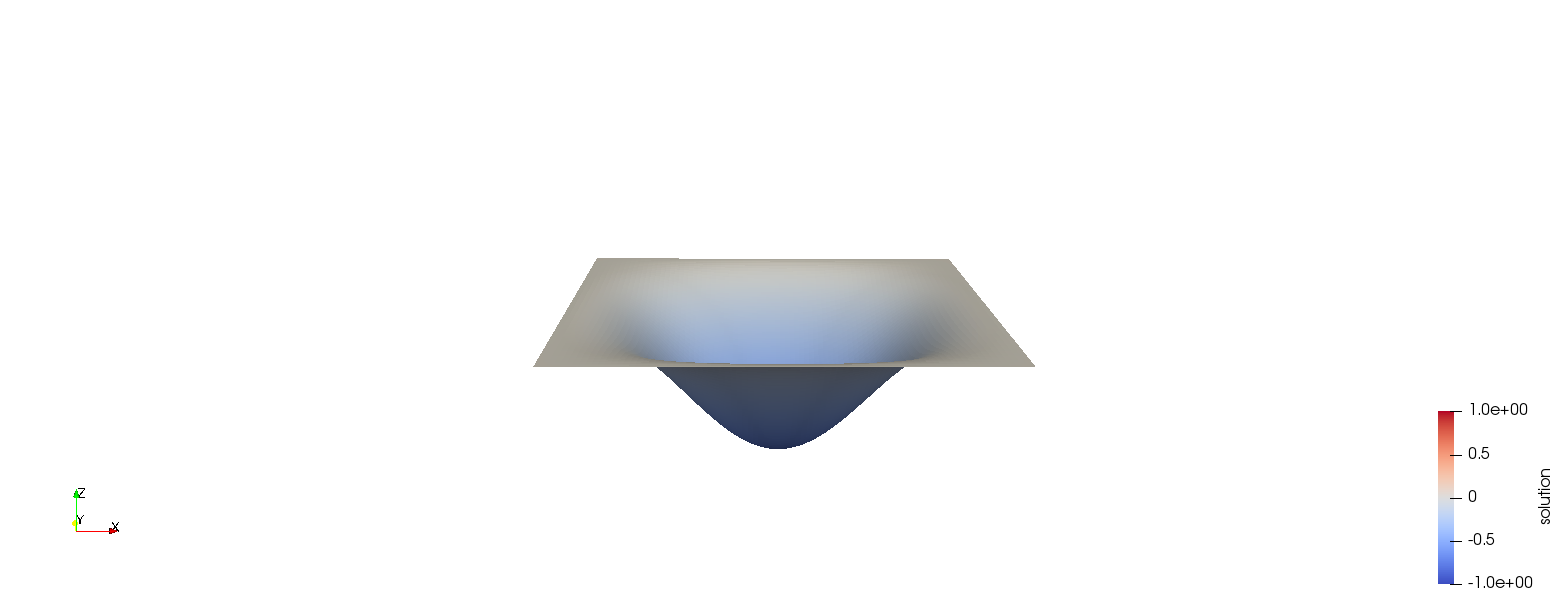} & \multirow[t]{2}{*}{\includegraphics[clip, trim=1080 130 360 80, width=1.6cm, height=4.74cm]{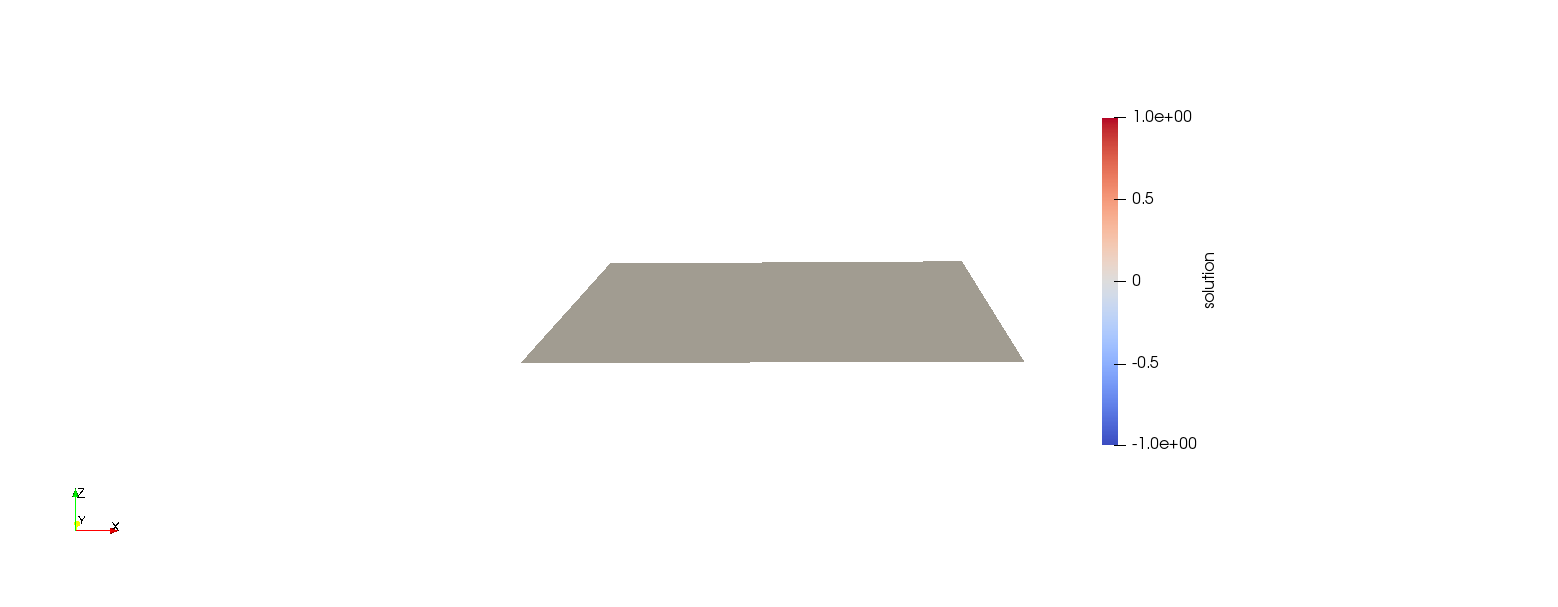}} \\
    \end{tabular}
    \caption{Numerical solution of Problem~\eqref{eq:instationary_plate} for the function~\eqref{eq:fct2} when the cGP-C$^1(3)$-method 
        with parameters $ \tau = 0.01, \, h = \frac{\sqrt{2}}{10}$ is applied.}
    \label{fig:solution2}
\end{figure}
 
For a comparison of the solutions of different time discretization methods, we consider the norms 
\begin{align*}
\norm{u - u_{\tau, h}}_{L^{\infty}(L^{2})} & =
\max_{t \in [0, T]} \left( \int_{\Omega} \abs{u(\vec{x}, t) - u_{\tau, h}(\vec{x}, t)}^{2} \dx \right)^{\frac{1}{2}}
\end{align*}
and
\begin{align*}
\norm{u - u_{\tau, h}}_{L^{2}(L^{2})} & = 
\left( \int_{I} \int_{\Omega} \abs{u(\vec{x}, t) - u_{\tau, h}(\vec{x}, t)}^{2} \dx \dt \right)^{\frac{1}{2}}.
\end{align*}
In order to approximate the $ \norm{\cdot}_{L^\infty(L^2)} $ norm, we first evaluate the maximum only in the discrete temporal points in which we have computed the discrete solution 
$ u_{\tau, h} $ and denote this value by $ \norm{\cdot}_{L_{\tau}^\infty(L^2)} $.
Secondly, we determine the maximum by additionally evaluating the discrete solution in the time points
\begin{align*}
    I_{\tau} & = \{t_{n, j} : t_{n, j} = t_{n-1} + j \cdot \frac{1}{100} \cdot \tau_n, \, j = 1, \ldots, 99, \, n = 1, \ldots, N \},
\end{align*}
which we denote by $ \norm{\cdot}_{L^\infty(L^2)} $.

We compute both norms on a sequence of spatial and temporal meshes in order to determine the numerical convergence orders.
We start with $ \tau_{0} = 0.1 , \, h_{0} = \frac{\sqrt{2}}{5} $ and halve both after each pass.
With $ e_{\tau, h} $ we denote the error with time step $ \tau $ and mesh size $ h $, 
then the following formula 
\begin{align*}
EOC = \log_{2}\left(\frac{e_{\tau, h}}{e_{\frac{\tau}{2}, \frac{h}{2}}}\right)
\end{align*}
is applied to compute the experimental order of convergence (EOC).

The discretization errors and the corresponding convergence orders for the Crank--Nicolson, cGP(2)- and cGP-C$^1$(3)-method in case of function~\eqref{eq:fct2} 
are presented in Table~\ref{tab:fct2_errors}.

\begin{table}[htb]
    \centering
    \small
    \begin{tabular}{cccccccc}
        \hline $ \tau $ & $ h $ & $ \norm{u - u_{\tau, h}}_{L_\tau^{\infty}(L^{2})} $ & order & $ \norm{u - u_{\tau, h}}_{L^{\infty}(L^{2})} $ & order & $ \norm{u - u_{\tau, h}}_{L^{2}(L^{2})} $ & order \\ 
        \hline \multicolumn{8}{c}{cGP(1)} \\
        \hline $ \tau_{0} / 2^{0} $ & $ h_{0} / 2^{0} $ & 3.296e-03 & -- & 1.794e-02 & -- & 9.081e-03 & -- \\ 
        $ \tau_{0} / 2^{1} $ & $ h_{0} / 2^{1} $ & 7.835e-04 & 2.07 & 4.794e-03 & 1.90 & 2.281e-03 & 1.99 \\ 
        $ \tau_{0} / 2^{2} $ & $ h_{0} / 2^{2} $ & 1.877e-04 & 2.06 & 1.233e-03 & 1.96 & 5.746e-04 & 1.99 \\ 
        $ \tau_{0} / 2^{3} $ & $ h_{0} / 2^{3} $ & 4.746e-05 & 1.98 & 3.080e-04 & 2.00 & 1.435e-04 & 2.00 \\ 
        $ \tau_{0} / 2^{4} $ & $ h_{0} / 2^{4} $ & 1.194e-05 & 1.99 & 7.701e-05 & 2.00 & 3.589e-05 & 2.00 \\ 
        \hline \multicolumn{8}{c}{cGP(2)} \\
        \hline $ \tau_{0} / 2^{0} $ & $ h_{0} / 2^{0} $ & 1.058e-03 & -- & 1.059e-03 & -- & 8.369e-04 & -- \\ 
        $ \tau_{0} / 2^{1} $ & $ h_{0} / 2^{1} $ & 7.258e-05 & 3.87 & 1.063e-04 & 3.32 & 6.731e-05 & 3.64 \\ 
        $ \tau_{0} / 2^{2} $ & $ h_{0} / 2^{2} $ & 4.553e-06 & 3.99 & 1.247e-05 & 3.09 & 6.635e-06 & 3.34 \\ 
        $ \tau_{0} / 2^{3} $ & $ h_{0} / 2^{3} $ & 2.867e-07 & 3.99 & 1.510e-06 & 3.05 & 7.625e-07 & 3.12 \\ 
        $ \tau_{0} / 2^{4} $ & $ h_{0} / 2^{4} $ & 1.798e-08 & 4.00 & 1.857e-07 & 3.02 & 9.310e-08 & 3.03 \\ 
        \hline \multicolumn{8}{c}{cGP-C$^1$(3)} \\
        \hline $ \tau_{0} / 2^{0} $ & $ h_{0} / 2^{0} $ & 1.165e-03 & -- & 1.231e-03 & -- & 8.533e-04 & -- \\ 
        $ \tau_{0} / 2^{1} $ & $ h_{0} / 2^{1} $ & 8.141e-05 & 3.84 & 8.151e-05 & 3.92 & 5.673e-05 & 3.91 \\ 
        $ \tau_{0} / 2^{2} $ & $ h_{0} / 2^{2} $ & 5.314e-06 & 3.94 & 5.314e-06 & 3.94 & 3.363e-06 & 4.08 \\ 
        $ \tau_{0} / 2^{3} $ & $ h_{0} / 2^{3} $ & 2.998e-07 & 4.15 & 2.999e-07 & 4.15 & 2.005e-07 & 4.07 \\ 
        $ \tau_{0} / 2^{4} $ & $ h_{0} / 2^{4} $ & 1.823e-08 & 4.04 & 1.824e-08 & 4.04 & 1.222e-08 & 4.04 \\
        \hline 
    \end{tabular} 
    \caption{Numerical errors and convergence orders for the Crank--Nicolson (cGP(1)), cGP(2)- and cGP-C$^1$(3)-method for the function~\eqref{eq:fct2}.}
    \label{tab:fct2_errors}
\end{table}

As seen from this table, the discretization errors for the cGP(2)- and the cGP-C$^1(3)$-method are smaller than the corresponding values for 
the Crank--Nicolson method.
Moreover, the cGP(2)-method with global continuous and piecewise quadratic functions gives higher convergence orders than the 
Crank--Nicolson method with piecewise linear functions. The cGP-C$^1(3)$-method has the highest convergence orders in both norms of all the methods studied.

The numerical convergence orders in the $ \norm{\cdot}_{L^{2}(L^{2})} $ norm tend to $2$ for the Crank--Nicolson method and 
for the cGP(2)-method to $ 3 $. 
In the $ \norm{\cdot}_{L_\tau^{\infty}(L^{2})} $ norm it can be seen even that the convergence order is $ 4 $ 
for the cGP(2)-method whereas it is $ 2 $ for the Crank--Nicolson algorithm. 
This observed superconvergence in the $ \norm{\cdot}_{L_\tau^{\infty}(L^{2})} $ norm is due to the evaluation in the Gauss-Lobatto points. Comparing the evaluation of the maximum only in the time points where we have computed the discrete solution as in $ \norm{\cdot}_{L_\tau^{\infty}(L^{2})} $ with the evaluation on a finer temporal mesh as 
performed with $ \norm{\cdot}_{L^{\infty}(L^{2})} $, verifies this superconvergence effect.
For a more detailed study of the superconvergence in the case of the wave equation, we refer to, e.g., \cite{Bause2020post-processed}.
All convergence orders for the cGP-C$^1$(3)-method tend to 4. This confirms the convergence order proven in Theorem~\ref{thm:estimat_U}.

Consequently, we obtain better convergence results in all discrete time points for the cGP-C$^1(3)$-method, although we have the same numerical costs in terms of degrees of freedom as for the cGP(2)-method.

\subsection{Vibration in heterogeneous media}
In this section, we want to show that the cGP-C$^1(3)$-method is also applicable for more complex problems. Additionally, we stress the superiority of this method by analyzing and comparing the number of non-zero entries in the system matrix which is involved in the linear system of equations in every time step.

As a second example in this work, we consider the modified problem
\begin{subequations}
    \label{eq:structure_health}
    \begin{align}
        \partial_{tt} u(\vec{x}, t) + \laplace [c(\vec{x}) \laplace u(\vec{x}, t)] & =  f(\vec{x}, t) & & \qquad \text{in} \ \Omega \times \left(0, T\right], \\
        u(\vec{x}, 0) & = u_{0}(\vec{x}) & & \qquad \text{in} \ \Omega, \\
        \partial_{t} u(\vec{x}, 0) & = u_{1}(\vec{x}) & & \qquad \text{in} \ \Omega, \\
        u(\vec{x}, t) & = 0 & & \qquad \text{on} \ \boundary \Omega \times \left(0, T\right], \\
        \partial_{\vec{n}} u(\vec{x}, t) & = 0 & & \qquad \text{on} \ \boundary \Omega \times \left(0, T\right],
    \end{align}
\end{subequations}
which arises from structural health monitoring and can be treated analogously. The coefficient $ c > 0 $ encodes the stiffness of the involved materials. Here we use the setting
\begin{align*}
    \Omega = (-1, 1)^2, \qquad T = \frac{3}{100}, \qquad
    c(\vec{x}) = 
    \begin{cases}
        1, & \text{if } x_2 < 0.2, \\
        9, & \text{if } x_2 \geq 0.2,
    \end{cases}
    \qquad f = 0,
\end{align*}
with the initial values
\begin{align*}
    u_0 &= 0.2 \cdot \exp(- \abs{10 \cdot \vec{x}}^2) \cdot (1 - x_1^2)^2 \cdot (1 - x_2^2)^2, \\
    u_1 &= 0.
\end{align*}
The first few time steps of the numerical solution of Problem~\eqref{eq:structure_health} using the cGP-C$^1(3)$-method are depicted in Figure~\ref{fig:structure_health}.
\begin{figure}[htbp]
    \centering
    \begin{tabular}{cccc}
        \includegraphics[clip, trim=515 40 515 40, width=3.5cm]{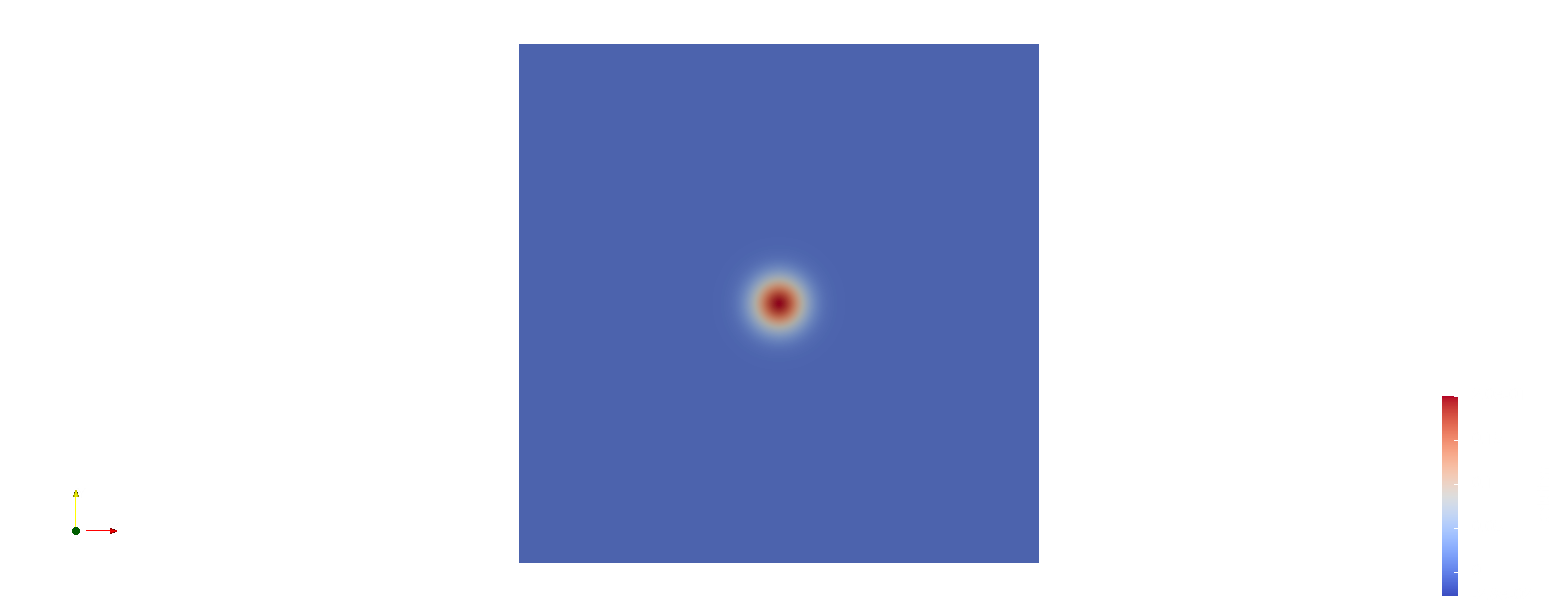} & \includegraphics[clip, trim=515 40 515 40, width=3.5cm]{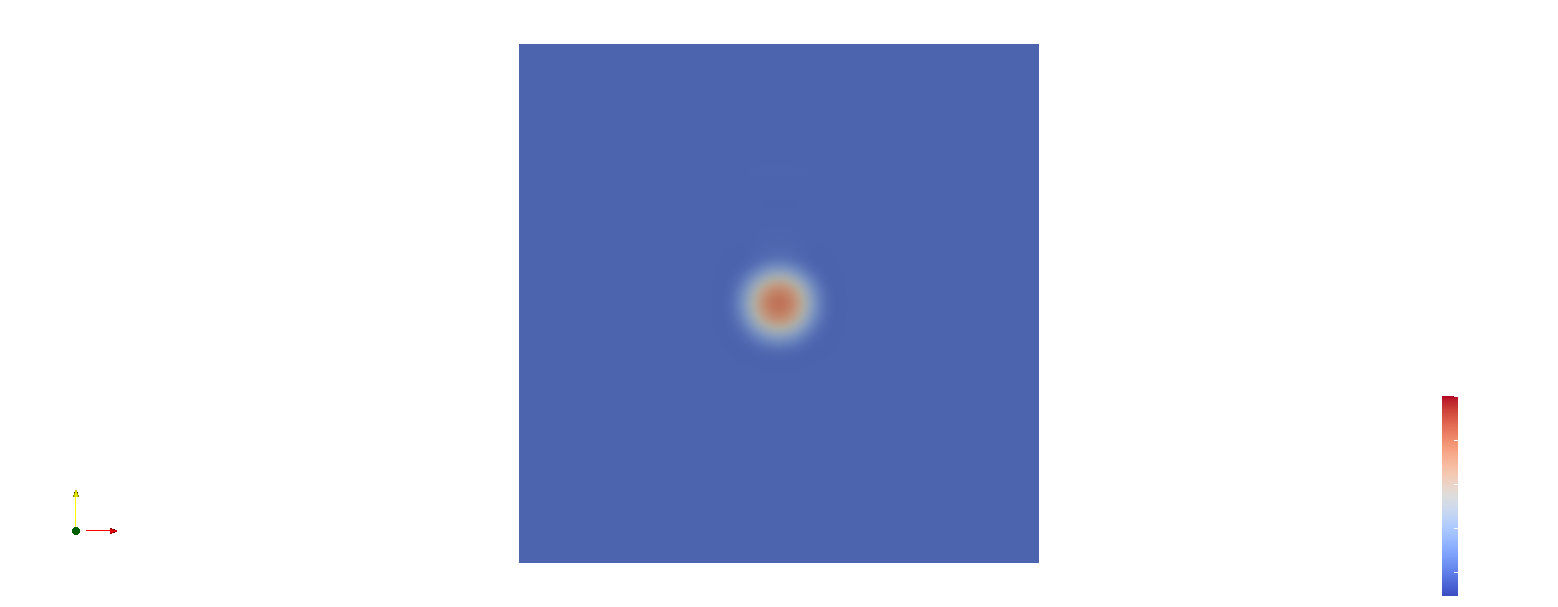} & \includegraphics[clip, trim=515 40 515 40, width=3.5cm]{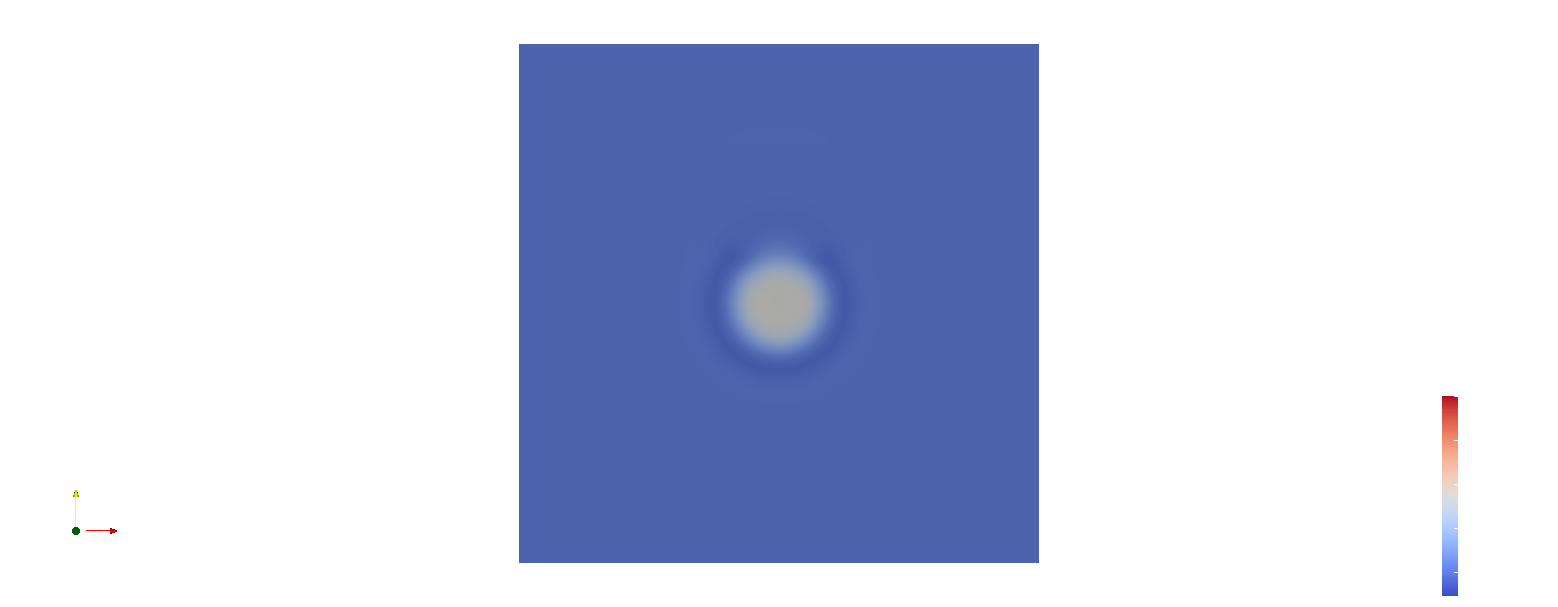} & \includegraphics[clip, trim=515 40 515 40, width=3.5cm]{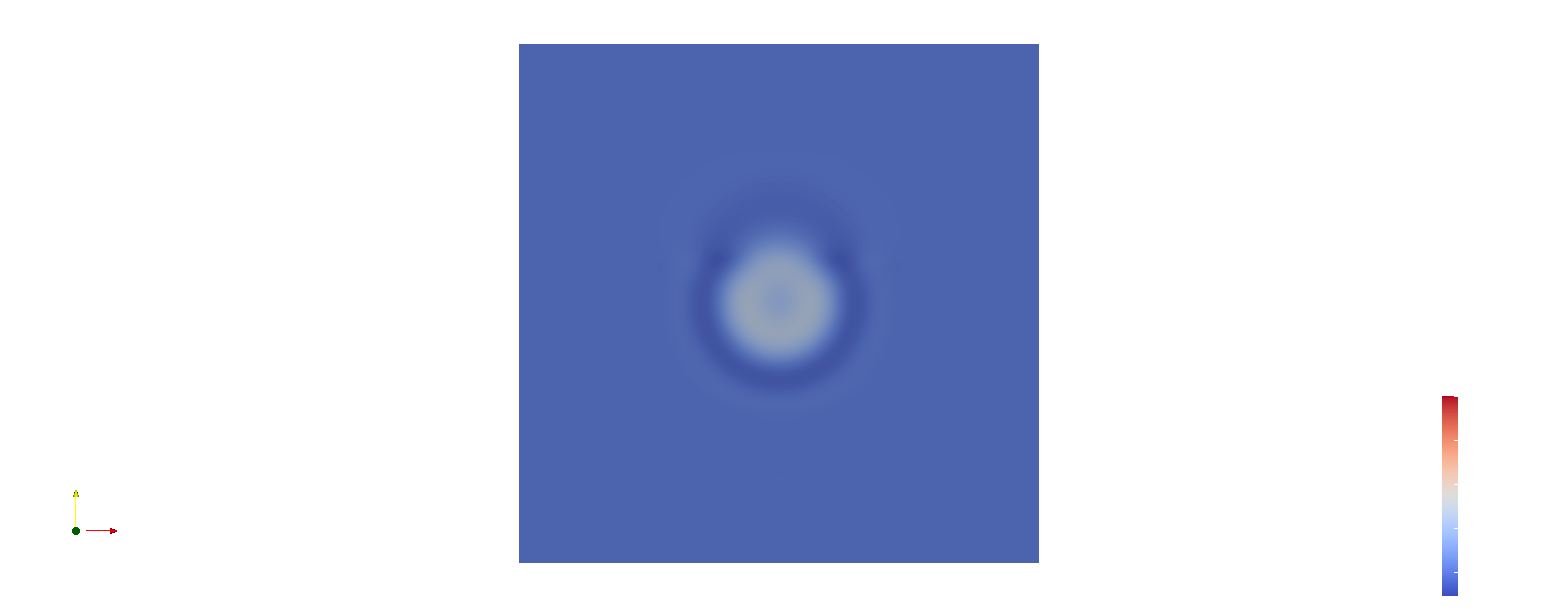} \\ 
        \includegraphics[clip, trim=515 40 515 40, width=3.5cm]{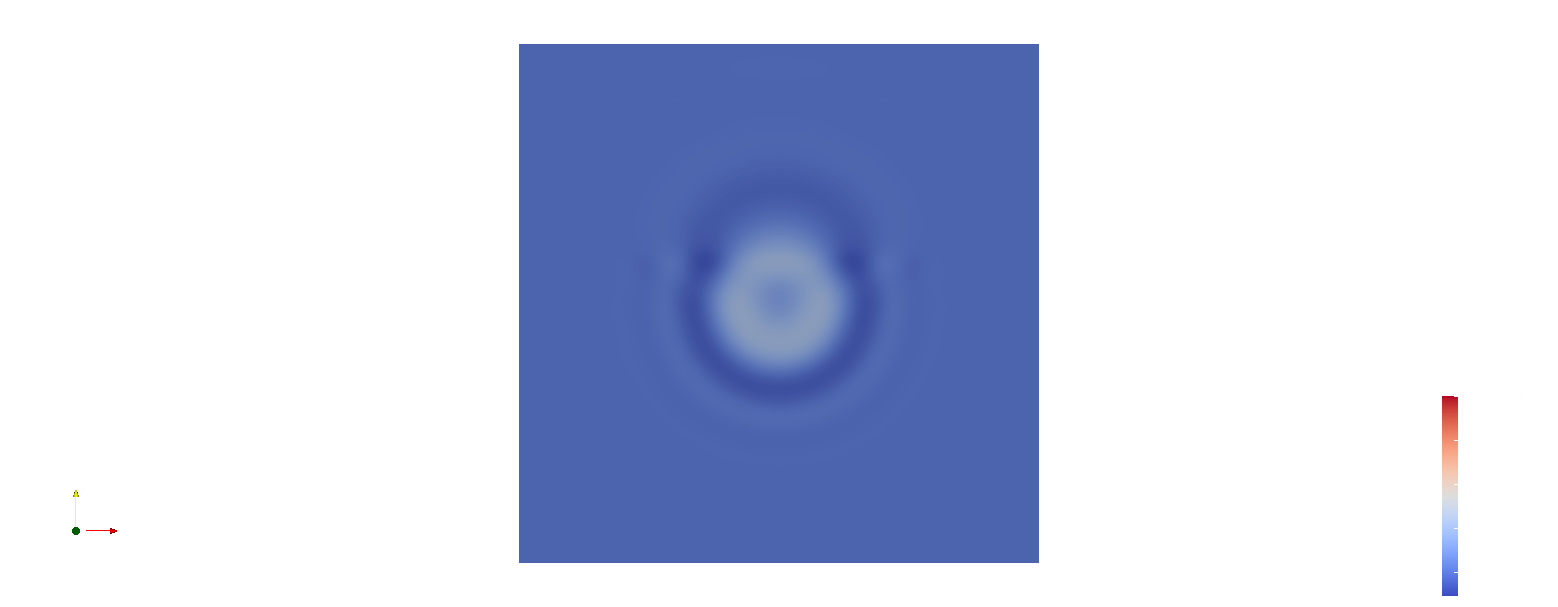} & \includegraphics[clip, trim=515 40 515 40, width=3.5cm]{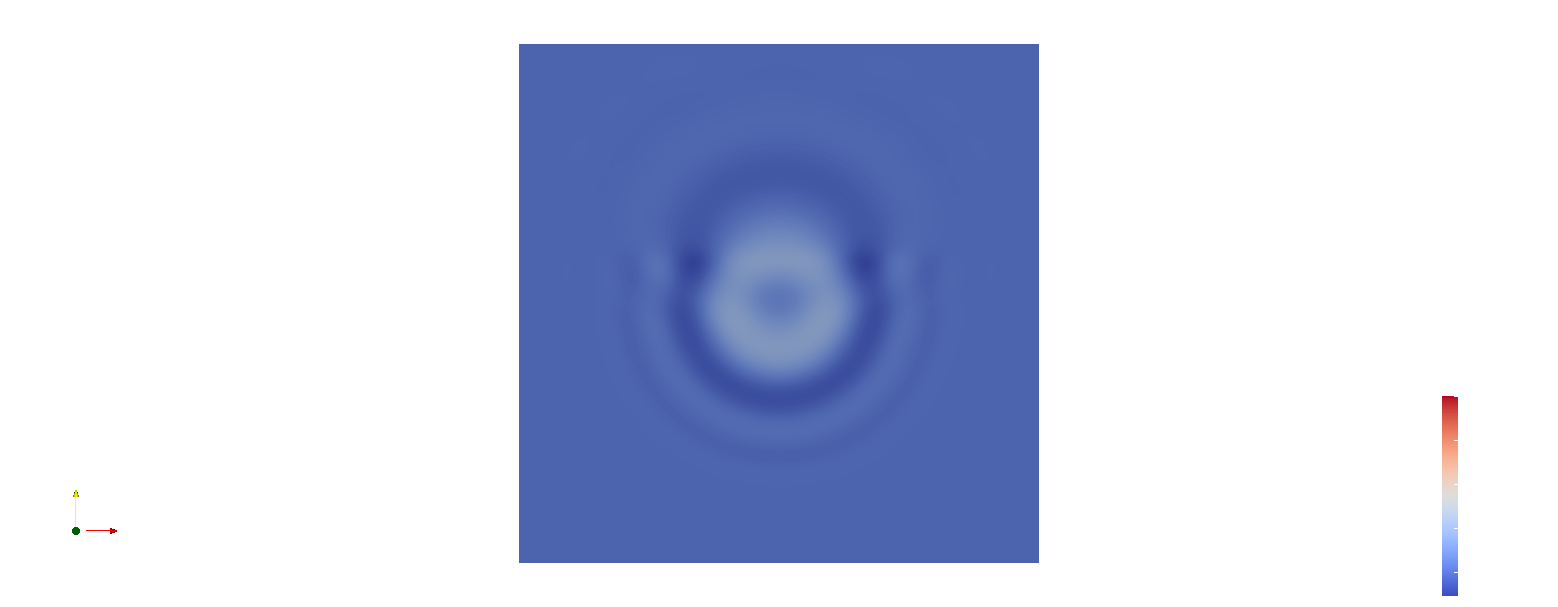} & \includegraphics[clip, trim=515 40 515 40, width=3.5cm]{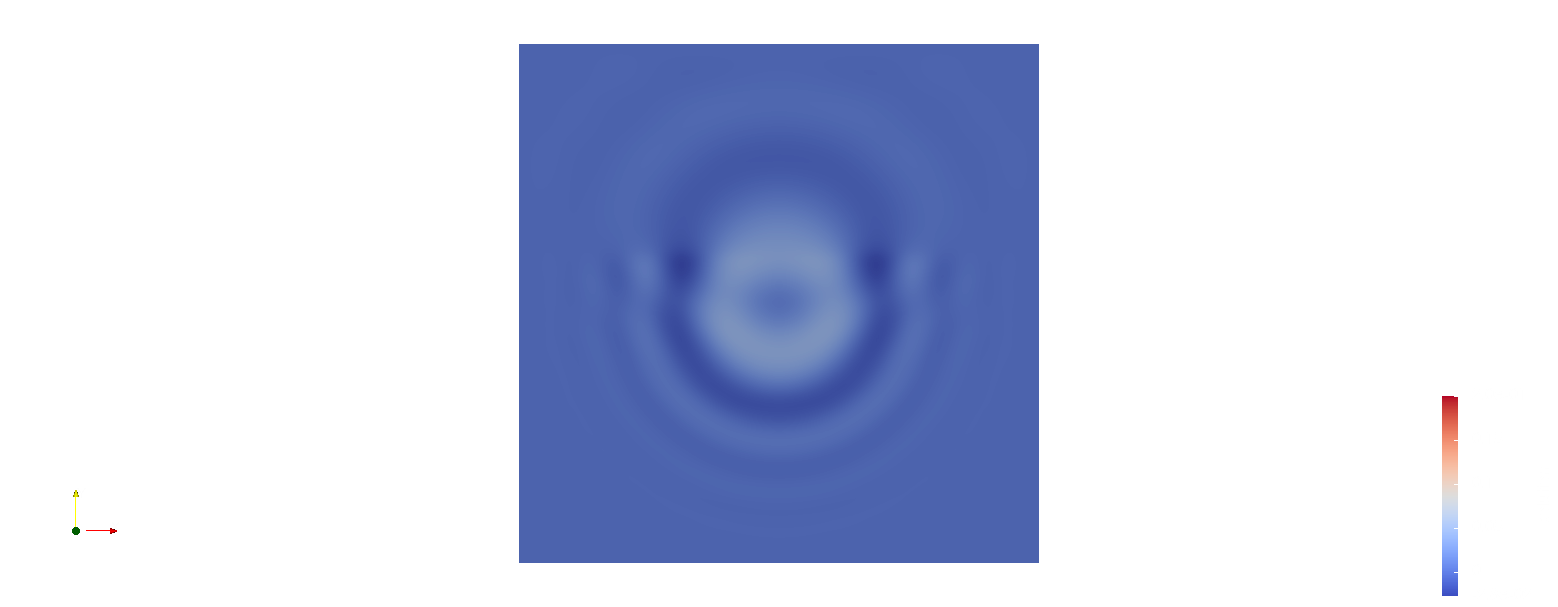} & \includegraphics[clip, trim=515 40 515 40, width=3.5cm]{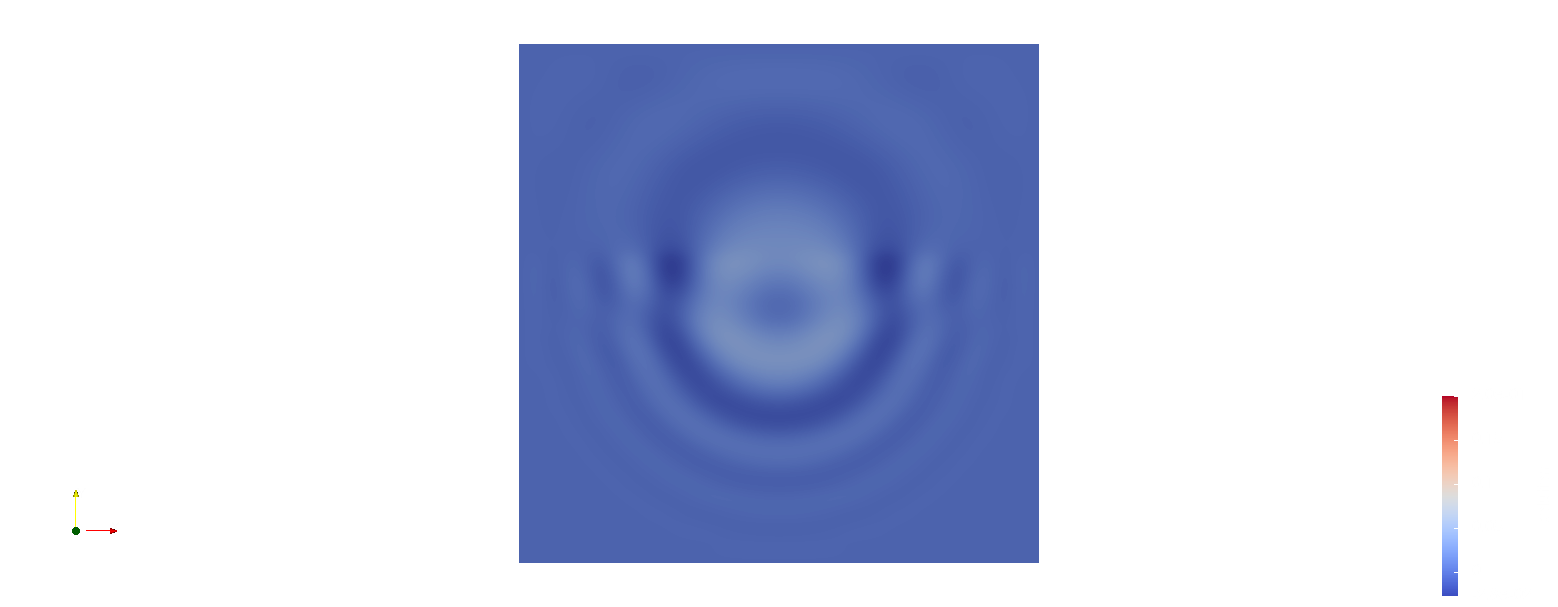}
    \end{tabular}
    \caption{First few time steps of the numerical solution of Problem \eqref{eq:structure_health} using the cGP-C$^1(3)$ method with $ \tau = 0.01 $.}
    \label{fig:structure_health}
\end{figure}
We define the control region $ \Omega_c = (0.75 - l_c, 0.75 + l_c) \times (- l_c, l_c) $ with $ l_c = 1/32 $ that simulates a senor and evaluate the expression
\begin{align}\label{eq:control_quantity}
    u_c(t) = \int_{\Omega_c} u_{\tau, h}(\vec{x}, t) \dx
\end{align}
there.
Our first aim is to examine and compare the respective signals at the sensor  for the three time discretization methods considered using different time step sizes.
All calculations were performed with a fixed $ 32 \times 32 $ spatial mesh discretization and a direct solver for the linear system of equations. In the following, we use the time step size $ \tau = 1/n_t $, where $ n_t $ denotes the number of time intervals.
The corresponding evaluations of the control quantity are depicted in Figure~\ref{fig:control_quantity}.

\begin{figure}[htbp]
    \centering
    \begin{tabular}{c}
        \includegraphics[width=13cm]{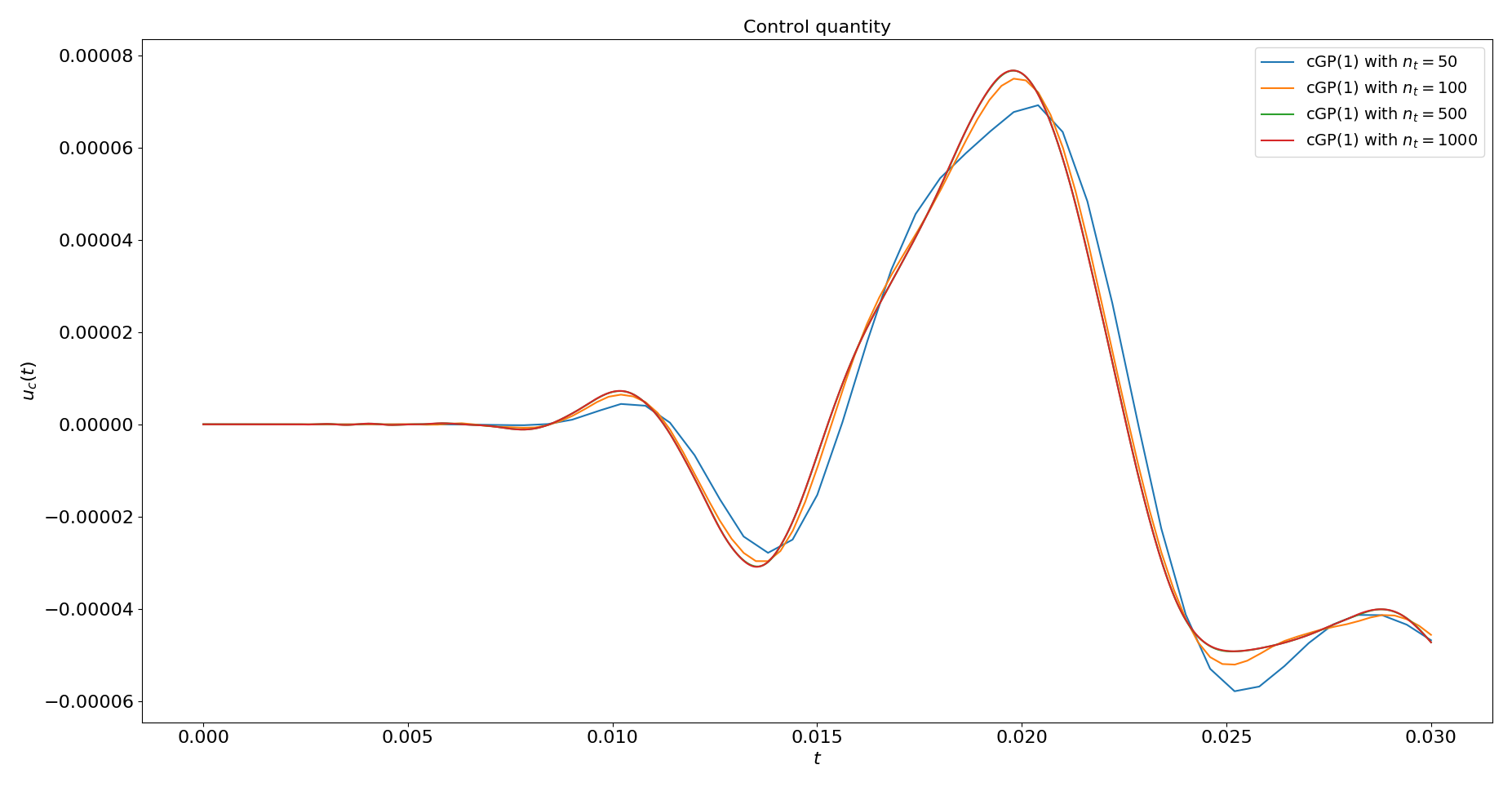} \\
        \includegraphics[width=13cm]{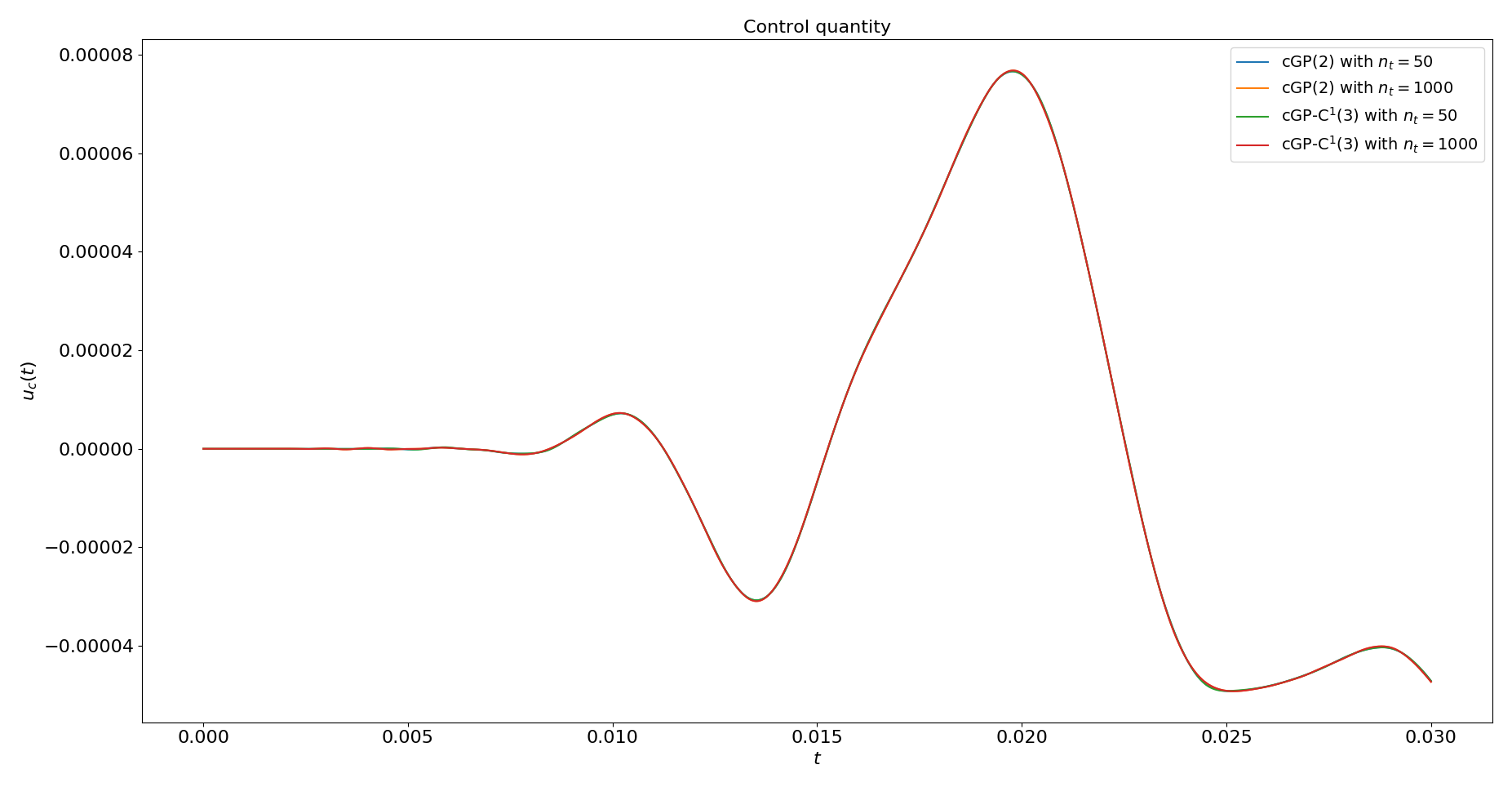}
    \end{tabular}
    \caption{Evaluation of the control quantity \eqref{eq:control_quantity} for the Crank--Nicolson (cGP(1)-), cGP(2)- and cGP-C$^1(3)$-method with different time step sizes.}
    \label{fig:control_quantity}
\end{figure}

The Crank--Nicolson (cGP(1)-) method has lower accuracy in the control quantity~\eqref{eq:control_quantity}, as can be seen in the upper graph in Figure~\ref{fig:control_quantity}. For large time step sizes, the curve of the control quantity deviates significantly from the curve we obtain for small time steps.
In this example, the cGP(2)- and cGP-C$^1(3)$-method have comparable accuracy, as shown in the bottom graph in Figure~\ref{fig:control_quantity}.
Furthermore, we observe that both methods have high accuracy even for large time steps.

Next, we compare the number of non-zero entries in the system matrix resulting from each time discretization scheme. The number of non-zero entries (\texttt{nze}) as well as the number of degrees of freedom (\texttt{dof}) for the Crank--Nicolson, cGP(2)- and cGP-C$^1(3)$-method are summarized in Table~\ref{tab:nze}.
The Crank--Nicolson method has fewer degrees of freedom and a smaller number of non-zero entries, but provides worse accuracy as already seen in the above example.
Although the system matrices of the cGP(2)- and cGP-C$^1(3)$-method have the same size, the cGP-C$^1(3)$-method has fewer non-zero entries in the system matrix.
This reduced number of non-zero entries is especially advantageous for the application of iterative solvers.
We recall that both methods have the same order of convergence in the discrete time points as seen in the previous section.
In \cite{AB20_2}, iterative solvers are applied for solving the linear systems of the cGP-C$^1(3)$ approximation of the wave equation. In this case, a strong superiority of the cGP-C$^1(3)$ approach over the cGP(2) one is observed with respect to accuracy and runtime of the simulations.

\begin{table}[htbp]
    \centering
    \begin{tabular}{rrrrrrr}
        \hline \multicolumn{1}{c}{\multirow{2}{*}{$ \vert \Tau_h \vert $}} & \multicolumn{2}{c}{cGP(1)} & \multicolumn{2}{c}{cGP(2)} & \multicolumn{2}{c}{cGP-C$^1$(3)} \\
         & \multicolumn{1}{c}{\texttt{dof}} & \multicolumn{1}{c}{\texttt{nze}} & \multicolumn{1}{c}{\texttt{dof}} & \multicolumn{1}{c}{\texttt{nze}} & \multicolumn{1}{c}{\texttt{dof}} & \multicolumn{1}{c}{\texttt{nze}} \\
        \hline
        256 & 2312 & 1.4e+05 & 4624 & 5.0e+05 & 4624 & 3.6e+05 \\
        1024 & 8712 & 5.7e+05 & 17424 & 2.0e+06 & 17424 & 1.4e+06 \\
        4096 & 33800 & 2.3e+06 & 67600 & 7.9e+06 & 67600 & 5.7e+06 \\
        16384 & 133128 & 9.1e+06 & 266256 & 3.2e+07 & 266256 & 2.3e+07 \\
        65536 & 528392 & 3.7e+07 & 1056784 & 1.3e+08 & 1056784 & 9.3e+07 \\
        262144 & 2105352 & 1.5e+08 & 4210704 & 5.2e+08 & 4210704 & 3.8e+08 \\
        \hline
    \end{tabular}
    \caption{Number of non-zero entries in the respective system matrix and number of degrees of freedom for the Crank--Nicolson (cGP(1)-), cGP(2)- and cGP-C$^1(3)$-method using different spatial refinements.}
    \label{tab:nze}
\end{table}


\section{Acknowledgement}

The second and third authors acknowledge the support of this work by the Deutsche Forschungsgemeinschaft
(DFG, German Research Foundation) as part of the project ``Physics-oriented solvers for multicompartmental
poromechanics'' under grant number 456235063.

\clearpage
\bibliographystyle{plain}
\bibliography{literature}

\begin{thebibliography}{10}

\bibitem{adams:2003}
R.~A. Adams and J.~J.~F. Fournier.
\newblock {\em Sobolev {S}paces}.
\newblock Elsevier, Amsterdam, 2 edition, 2003.

\bibitem{AMN09}
G.~Akrivis, C.~Makridakis, and R.~H. Nochetto.
\newblock Optimal order a posteriori error estimates for a class of
  {R}unge--{K}utta and {G}alerkin methods.
\newblock {\em Numer. Math.}, 114:133--160, 2009.

\bibitem{AMN11}
G.~Akrivis, C.~Makridakis, and R.~H. Nochetto.
\newblock Galerkin and {R}unge--{K}utta methods: unified formulation, a
  posteriori error estimates and nodal superconvergence.
\newblock {\em Numer. Math.}, 118:429--456, 2011.

\bibitem{AB21}
M.~Anselmann and M.~Bause.
\newblock Higher order {G}alerkin-collocation time discretization with
  {N}itsche’s method for the {N}avier–{S}tokes equations.
\newblock {\em Mathematics and Computers in Simulation, in press}, 2020.

\bibitem{AB20_2}
M.~Anselmann and M.~Bause.
\newblock Numerical study of {G}alerkin-collocation approximation in time for
  the wave equation.
\newblock In W.~Dörfler et~al., editor, {\em Mathematics of Wave Phenomena.
  Trends in Mathematics}, pages 15--36, Cham, 2020. Birkhäuser.

\bibitem{Anselmann2020Galerkin}
M.~Anselmann, M.~Bause, S.~Becher, and G.~Matthies.
\newblock Galerkin-collocation approximation in time for the wave equation and
  its post-processing.
\newblock {\em ESAIM Math. Model. Numer. Anal.}, 54(6):2099--2123, 2020.

\bibitem{anton2016}
P.~F. Antonietti, B.~A. De~Dios, I.~Mazzieri, and A.~Quarteroni.
\newblock Stability analysis of discontinuous {G}alerkin approximations to the
  elastodynamics problem.
\newblock {\em J. Sci. Comput.}, 68(1):143--170, 2016.

\bibitem{AFS68}
J.~H. Argyris, I.~Fried, and D.~W. Scharpf.
\newblock The tuba family of plate elements for the matrix displacement method.
\newblock {\em Aeronaut. J. Roy. Aeronaut. Soc.}, 72:701--709, 1969.

\bibitem{BL94}
L.~Bales and I.~Lasiecka.
\newblock Continuous finite elements in space and time for the nonhomogeneous
  wave equation.
\newblock {\em Computers Math. Appl.}, 27:91--102, 1994.

\bibitem{BGR10}
W.~Bangerth, M.~Geiger, and R.~Rannacher.
\newblock Adaptive {G}alerkin finite element methods for the wave equation.
\newblock {\em Comput. Meth. Appl. Math.}, 10:3--48, 2010.

\bibitem{AB19}
M.~Bause and M.~Anselmann.
\newblock Comparative study of continuously differentiable {G}alerkin time
  discretizations for the wave equation.
\newblock {\em PAMM}, 19, 2019.

\bibitem{BBK20}
M.~Bause, M.~P. Bruchh\"{a}user, and U.~K\"{o}cher.
\newblock Flexible goal-oriented adaptivity for higher-order space-time
  discretizations of transport problems with coupled flow.
\newblock {\em Comp. Math. Appl., in press}, 2020.

\bibitem{Bause2020post-processed}
M.~Bause, U.~K\"{o}cher, F.~A. Radu, and F.~Schieweck.
\newblock Post-processed {G}alerkin approximation of improved order for wave
  equations.
\newblock {\em Math. Comp.}, 89(322):595--627, 2020.

\bibitem{Bause2017error}
M.~Bause, F.~A. Radu, and U.~K\"{o}cher.
\newblock Error analysis for discretizations of parabolic problems using
  continuous finite elements in time and mixed finite elements in space.
\newblock {\em Numer. Math.}, 137(4):773--818, 2017.

\bibitem{BM19}
S.~Becher and G.~Matthies.
\newblock Variational time discretizations of higher order and higher
  regularity.
\newblock {\em arXiv:2003.04056}, 2020.

\bibitem{BMW17}
S.~Becher, G.~Matthies, and D.~Wenzel.
\newblock Variational methods for stable time discretization of first-order
  differential equations.
\newblock In K.~Georgiev, M.~Todorov M, and I.~Georgiev, editors, {\em Advanced
  Computing in Industrial Mathematics. BGSIAM}, pages 63--75, Cham, 2018.
  Springer.

\bibitem{BFS65}
F.~K. Bogner, R.~L. Fox, and L.~A. Schmit.
\newblock The generation of interelement compatible stiffness and mass matrices
  by the use of interpolation formulae.
\newblock In {\em Proc. Conf. Matrix Methods in Struct. Mech., AirForce Inst.
  of Tech., Wright Patterson AF Base}, pages 397--444, Ohio, 1965.

\bibitem{BPGK18}
F.~Bonaldi, D.~A.~Di Pietro, G.~Geymonat, and F.~Krasucki.
\newblock A hybrid high-order method for {K}irchhoff-{L}ove plate bending
  problems.
\newblock {\em ESAIM Math. Model. Numer. Anal.}, 52:393--421, 2018.

\bibitem{BS05}
S.~C. Brenner and L.-Y. Sung.
\newblock ${C}^0$ interior penalty methods for fourth order elliptic boundary
  value problems on polygonal domains.
\newblock {\em J. Sci. Comput.}, 22/23:83--118, 2005.

\bibitem{BM13}
F.~Brezzi and L.~D. Marini.
\newblock Virtual element methods for plate bending problems.
\newblock {\em Comput. Methods Appl. Mech. Engrg.}, 253:455--462, 2013.

\bibitem{Burman2020cut}
E.~Burman, P.~Hansbo, and M.~G. Larson.
\newblock Cut {B}ogner-{F}ox-{S}chmit elements for plates.
\newblock {\em Adv. Model. and Simul. in Eng. Sci.}, 7(27), 2020.

\bibitem{CN21}
C.~Carstensen and N.~Nataraj.
\newblock Lower-order equivalent nonstandard finite element methods for
  biharmonic plates.
\newblock {\em arXiv:2102.08125}, pages 1--37, 2021.

\bibitem{CM16}
C.~Chinosi and L.~D. Marini.
\newblock Virtual element method for fourth order problems: $l^2$-estimates.
\newblock {\em Comput. Math. Appl.}, 72:1959--1967, 2016.

\bibitem{CT66}
R.~W. Clough and J.~L. Tocher.
\newblock Finite element stiffness matrices for analysis of plate bending.
\newblock {\em Matrix Methods in Structural Mechanics},
  (AFFDL-TR-66-80):515--545, 1966.

\bibitem{DE21}
Z.~Dong and A.~Ern.
\newblock Hybrid high-order and weak {G}alerkin methods for the biharmonic
  problem.
\newblock {\em arXiv:2103.16404}, pages 1--28, 2021.

\bibitem{Ern2016discontinuous}
A.~Ern and F.~Schieweck.
\newblock Discontinuous {G}alerkin method in time combined with a stabilized
  finite element method in space for linear first-order {PDE}s.
\newblock {\em Math. Comp.}, 85(301):2099--2129, 2016.

\bibitem{farago}
I.~Farag\'{o}.
\newblock Convergence and stability constant of the theta-method.
\newblock {\em Conference Applications of Mathematics}, 2013.

\bibitem{FP96}
D.~A. French and T.~E. Peterson.
\newblock A continuous space-time finite element method for the wave equation.
\newblock {\em Math. Comput.}, 65:491--506, 1996.

\bibitem{GH09}
E.~H. Georgoulis and P.~Houston.
\newblock Discontinuous {G}alerkin methods for the biharmonic problem.
\newblock {\em IMA J. Numer. Anal.}, 29:573--594, 2009.

\bibitem{grote2006}
M.~J. Grote, A.~Schneebeli, and D.~Sch\"{o}tzau.
\newblock Discontinuous {G}alerkin finite element method for the wave equation.
\newblock {\em SIAM J. Numer. Anal.}, 44(6):2408--2431, 2006.

\bibitem{H00}
T.~J.~R. Hughes.
\newblock {\em The {F}inite {E}lement {M}ethod}.
\newblock Dover Publications, 2000.

\bibitem{J93}
C.~Johnson.
\newblock Discontinuous {G}alerkin finite element methods for second order
  hyperbolic problems.
\newblock {\em Comput. Methods Appl. Mech. Engrg.}, 107:117--129, 1993.

\bibitem{JB09}
H.~Joulak and B.~Beckermann.
\newblock On {G}autschi's conjecture for generalized {G}auss--{R}adau and
  {G}auss--{L}obatto formulae.
\newblock {\em J. Comp. Appl. Math.}, 233:768--774, 2009.

\bibitem{Karakashian1999aspacetime}
O.~Karakashian and C.~Makridakis.
\newblock A space-time finite element method for the nonlinear
  {S}chr{\"o}dinger equation: the continuous galerkin method.
\newblock {\em SIAM J. Numer. Anal.}, 36, 1999.

\bibitem{Karakashian2005convergence}
O.~Karakashian and C.~Makridakis.
\newblock Convergence of a continuous {G}alerkin method with mesh modification
  for nonlinear wave equations.
\newblock {\em Math. Comp.}, 74(249):85--102, 2005.

\bibitem{KB14}
U.~K\"ocher and M.~Bause.
\newblock Variational space-time methods for the wave equation.
\newblock {\em J. Sci. Comput.}, 61:424--453, 2014.

\bibitem{KBB15}
U.~K\"{o}cher, M.~P. Bruchh\"{a}user, and M.~Bause.
\newblock Efficient and scalable data structures and algorithms for
  goal-oriented adaptivity of space–time {FEM} codes.
\newblock {\em SoftwareX}, 10:100239, 2019.

\bibitem{Lions1971optimal}
J.-L. Lions.
\newblock {\em Optimal control of systems governed by partial differential
  equations}.
\newblock Translated from the French by S. K. Mitter. Die Grundlehren der
  mathematischen Wissenschaften, Band 170. Springer-Verlag, New York-Berlin,
  1971.

\bibitem{Lions1972non-homogeneous}
J.-L. Lions and E.~Magenes.
\newblock {\em Non-homogeneous boundary value problems and applications. {V}ol.
  {II}}.
\newblock Springer-Verlag, New York-Heidelberg, 1972.
\newblock Translated from the French by P. Kenneth, Die Grundlehren der
  mathematischen Wissenschaften, Band 182.

\bibitem{M68}
L.~Morley.
\newblock The triangular equilibrium element in the solution of plate bending
  problems.
\newblock {\em Aero. Quart.}, 19:149--169, 1968.

\bibitem{MS03}
I.~Mozolevski and E.~S\"uli.
\newblock A priori error analysis for the hp-version of the discontinuous
  {G}alerkin finite element method for the biharmonic equation.
\newblock {\em Comput. Methods Appl. Math.}, 3:596--607, 2003.

\bibitem{MWY14}
L.~Mu, J.~Wang, and X.~Ye.
\newblock Weak {G}alerkin finite element methods for the biharmonic equation on
  polytopal meshes.
\newblock {\em Numer. Methods Partial Differential Equations}, 30:1003--1029,
  2014.

\bibitem{SM07}
E.~S\"uli and I.~Mozolevski.
\newblock hp-version interior penalty {DGFEM}s for the biharmonic equation.
\newblock {\em Comput. Methods Appl. Mech. Engrg.}, 196:1851--1863, 2007.

\bibitem{WX06}
M.~Wang and J.~Xu.
\newblock The {M}orley element for fourth order elliptic equations in any
  dimensions.
\newblock {\em Numer. Math.}, 103:155--169, 2006.

\bibitem{W84}
W.~L. Wood.
\newblock A unified set of single step algorithms. {P}art {II}: {T}heory.
\newblock {\em Int. J. Numer. Meth. Eng.}, 20:2303--2309, 1984.

\bibitem{W90}
W.~L. Wood.
\newblock {\em Practical {T}ime-stepping {S}chemes}.
\newblock Clarendon Press, 1990.

\bibitem{YZZ20}
X.~Ye, S.~Zhang, and Z.~Zhang.
\newblock A new {P}1 weak {G}alerkin method for the biharmonic equation.
\newblock {\em J. Comput. Appl. Math.}, 364:112337, 2020.

\bibitem{ZZ15}
R.~Zhang and Q.~Zhai.
\newblock A weak {G}alerkin finite element scheme for the biharmonic equations
  by using polynomials of reduced order.
\newblock {\em J. Sci. Comput.}, 64:559--585, 2015.

\bibitem{ZCZ16}
J.~Zhao, S.~Chen, and B.~Zhang.
\newblock The nonconforming virtual element method for plate bending problems.
\newblock {\em Math. Models Methods Appl. Sci.}, 26:1671--1687, 2016.

\bibitem{ZZCM18}
J.~Zhao, B.~Zhang, S.~Chen, and S.~Mao.
\newblock The {M}orley-type virtual element for plate bending problems.
\newblock {\em J. Sci. Comput.}, 76:610--629, 2018.

\bibitem{ZT00}
O.~C. Zienkiewicz and R.~L. Taylor.
\newblock {\em The finite element method. {V}ol. 2: {S}olid {M}echaincs}.
\newblock Butterworth--Heinemann, 5 edition, 2000.

\end{thebibliography}
\clearpage
\appendix

\section{Discrete formulation of the Galerkin--collocation}
We consider \eqref{eq:collocation} with $ k = 3 $ which is the easiest case of Galerkin--collocation.
We define the reference time interval $ \hat{I} = [0, 1] $ and the reference element transformation $ F_{I_{n}} $ along with its inverse
\begin{alignat*}{2}
    F_{I_{n}}: \hat{I} & \to \closure{I}_{n}, & \qquad F_{I_{n}}(\hat{t})
    & = \tau_{n} \hat{t} + t_{n-1}, \\
    F_{I_{n}}^{-1}: \closure{I}_{n} & \to \hat{I}, & \qquad F_{I_{n}}^{-1}(t)
    & = \frac{t - t_{n - 1}}{\tau_{n}}.
\end{alignat*}
Furthermore, we define the mass matrix $ \mat{M} \in \mathbb{R}^{J \times J} $ and the operator matrix $ \mat{A} \in \mathbb{R}^{J \times J} $
\begin{alignat}{2}
    \mat{M}_{ij} & = \int_{\Omega} \varphi_{i} \, \varphi_{j} \dx, & \quad & i, j = 1, \ldots, J, \label{eq:mass_matrix} \\
    \mat{A}_{ij} & = \int_{\Omega} \laplace \varphi_{i} \, \laplace \varphi_{j} \dx, & \quad & i, j = 1, \ldots, J, \label{eq:biharm_matrix}
\end{alignat}
where $ \varphi_{i}, \, i = 1, \ldots, J $ are the global space basis functions of $ V_{h} $.
Consider the one dimensional 
Bogner--Fox--Schmit element for the time discretization, then the temporal basis functions 
$ \{\hat{\xi}_{k}\}_{k=0}^{3} \subset \mathbb{P}_{3}(\hat{I}; \mathbb{R}) $ on the reference element are given as 
\begin{align}
    \label{eq:time_ref_base}
    \begin{aligned}
        \hat{\xi}_{0}(\hat{t}) & = 1 - 3 \hat{t}^{2} + 2 \hat{t}^{3}, & \qquad
        \hat{\xi}_{1}(\hat{t}) & = \hat{t} - 2 \hat{t}^{2} + \hat{t}^{3}, \\
        \hat{\xi}_{2}(\hat{t}) & = 3 \hat{t}^{2} - 2 \hat{t}^{3}, & \qquad
        \hat{\xi}_{3}(\hat{t}) & = - \hat{t}^{2} + \hat{t}^{3}.
    \end{aligned}
\end{align}
We obtain the basis 
$ \{\xi_{k}\}_{k=0}^{3} \subset \mathbb{P}_{3}(\closure{I}_{n}; \mathbb{R}) $ on the time interval $ I_{n} $ by
\begin{align}\label{eq:mapped_basis}
    \xi_{k}(t) = \hat{\xi}_{k}\left(F_{I_{n}}^{-1}(t)\right) \quad \forall t \in \closure{I}_{n}, \, k = 0, \ldots, 3.
\end{align}
For the derivation of the Galerkin--collocation we require the integrals of the following basis functions 
\begin{gather}
    \label{eq:gcint}
        \int_{I_{n}} \xi_{0} \dt = \frac{\tau_{n}}{2}, \qquad
        \int_{I_{n}} \xi_{1} \dt = \frac{\tau_{n}}{12}, \qquad
        \int_{I_{n}} \xi_{2} \dt = \frac{\tau_{n}}{2}, \qquad
        \int_{I_{n}} \xi_{3} \dt = - \frac{\tau_{n}}{12},
\end{gather}
and also the integrals of the time derivatives 
\begin{gather}
    \label{eq:gcdiffint}
        \int_{I_{n}} \partial_{t} \xi_{0} \dt = -1, \qquad
        \int_{I_{n}} \partial_{t} \xi_{1} \dt = 0, \qquad
        \int_{I_{n}} \partial_{t} \xi_{2} \dt = 1, \qquad
        \int_{I_{n}} \partial_{t} \xi_{3} \dt = 0.
\end{gather}
For the discrete functions $ w_{\tau, h} \in \mathbb{P}_{3}(\closure{I}_{n}; V_{h}) $ we use the ansatz
\begin{align}
    \label{eq:varbasis}
    \begin{aligned}
        w_{\tau, h}(\vec{x}, t) & = \sum_{k = 0}^{3} \sum_{j = 1}^{J} w_{n, k, j} \, \varphi_{j}(\vec{x}) \, \xi_{k}(t) 
        = \sum_{k = 0}^{3} w_{n, k}(\vec{x}) \, \xi_{k}(t) \qquad \forall (\vec{x}, t) \in \Omega \times \closure{I}_{n}
    \end{aligned}
\end{align}
with the space dependent coefficient functions 
$
    w_{n, k}(\vec{x}) = \sum_{j = 1}^{J} w_{n, k, j} \, \varphi_{j}(\vec{x}) \in V_{h}
    $
and the constants $ w_{n,k,j} \in \mathbb{R} $. 
For the test functions from $ \mathbb{P}_{0}(\closure{I}_{n}; V_{h}) $ we use the basis 
\begin{align*}
    \mathcal{B} = \{\varphi_{1} \zeta_{0}, \ldots, \varphi_{J} \zeta_{0} \},
\end{align*}
where $ \zeta_{0} \equiv 1 $ on $ \closure{I}_{n} $. 
To evaluate the source term $f$ in equation~ \eqref{eq:gcpde}, we use the time interpolant $ I_{\tau}^H f $, 
which is given on the interval $ I_{n} $ by 
\begin{align}
    \label{eq:interpolator}
    \begin{aligned}
        I_{\tau}^H\restr{I_{n}} f(t) & = f\restr{I_{n}}(t_{n-1}) \, \xi_{0}(t) + \tau_{n} \partial_{t} f\restr{I_{n}}(t_{n-1}) \, \xi_{1}(t) \\
        & \quad + f\restr{I_{n}}(t_{n}) \, \xi_{2}(t) + \tau_{n} \partial_{t} f\restr{I_{n}}(t_{n}) \, \xi_{3}(t) \qquad \forall t \in \closure{I}_{n}.
    \end{aligned}
\end{align} 
There $ f\restr{I_{n}}(t_{n-1}) $ and $ f\restr{I_{n}}(t_{n}) $ are to be understood as the corresponding one sided limit values. 

We want now to derive the discrete formulation of the Galerkin--collocation \eqref{eq:collocation} for the 
dynamic plate vibration problem \eqref{eq:instationary_plate}. 
To do this, we consider first equation~\eqref{eq:gcpde}. 
Due to the definition of the discrete operator $ {A}_{h} $ from \eqref{A_h}, it is equivalent to
\begin{align*}
    \int_{I_{n}} \int_{\Omega} \partial_{t} u^1_{\tau, h} \, \varphi^1_{\tau, h} \dx \dt
    + \int_{I_{n}} \int_{\Omega} \laplace u^0_{\tau,h} \, \laplace \varphi^1_{\tau, h} \dx \dt
    = \int_{I_{n}} \int_{\Omega} f \, \varphi^1_{\tau, h} \dx \dt.
\end{align*}
Insertion of the basis representation~\eqref{eq:varbasis} for $ (u^0_{\tau, h}, u^1_{\tau, h}) \in \mathbb{P}_{3}(\closure{I}_{n}; V_{h})^{2} $ 
and approximation of the source term function $ f $ with the interpolant $ I_{\tau}^H f $ yields 
\begin{align}
    \label{eq:derpde}
    \begin{aligned}
        & \int_{I_{n}} \int_{\Omega} \partial_{t} \left( \sum_{k=0}^{3} u^1_{n,k} \, \xi_{k} \right) \, \varphi^1_{\tau, h} \dx \dt
        + \int_{I_{n}} \int_{\Omega} \laplace \left( \sum_{k=0}^{3} u^0_{n,k} \, \xi_{k} \right) \, \laplace \varphi^1_{\tau, h} \dx \dt \\
        &\quad = \int_{I_{n}} \int_{\Omega} I_{\tau}^H f \ \varphi^1_{\tau, h} \dx \dt.
    \end{aligned}
\end{align}
Subsequent separation of the time-dependent and space-dependent functions together with \eqref{eq:gcdiffint} results in 
\begin{align*}
    \int_{I_{n}} \int_{\Omega} \partial_{t} \left( \sum_{k=0}^{3} u^1_{n,k} \, \xi_{k} \right) \, \varphi^1_{\tau, h} \dx \dt
    & = \sum_{k=0}^{3} \left( \int_{I_{n}} \partial_{t} \xi_{k} \dt \int_{\Omega}  u^1_{n,k} \, \varphi^1_{\tau, h} \dx \right) \\
    & 
    = \int_{\Omega} \left( - u^1_{n,0} + u^1_{n,2} \right) \, \varphi^1_{\tau, h} \dx
\end{align*}
for the first term. Analogously, application of~\eqref{eq:gcint} results in  
\begin{align*}
    & \int_{I_{n}} \int_{\Omega} \laplace \left( \sum_{k=0}^{3} u^0_{n,k} \, \xi_{k} \right) \, \laplace \varphi^1_{\tau, h} \dx \dt 
    = \sum_{k=0}^{3} \left( \int_{I_{n}} \xi_{k} \dt \int_{\Omega} \laplace u^0_{n,k} \, \laplace \varphi^1_{\tau, h} \dx \right) \\
    & \quad 
    = \int_{\Omega} \laplace \left( \frac{\tau_{n}}{2} u^0_{n,0} + \frac{\tau_{n}}{12} u^0_{n,1} + \frac{\tau_{n}}{2} u^0_{n,2} - \frac{\tau_{n}}{12} u^0_{n,3} \right) \, \laplace \varphi^1_{\tau, h} \dx
\end{align*}
for the second term. 
On the right-hand side of~\eqref{eq:derpde} we obtain with the definition of the interpolant~\eqref{eq:interpolator} in the same manner
\begin{align*}
    & \int_{I_{n}} \int_{\Omega} I_{\tau}^H f \ \varphi^1_{\tau, h} \dx \dt \\
    & \quad = \int_{I_{n}} \int_{\Omega} \left( f(t_{n-1}) \, \xi_{0} + \tau_{n} \partial_{t} f(t_{n-1}) \, \xi_{1} + f(t_{n}) \, \xi_{2} + \tau_{n} \partial_{t} f(t_{n}) \, \xi_{3} \right) \, \varphi^1_{\tau, h} \dx \dt \\
    & \quad = \int_{\Omega} \left( \frac{\tau_{n}}{2} f(t_{n-1}) + \frac{\tau_{n}^{2}}{12} \partial_{t} f(t_{n-1}) + \frac{\tau_{n}}{2} f(t_{n}) - \frac{\tau_{n}^{2}}{12} \partial_{t} f(t_{n}) \right) \, \varphi^1_{\tau, h} \dx.
\end{align*}

Next, we consider equation~\eqref{eq:gcvelo} 
\begin{align}
    \label{eq:gcvelofinal}
    \int_{I_{n}} \int_{\Omega} \partial_{t} u^0_{\tau, h} \, \varphi^0_{\tau, h} \dx \dt - \int_{I_{n}} \int_{\Omega} u^1_{\tau, h} \, \varphi^0_{\tau, h} \dx \dt = 0.
\end{align}
Again we employ the basis representation~\eqref{eq:varbasis} and obtain 
\begin{align*}
    \begin{aligned}
        \int_{I_{n}} \int_{\Omega} \partial_{t} u^0_{\tau, h} \, \varphi^0_{\tau, h} \dx \dt
        & = \int_{I_{n}} \int_{\Omega} \partial_{t} \left( \sum_{k=0}^{3} u^0_{n,k} \, \xi_{k} \right) \, \varphi^0_{\tau, h} \dx \dt 
        = \int_{\Omega} \left( - u^0_{n,0} + u^0_{n,2} \right) \, \varphi^0_{\tau, h} \dx
    \end{aligned}
\end{align*}
for the first term and 
\begin{align*}
    \begin{aligned}
        \int_{I_{n}} \int_{\Omega} u^1_{\tau, h} \, \varphi^0_{\tau, h} \dx \dt
        & =\int_{I_{n}} \int_{\Omega} \left( \sum_{k=0}^{3} u^1_{n,k} \, \xi_{k} \right) \, \varphi^0_{\tau, h} \dx \dt \\
        & = \int_{\Omega} \left( \frac{\tau_{n}}{2} u^1_{n,0} + \frac{\tau_{n}}{12} u^1_{n,1} + \frac{\tau_{n}}{2} u^1_{n,2} - \frac{\tau_{n}}{12} u^1_{n,3} \right) \, \varphi^0_{\tau, h} \dx
    \end{aligned}
\end{align*}
for the second, where we have again used~\eqref{eq:gcint} and \eqref{eq:gcdiffint}.

Now we analyze the collocation condition~\eqref{eq:gcpdet}. 
We multiply the equation with a test function $ \varphi^3_{\tau, h} \in V_{h} $ and integrate over the space domain. 
Using the definition of $A_{h} $, \eqref{A_h}, condition~\eqref{eq:gcpdet} 
is then equivalent to
\begin{align}
    \label{eq:gcpdet1}
    \int_{\Omega} \partial_{t} u^1_{\tau, h}(t_{n}) \, \varphi^3_{\tau, h} + \laplace u^0_{\tau, h}(t_{n}) \, \laplace \varphi^3_{\tau, h} \dx & = \int_{\Omega} f(t_{n}) \, \varphi^3_{\tau, h} \dx.
\end{align}

Next, let us take a closer look at the function evaluations. Using~\eqref{eq:varbasis}, \eqref{eq:time_ref_base},~\eqref{eq:mapped_basis} 
and the chain rule we obtain 
\begin{align}
    \label{eq:diffrepr}
    \begin{aligned}
        \partial_{t} w_{\tau, h}(t_{n}) & = \sum_{k=0}^{3} w_{n, k}(\vec{x}) \, \partial_{t} \xi_{k}(t_{n})
        = \sum_{k=0}^{3} w_{n, k}(\vec{x}) \, \partial_{\hat{t}} \hat{\xi}_{k}\left(F_{I_{n}}^{-1}(t_{n})\right) \frac{1}{\tau_{n}} \\
        & = \sum_{k=0}^{3} w_{n, k}(\vec{x}) \, \partial_{\hat{t}} \hat{\xi}_{k}(1) \frac{1}{\tau_{n}}
        = \frac{1}{\tau_{n}} w_{n, 3}(\vec{x})
    \end{aligned}
\end{align}
for $ w_{\tau, h} \in \mathbb{P}_{3}(\closure{I}_{n}; V_{h}) $. The same procedure delivers also the identity 
\begin{align}
    \label{eq:repr}
    \begin{aligned}
        w_{\tau, h}(t_{n}) & = \sum_{k=0}^{3} w_{n, k}(\vec{x}) \, \xi_{k}(t_{n})
        = \sum_{k=0}^{3} w_{n, k}(\vec{x}) \, \hat{\xi}_{k}\left(F_{I_{n}}^{-1}(t_{n})\right) \\
        & = \sum_{k=0}^{3} w_{n, k}(\vec{x}) \, \hat{\xi}_{k}(1)
        = w_{n, 2}(\vec{x}).
    \end{aligned}
\end{align}
Inserting~\eqref{eq:diffrepr} and \eqref{eq:repr} in~\eqref{eq:gcpdet1} thus results in 
\begin{align}
    \label{eq:gcpdetfinal}
    \int_{\Omega} \frac{1}{\tau_{n}} u^1_{n, 3} \, \varphi^3_{\tau, h} + \laplace u^0_{n, 2} \, \laplace \varphi^3_{\tau, h} \dx & = \int_{\Omega} f(t_{n}) \, \varphi^3_{\tau, h} \dx.
\end{align}

The collocation condition~\eqref{eq:gcvelot} reads as 
\begin{align*}
    \partial_{t} u^0_{\tau, h}(t_{n}) - u^1_{\tau, h}(t_{n}) = 0.
\end{align*}
This equation is multiplied with a test function $ \varphi^2_{\tau, h} \in V_{h} $ and 
integrated over the space domain $ \Omega $ to obtain the equation 
\begin{align}
    \label{eq:gcvelot1}
    \int_{\Omega} \partial_{t} u^0_{\tau, h}(t_{n}) \, \varphi^2_{\tau, h} - u^1_{\tau, h}(t_{n}) \, \varphi^2_{\tau, h} \dx & = 0.
\end{align}
After that we use the identities~\eqref{eq:diffrepr} -- \eqref{eq:repr} and insert them in equation~\eqref{eq:gcvelot1}. 
With this we conclude 
\begin{align}
    \label{eq:gcvelotfinal}
    \int_{\Omega} \frac{1}{\tau_{n}} u^0_{n,3} \, \varphi^2_{\tau, h} - u^1_{n,2} \, \varphi^2_{\tau, h} \dx & = 0.
\end{align}

Overall, \eqref{eq:derpde}, \eqref{eq:gcvelofinal}, \eqref{eq:gcpdetfinal} and \eqref{eq:gcvelotfinal} lead to the system: 
\begin{subequations}
    \label{eq:pde_system}
    \begin{align}
        \begin{aligned}
            & 
            \int_{\Omega} \left( - u^0_{n,0} + u^0_{n,2} \right) \, \varphi^0_{\tau, h} \dx 
            \\
            & \quad 
            - \int_{\Omega} \left( \frac{\tau_{n}}{2} u^1_{n,0} + \frac{\tau_{n}}{12} u^1_{n,1} + \frac{\tau_{n}}{2} u^1_{n,2} - \frac{\tau_{n}}{12} u^1_{n,3} \right) \, \varphi^0_{\tau, h} \dx = 0
        \end{aligned}
    \end{align}
    \begin{align}
        \begin{aligned}
            & \int_{\Omega} \left( - u^1_{n,0} + u^1_{n,2} \right) \, \varphi^1_{\tau, h} \dx \\
            & \quad + \int_{\Omega} \laplace \left( \frac{\tau_{n}}{2} u^0_{n,0} + \frac{\tau_{n}}{12} u^0_{n,1} + \frac{\tau_{n}}{2} u^0_{n,2} - \frac{\tau_{n}}{12} u^0_{n,3} \right) \, \laplace \varphi^1_{\tau, h} \dx \\
            & = \int_{\Omega} \left( \frac{\tau_{n}}{2} f(t_{n-1}) + \frac{\tau_{n}^{2}}{12} \partial_{t} f(t_{n-1}) + \frac{\tau_{n}}{2} f(t_{n}) - \frac{\tau_{n}^{2}}{12} \partial_{t} f(t_{n}) \right) \, \varphi^1_{\tau, h} \dx
        \end{aligned}
    \end{align}
    \begin{align}
        \begin{aligned}
            \int_{\Omega} \frac{1}{\tau_{n}} u^0_{n,3} \, \varphi^2_{\tau, h} - u^1_{n,2} \, \varphi^2_{\tau, h} \dx & = 0
        \end{aligned}
    \end{align}
    \begin{align}
        \begin{aligned}
            \int_{\Omega} \frac{1}{\tau_{n}} u^1_{n, 3} \, \varphi^3_{\tau, h} + \laplace u^0_{n, 2} \, \laplace \varphi^3_{\tau, h} \dx & = \int_{\Omega} f(t_{n}) \, \varphi^3_{\tau, h} \dx
        \end{aligned}
    \end{align}
\end{subequations}
The remaining collocation conditions~\eqref{eq:gcu} and \eqref{eq:gcv} yield the identities 
\begin{align*}
    u^0_{n, 0} = u^0_{n-1, 2}, \qquad
    u^0_{n, 1} = u^0_{n-1, 3}, \qquad 
    u^1_{n, 0} = u^1_{n-1, 2}, \qquad
    u^1_{n, 1} = u^1_{n-1, 3},
\end{align*}
thereby reducing the number of unknowns. 

Finally, we bring all known terms in~\eqref{eq:pde_system} to the right side and on the time interval 
$ [t_{n-1}, t_{n}] $ search for the coefficient functions $ \left(u^0_{n,2}, u^0_{n,3}, u^1_{n,2}, u^1_{n,3}\right) \in V_{h}^{4} $ 
from the basis representation~\eqref{eq:varbasis}, which satisfy 
\begin{subequations}
    \label{eq:final_pde_system}
    \begin{align}
        \begin{aligned}
            & 
            \int_{\Omega} u^0_{n,2} \, \varphi^0_{\tau, h} \dx
            - \int_{\Omega} \left( \frac{\tau_{n}}{2} u^1_{n,2} - \frac{\tau_{n}}{12} u^1_{n,3} \right) \, \varphi^0_{\tau, h} \dx \\
            & \quad 
            = \int_{\Omega} u^0_{n,0} \, \varphi^0_{\tau, h} \dx
            + \int_{\Omega} \left( \frac{\tau_{n}}{2} u^1_{n,0} + \frac{\tau_{n}}{12} u^1_{n,1} \right) \, \varphi^0_{\tau, h} \dx
        \end{aligned}
    \end{align}
    \vspace{-2ex}
    \begin{align}
        \begin{aligned}
            &  \int_{\Omega} u^1_{n,2} \, \varphi^1_{\tau, h} \dx
            + \int_{\Omega} \laplace \left( \frac{\tau_{n}}{2} u^0_{n,2} - \frac{\tau_{n}}{12} u^0_{n,3} \right) \, \laplace \varphi^1_{\tau, h} \dx \\
            & \quad = \int_{\Omega} \left( \frac{\tau_{n}}{2} f(t_{n-1}) + \frac{\tau_{n}^{2}}{12} \partial_{t} f(t_{n-1}) + \frac{\tau_{n}}{2} f(t_{n}) - \frac{\tau_{n}^{2}}{12} \partial_{t} f(t_{n}) \right) \, \varphi^1_{\tau, h} \dx \\
            & \qquad + \int_{\Omega} u^1_{n,0} \, \varphi^1_{\tau, h} - \laplace \left( \frac{\tau_{n}}{2} u^0_{n,0} + \frac{\tau_{n}}{12} u^0_{n,1} \right) \, \laplace \varphi^1_{\tau, h} \dx
        \end{aligned}
    \end{align}
    \vspace{-2ex}
    \begin{align}
        \begin{aligned}
            \int_{\Omega} \frac{1}{\tau_{n}} u^0_{n,3} \, \varphi^2_{\tau, h} - u^1_{n,2} \, \varphi^2_{\tau, h} \dx & = 0
        \end{aligned}
    \end{align}
    \vspace{-2ex}
    \begin{align}
        \begin{aligned}
            \int_{\Omega} \frac{1}{\tau_{n}} u^1_{n, 3} \, \varphi^3_{\tau, h} + \laplace u^0_{n, 2} \, \laplace \varphi^3_{\tau, h} \dx
            & = \int_{\Omega} f(t_{n}) \, \varphi^3_{\tau, h} \dx
        \end{aligned}
    \end{align}
\end{subequations}
for all test functions $ \left(\varphi^0_{\tau, h}, \varphi^1_{\tau, h}, \varphi^2_{\tau, h}, \varphi^3_{\tau, h}\right) \in V_{h}^{4} $.

Using the definitions~\eqref{eq:mass_matrix} and \eqref{eq:biharm_matrix} of the mass matrix $ \mat{M} $ 
and operator matrix $ \mat{A} $, respectively  
we can write the system of equations~\eqref{eq:final_pde_system} as a linear system of equations 
\begin{subequations}
    \label{eq:final_system}
    \begin{align}
        \begin{aligned}
            \mat{M} \vec{u}^0_{n, 2} - \mat{M} \left( \frac{\tau_{n}}{2} \vec{u}^1_{n, 2} - \frac{\tau_{n}}{12} \vec{u}^1_{n, 3} \right) = \mat{M} \vec{u}^0_{n, 0} + \mat{M} \left( \frac{\tau_{n}}{2} \vec{u}^1_{n, 0} + \frac{\tau_{n}}{12} \vec{u}^1_{n, 1} \right)
        \end{aligned}
    \end{align}
    \begin{align}
        \begin{aligned}
            & \quad \mat{M} \vec{u}^1_{n, 2} + \mat{A} \left( \frac{\tau_{n}}{2} \vec{u}^0_{n, 2} - \frac{\tau_{n}}{12} \vec{u}^0_{n, 3} \right) \\
            & = \mat{M} \left( \frac{\tau_{n}}{2} \vec{f}_{n, 0} + \frac{\tau_{n}^{2}}{12} \vec{f}_{n, 1} + \frac{\tau_{n}}{2} \vec{f}_{n, 2} - \frac{\tau_{n}^{2}}{12} \vec{f}_{n, 3} \right) \\
            & \quad + \mat{M} \vec{u}^1_{n, 0} - \mat{A} \left( \frac{\tau_{n}}{2} \vec{u}^0_{n, 0} + \frac{\tau_{n}}{12} \vec{u}^0_{n, 1} \right)
        \end{aligned}
    \end{align}
    \vspace{-2ex}
    \begin{align}
        \frac{1}{\tau_{n}} \mat{M} \vec{u}^0_{n, 3} - \mat{M} \vec{u}^1_{n, 2} & = \vec{0}
    \end{align}
    \vspace{-2ex}
    \begin{align}
        \begin{aligned}
            \frac{1}{\tau_{n}} \mat{M} \vec{u}^1_{n, 3} + \mat{A} \vec{u}^0_{n, 2} & = \mat{M} \vec{f}_{n, 2}.
        \end{aligned}
    \end{align}
\end{subequations}
We summarize 
the equations 
\eqref{eq:final_system} more compactly. On 
the time interval 
$ [t_{n-1}, t_{n}] $ we search for the coefficient vector 
\begin{align*}
    \vec{x} = \left( \left(\vec{u}^0_{n, 2}\right)^{\top}, \left(\vec{u}^0_{n, 3}\right)^{\top}, \left(\vec{u}^1_{n, 2}\right)^{\top}, \left(\vec{u}^1_{n, 3}\right)^{\top} \right)^{\top} \in 
    \mathbb{R}^{4J}
\end{align*}
as a solution of the linear system of equations 
\begin{align*}
    \mat{S} \vec{x} = \vec{b}
\end{align*}
with matrix 
\begin{align*}
    \mat{S} =
    \begin{pmatrix}
        \mat{M} & \mat{0} & - \frac{\tau_{n}}{2} \mat{M} & \frac{\tau_{n}}{12} \mat{M} \\ 
        \frac{\tau_{n}}{2} \mat{A} & -\frac{\tau_{n}}{12} \mat{A} & \mat{M} & \mat{0} \\ 
        \mat{0} & \frac{1}{\tau_{n}} \mat{M} & - \mat{M} & \mat{0} \\ 
        \mat{A} & \mat{0} & \mat{0} & \frac{1}{\tau_{n}} \mat{M}
    \end{pmatrix}
    \in \mathbb{R}^{4J \times 4J}
\end{align*}
and right-hand side 
\begin{align*}
    \vec{b} =
    \begin{pmatrix}
        \mat{M} \left(\vec{u}^0_{n, 0} + \frac{\tau_{n}}{2} \vec{u}^1_{n, 0} + \frac{\tau_{n}}{12} \vec{u}^1_{n, 1}\right) \\ 
        b_{n, 1} \\
        \vec{0} \\ 
        \mat{M} \vec{f}_{n, 2}
    \end{pmatrix}
    \in \mathbb{R}^{4J},
\end{align*}
where $ b_{n, 1} = \mat{M} \left( \frac{\tau_{n}}{2} \vec{f}_{n,0} + \frac{\tau_{n}^{2}}{12} \vec{f}_{n, 1} + \frac{\tau_{n}}{2} 
\vec{f}_{n,2} - \frac{\tau_{n}^{2}}{12} \vec{f}_{n,3} + \vec{u}^1_{n, 0} \right) - \mat{A} \left(\frac{\tau_{n}}{2} \vec{u}^0_{n,0} 
+ \frac{\tau_{n}}{12} \vec{u}^0_{n,1}\right) $.

\section{Supplementary material}
\label{sec:supplement}

Recall $ k \geq 3 $.
In the following, we consider Problem \eqref{eq:instationary_plate} with a Lipschitz continuous function $ f(u) $ as right-hand side.
Let $ \Tau_{hn} $ be a partition of $ \Omega $ on the time interval $ I_{n} $. Then, we define the corresponding finite element space by
\begin{align*}
    V_h^n = \set{v_h \in C^1(\Omega)}{v_h\restr{T} \in \mathbb{Q}_3(T),\,v_h\restr{\boundary \Omega} = 0, \, \partial_{\vec{n}}v_h\restr{\boundary \Omega} = 0 \ \forall T \in \Tau_{hn}} \cap H_0^2(\Omega)
\end{align*}
with maximum diameter $ h_n = \max_{K \in \Tau_{hn}} h_K $.
Moreover, we will use the spaces
\begin{align*}
    V_{k}^{n} & = \set{v: \Omega \times I_n \to \R}{v \in \mathbb{P}_{k}(I_{n}; V_h^n)}, \quad 
    V_{k} 
    = \set{v: \Omega \times \left(0, T\right] \to \R}{v\restr{I_n} \in V_k^n}.
\end{align*}
Analogous to \eqref{R_h}, we define the elliptic operator $ R_h^n: V \to V_h^n $ by
\begin{align*}
    \scalar{\laplace R_h^n w}{\laplace v_h} = \scalar{\laplace w}{\laplace v_h} \qquad \forall v_h \in V_h^n.
\end{align*}
We define $ A_h^n: V \to V_h^n $ by
\begin{equation*}
\scalar{A_h^n w}{v_h}=(\Delta w,\Delta v_h) \qquad \forall v_h \in V_h^n
\end{equation*}
and $ \mathcal{L}_h^n: V \times H \to V_h^n \times V_h^n $ by
\begin{align*}
    \mathcal{L}_h^n =
    \begin{pmatrix}
        0 & -P_h^n \\
        A_h^n & 0
    \end{pmatrix}
\end{align*}
analogous to \eqref{A_h_mathcal}, where $ P_h^n $ is the $ L^2$-projection onto $ V_h^n $. Furthermore, we define 
\begin{align*}
    \Pi^n =
    \begin{pmatrix}
        R_h^n & 0 \\
        0 & P_h^n
    \end{pmatrix}.
\end{align*}
Now we can write our problem as:
Find $ U_{\tau, h} = (u_{\tau, h}^0, u_{\tau, h}^1) \in V_k \times V_k $ satisfying
\begin{align}
    \label{eq:adaptive_plate}
    \begin{aligned}
    \int_{I_n} \multiscalar{\partial_t U_{\tau, h}}{\Phi} + \multiscalar{\mathcal{L}_h^n U_{\tau, h}}{\Phi} \dt & = \int_{I_n} \multiscalar{F_h(U_{\tau, h})}{\Phi} \dt, \quad \forall \Phi \in V_{k-1}^n \times V_{k-1}^n, \\
    U_{\tau, h}(t_{n-1}^+) & = \Pi^{n} U_{\tau, h}(t_{n-1}), \quad n = 1, \ldots, N,
    \end{aligned}
\end{align}
where $ F_h(U_{\tau, h}) = (0, P_h^{n} f(u_{\tau, h}^0)) $ and $ U_{\tau, h}(t_0) = (u_0, u_1) $.
For our analysis we also require the Gauss-Legendre quadrature 
\begin{equation}\label{2.1}
\int_{0}^1 g(\bar{\tau}) \mathrm{d}\bar{\tau} \approx \sum_{j=1}^k w_j g(\bar{\tau}_j), \quad 0< \bar{\tau}_1 < \ldots < \bar{\tau}_k <1,
\end{equation}
where $w_j$ denote the weights and $\bar{\tau}_j$, $j=1,\ldots,k$ the abscissas. 
Remember that~\eqref{2.1} is exact for all polynomials of degree smaller or equal to $2k-1$. 


We denote with $\{l_i\}_{i=1}^k$ the Lagrange polynomials of degree $k-1$ associated with $\bar{\tau}_1,\ldots,\bar{\tau}_k$ 
and with $\{\hat{l}_i\}_{i=0}^k$ the Lagrange polynomials of degree $k$ associated with the $k+1$ points 
$0=\bar{\tau}_0 < \bar{\tau}_1<\ldots < \bar{\tau}_k$. 


We map $[0,1]$ onto $\bar{I}_{n}$ via the linear transformation $t=t_{n-1} + \bar{\tau} \tau_n$ and adapt~\eqref{2.1} by defining 
its abscissas and weights as given below 
\begin{align*}
&t_{n,i}=t_{n-1}+\bar{\tau}_i \tau_n, \nonumber \\
&l_{n,i}(t) = l_i(\bar{\tau}), \quad t=t_{n-1} + \bar{\tau} \tau_n, \\
&w_{n,i}=\int_{t_{n-1}}^{t_n} l_{n,i}(t)\dt = \tau_n \int_0^1 l_i(\bar{\tau})\mathrm{d}{\bar{\tau}} = \tau_n w_i, \quad i=1,\ldots,k.\nonumber
\end{align*}

In particular, the following representation holds 
\begin{equation*}
U_{\tau,h}(\vec{x},t) = \sum_{j=0}^k \hat{l}_{n,j}(t) U_{\tau,h}^{n,j}(\vec{x}),\qquad (\vec{x},t)\in \Omega\times I_n, \; t_{n,0}=t_{n-1},
\end{equation*}
where $U_{\tau,h}^{n,j}=U_{\tau, h}(\vec{x},t_{n,j})\in V_h^n \times V_h^n$ and $U_{\tau,h}^{n,0}=U_{\tau,h}(t_{n-1}^+)=\Pi^n U_{\tau,h}(t_{n-1})$ is given.

We denote by $ t_{n, j}^{GL} $ the points and by $ w_{n, j}^{GL} $ the weights of the $ (k + 1)$-point Gauss-Lobatto quadrature formula on the interval $ I_n $, which is exact for polynomials up to degree $ 2k -1 $. Moreover, we define the associated Lagrange interpolator by $ I_{n}^{GL} $.


The following norm equivalence, see~\cite{Karakashian2005convergence}, will be used in the analysis 
\begin{equation*}
C_1 \left(\tau_n \sum_{j=0}^k \Vert v^j\Vert^2 \right)^{1/2} \le
\Vert v \Vert_{L^2(I_n;L^2)}
\le C_2 \left(\tau_n \sum_{j=0}^k \Vert v^j\Vert^2 \right)^{1/2},\quad v\in V_k
\end{equation*}
and is a consequence of 
\begin{equation}\label{2.5}
\max_{t \in I_n}\vert y(t)\vert \le C \tau_n^{-1/2} \left( \int_{I_n} \vert y(t)\vert^2 \dt \right)^{1/2}, \quad \forall y\in \mathbb{P}_k(I_n)
\end{equation}
and 
$$
\int_{I_n}\vert \hat{l}_{n,j}(t) \vert \dt \le c \tau_n,
$$
where $v=\sum_{j=0}^k \hat{l}_{n,j}v^j \in V_k^n$.

We also consider the $L^2$-projection operator $P_t^{n,k-1}:\mathbb{P}_k[t_{n-1},t_n]\rightarrow \mathbb{P}_{k-1}[t_{n-1},t_n] $ for which it holds 
\begin{equation}\label{2.6}
P_t^{n,k-1} = I^{GLe}_{n,k-1},
\end{equation}
where $I^{GLe}_{n,k-1}$ denotes the Lagrange interpolation operator corresponding to the $k$ Gauss-Legendre points 
$t_{n,1}<\ldots < t_{n,k}$.


Let $\Phi \in V_{k-1}^n \times V_{k-1}^n$, then 
\begin{align}\label{2.7}
\int_{I_n} ((U_{\tau,h})_t,\Phi) \dt = \sum_{i,j=1}^k m_{ij} (U_{\tau,h}^{n,j},\Phi^i)+ \sum_{i=1}^k m_{i0} (U_{\tau,h}^{n,0},\Phi^i),
\end{align}
where 
\begin{align*}
m_{ij} = \int_{I_n}  \hat{l}_{n,j}^{'}(t) l_{n,i}(t)\dt, \qquad i=1,\ldots,k,\; j=0,\ldots, k
\end{align*}
and $v^i = v(t_{n,i})$. 


The positivity of the matrix $\mathcal{M}=\{m_{ij}\}_{i,j=1}^k$ is crucial for the stability of the method.

The next lemma proven in~\cite{Karakashian1999aspacetime} 
demonstrates that $\tilde{\mathcal{M}}=D^{-1/2}\mathcal{M}D^{1/2}$ where $D=\text{diag}\{\bar{\tau}_1,\ldots,\bar{\tau}_k\}$ is positive definite.
\begin{lemma} \label{lemma2_1} 
For $\alpha:=\frac{1}{2}\min_{j}\frac{w_j}{\bar{\tau}_j}$ it holds 
\begin{equation*}
\vec{x}^T \mathcal{\tilde{M}}\vec{x} \ge \alpha \vert \vec{x} \vert^2= \alpha \left( \sum_{i=1}^k x_i^2 \right), \qquad \forall \vec{x} \in \mathbb{R}^k.
\end{equation*}
\end{lemma}


Subsequently, we will use the error splitting $ U_{\tau, h} - U = (U_{\tau, h} - W) + (W - U) $. For this purpose, we define
\begin{align}\label{omega_eta}
    \omega(\vec{x}, t) = R_{h}^n u(\vec{x}, t), \qquad
    \eta = u - \omega,
\end{align}
on $ \Omega \times I_{n} $ for all $ n = 1, \ldots, N $.  Furthermore, we define $ W = (W_0, W_1) $ by
\begin{align}
\label{W_1}
    W_1(\vec{x}, t)\restr{I_n} = I^{GL}_{n} \omega_t, \qquad
    W_1(t_0) = R_h^0 u_1,
\end{align}
and
\begin{align}\label{W_0}
    W_0(\vec{x}, t)\restr{I_n} = I^{GL}_{n} \left( \int_{t_{n-1}}^{t} W_1 \ds + \omega(t_{n-1}^+) \right).
\end{align}
Let $ Z(\vec{x}, t)\restr{I_n} = \int_{t_{n-1}}^{t} W_1 \ds + \omega(t_{n-1}^+) $ and $ \tilde{E} = \tilde{E}\restr{I_n} = U_{\tau,h} - W $.


\begin{lemma} 
\label{W_identity}
It is fulfilled that 
\begin{equation}\label{eq:3.8}
\int_{I_n} (W_{0,t},\varphi)\dt = \int_{I_n}(W_1,\varphi)\dt ,\quad \forall \varphi\in V_{k-1}.
\end{equation}
\end{lemma}
\begin{proof}
    From the definition of $ Z $, we directly obtain $ W_0\restr{I_n} = I_n^{GL} Z $. 
    Using partial integration together with the definition of the Lagrange interpolator at the Gauss-Lobatto points and the exactness of the quadrature rule, we have
    \begin{align*}
        \int_{I_{n}} \scalar{W_{0, t}}{\varphi} dt
        & = - \int_{I_{n}} \scalar{W_{0}}{\varphi_{t}} dt + \scalar{Z}{\varphi}(t_{n}) - \scalar{Z}{\varphi}(t_{n-1}^+) \\
        & = - \sum_{j=0}^{k} w_{n, j}^{GL} \scalar{Z}{\varphi_{t}}(t_{n, j}^{GL}) + \scalar{Z}{\varphi}(t_{n}) - \scalar{Z}{\varphi}(t_{n-1}^+) \\
        & = - \int_{I_{n}} \scalar{Z}{\varphi_{t}} dt + \scalar{Z}{\varphi}(t_{n}) - \scalar{Z}{\varphi}(t_{n-1}^+) \\
        & = \int_{I_{n}} \scalar{Z_{t}}{\varphi} dt = \int_{I_{n}} \scalar{W_{1}}{\varphi} dt,
    \end{align*}
    since $ \scalar{Z}{\varphi_t} \in \mathbb{P}_{2k-1} $.
\end{proof}

\begin{lemma}\label{lemma_3.2} 
We define
$$
\begin{array}{rcl}
A_I^n & := & I_{n}^{GL}\left(\int_{t_{n-1}}^t(I-I_{n}^{GL})u_t\ds\right), \\
A_{I\!I}^n & := & u_{tt}-W_{1,t}, \\
A_{I\!I\!I}^n & := &f(W_0)-f(u)-(I_{n}^{GL}-I) \Delta^2 u.
\end{array}
$$
Then 
\begin{equation}\label{3.9}
\int_{I_n}(\tilde{E}_t+\mathcal{L}_h^n \tilde{E},\Phi)\dt = \int_{I_n}\{ (f(u_{\tau, h}^0)-f(W_0),\varphi_1)+(\Delta^2 A_I^n,\varphi_1)+(A_{I\!I}^n+A_{I\!I\!I}^n,\varphi_1) \}\dt
\end{equation}
for all $ \Phi = (\varphi_0, \varphi_1) \in V_{k-1}^n \times V_{k-1}^n $ and $ n = 1, \ldots, N $.
\end{lemma}
\begin{proof}
    First of all, we mention that
    \begin{align*}
        \tilde{E}_t + \mathcal{L}_h^n \tilde{E} & = (\partial_t U_{\tau, h} + \mathcal{L}_h^n U_{\tau, h}) - (\partial_t W + \mathcal{L}_h^n W) 
        = F_h(U_{\tau, h}) - (\partial_t W + \mathcal{L}_h^n W),
    \end{align*}
    where $ F_h(U_{\tau, h}) = (0, P_h^n f(u_{\tau, h}^0)) $. Then, it follows
    \begin{align*}
        \int_{I_n} \multiscalar{\tilde{E}_t + \mathcal{L}_h^n \tilde{E}}{\Phi} \dt
        & = \int_{I_n} \multiscalar{F_h(U_{\tau, h})}{\Phi} - \multiscalar{W_t + \mathcal{L}_h^n W}{\Phi} \dt 
       \\
        & 
        = \int_{I_n} \multiscalar{f(u_{\tau, h}^0)}{\varphi_{1}} - \multiscalar{W_t + \mathcal{L}_h^n W}{\Phi} \dt.
    \end{align*}
    Using the definition of the operator $ \mathcal{L}_h^n $ and \eqref{eq:3.8}, we get
    \begin{align*}
        \int_{I_{n}} \multiscalar{W_t + \mathcal{L}_h^n W}{\Phi} \dt
        & = \int_{I_{n}} \scalar{W_{0,t} - W_{1}}{\varphi_0} + \scalar{W_{1, t}}{\varphi_1} + \scalar{\laplace W_0}{\laplace \varphi_1} \dt \\
        & = \int_{I_{n}} \scalar{W_{1, t}}{\varphi_1} + \scalar{\laplace W_0}{\laplace \varphi_1} \dt.
    \end{align*}
    With the definitions of $ W $ and $\omega $ along with the exactness of the quadrature rule, it follows 
   that
    \allowdisplaybreaks
    \begin{align*}
        \int_{I_{n}} \scalar{\laplace W_0}{\laplace \varphi_1} \dt
        & = \sum_{j=0}^{k} w_{n,j}^{GL} \scalar{\laplace Z}{\laplace \varphi_1}(t_{n, j}^{GL}) \\ 
        & = \sum_{j=0}^{k} w_{n,j}^{GL} \scalar{\laplace \left[\int_{t_{n-1}}^{t_{n, j}^{GL}} W_1 \ds + \omega(t_{n-1}^+)\right]}{\laplace \varphi_1(t_{n, j}^{GL})} \\
        & = \sum_{j=0}^{k} w_{n,j}^{GL} \scalar{\laplace \left[\int_{t_{n-1}}^{t_{n, j}^{GL}} I_{n}^{GL} u_t \ds + u(t_{n-1}^+)\right]}{\laplace \varphi_1(t_{n, j}^{GL})} \\
        & = \sum_{j=0}^{k} w_{n,j}^{GL} \scalar{\laplace \left[\int_{t_{n-1}}^{t_{n, j}^{GL}} I_{n}^{GL} u_t - u_t \ds\right]}{\laplace \varphi_1(t_{n, j}^{GL})} \\
        & \quad + \sum_{j=0}^{k} w_{n,j}^{GL} \scalar{\laplace \left[\int_{t_{n-1}}^{t_{n, j}^{GL}} u_t \ds + u(t_{n-1}^+) \right]}{\laplace \varphi_1(t_{n, j}^{GL})} \\
        & = T_1 + T_2.
    \end{align*}
    For the first term, we get
    \begin{align*}
        T_1 & = \int_{I_{n}} \scalar{\laplace I_{n}^{GL} \left[\int_{t_{n-1}}^{t} I_{n}^{GL} u_t - u_t \ds\right]}{\laplace \varphi_1} \dt \\
        & = \int_{I_{n}} \scalar{\laplace^2 I_{n}^{GL} \left[\int_{t_{n-1}}^{t} (I_{n}^{GL} - I) u_t \ds\right]}{\varphi_1} \dt \\
        & = - \int_{I_{n}} \scalar{\laplace^2 A_{I}^n}{\varphi_1} \dt.
    \end{align*}
    Using the exactness of the quadrature rule, partial integration and problem \eqref{eq:instationary_plate} with $ f = f(u) $, we obtain
    \begin{align*}
        T_2 & = \sum_{j=0}^{k} w_{n,j}^{GL} \scalar{\laplace \left[\int_{t_{n-1}}^{t_{n, j}^{GL}} u_t \ds + u(t_{n-1}^+) \right]}{\laplace \varphi_1(t_{n, j}^{GL})} \\
        & = \sum_{j=0}^{k} w_{n,j}^{GL} \scalar{\laplace u(t_{n, j}^{GL})}{\laplace \varphi_1(t_{n, j}^{GL})} \\
        & = \int_{I_{n}} \scalar{\laplace I_n^{GL} u}{\laplace \varphi_1} \dt \\
        & = \int_{I_{n}} \scalar{(I_n^{GL} - I) \laplace^2 u + f(u) - f(W_0)}{\varphi_1} \dt + \int_{I_{n}} \scalar{f(W_0) - u_{tt}}{\varphi_1} \dt \\
        & = - \int_{I_{n}} \scalar{A_{I\!I\!I}^n}{\varphi_1} \dt + \int_{I_{n}} \scalar{f(W_0) - u_{tt}}{\varphi_1} \dt.
    \end{align*}
    Combining the equations gives
    \begin{align*}
        \int_{I_n} \multiscalar{\tilde{E}_t + \mathcal{L}_h^n \tilde{E}}{\Phi} \dt
        & = \int_{I_n} \multiscalar{f(u_{\tau, h}^0) - f(W_0)}{\varphi_{1}} \dt + \int_{I_{n}} \scalar{u_{tt} - W_{1, t}}{\varphi_1} \dt \\
        & \quad
        + \int_{I_{n}} \scalar{\laplace^2 A_{I}^n}{\varphi_1} \dt
        + \int_{I_{n}} \scalar{A_{I\!I\!I}^n}{\varphi_1} \dt.
    \end{align*}
 This completes the proof.
\end{proof}

The next lemma demonstrates the approximation properties of $W_0$ and $W_1$ and further provides estimates for
$A_I^n$, $A_{I\!I}^n$ and $A_{I\!I\!I}^n$.


\begin{lemma}\label{lemma3_3} 
\begin{itemize}
\item[(i)]
Consider $W_0$ and $W_1$ as defined in \eqref{W_0} and \eqref{W_1}, respectively. For $p=2$ and $p=\infty$ it holds
\begin{align}\label{3.10}
\begin{aligned}
\Vert W_0 - u\Vert_{L^p(I_n;L^2)} & \le c \tau_n^{k+1} \Vert \vert u^{(k+1)} \vert
+\vert u^{(k+2)}\vert \Vert_{L^p(I_n;L^2)} \\
& \quad + ch_n^4 \Vert \vert u\vert + \tau_n \vert u_t \vert +\tau_n^2 \vert u_{tt}\vert \Vert_{L^p(I_n;H^4)},
\end{aligned}
\end{align}
\begin{align}\label{3.11}
\begin{aligned}
\Vert W_1 - u_t \Vert_{L^p(I_n;L^2)} & \le c \tau_n^{k+1} \Vert u^{(k+2)}\Vert_{L^p(I_n;L^2)}  
+ ch_n^4 \Vert \vert u_t\vert +\tau_n \vert u_{tt}\vert \Vert_{L^p(I_n;H^4)},
\end{aligned}
\end{align}
where $u^{(m)}:= \partial_t^m u$.
\item[(ii)]
Consider $A_I^n$, $A_{I\!I}^n$ and $A_{I\!I\!I}^n$ from Lemma~\ref{lemma_3.2}. They fulfill the following estimates 
\begin{align}\label{3.12}
\Vert \Delta^2 A_I^n\Vert_{L^2(I_n;L^2)} & \le c \tau_n^{k+1} \Vert \Delta^2 u^{(k+1)}\Vert_{L^2(I_n;L^2)},
\end{align}
\begin{align}\label{3.13}
\begin{aligned}
\vert \int_{I_n} (A_{I\!I}^n,\varphi)\dt\vert & \le c \left(\tau_n^{k+1}\Vert u^{(k+3)}\Vert_{L^2(I_n;L^2)}\right. \\
& \left.\quad +h_n^4 \Vert u_{tt}\Vert_{L^2(I_n;H^4)}\right)
\Vert \varphi\Vert_{L^2(I_n;L^2)}, \quad \forall \varphi\in V_{k-1},
\end{aligned}
\end{align}
\begin{align}\label{3.14}
\Vert A_{I\!I\!I}^n \Vert_{L^2(I_n;L^2)} & \le c \Vert W_0 - u\Vert_{L^2(I_n;L^2)} 
+ c \tau_n^{k+1} \Vert \Delta^2 u^{(k+1)}\Vert_{L^2(I_n;L^2)}.
\end{align}
\end{itemize}
\end{lemma}
\begin{proof}
\begin{itemize}
\item[(i)]
Here, the more difficult case $p=2$ is considered. We will require the following 
results as discussed in~\cite{Karakashian2005convergence} 
\begin{align} \label{3.15}
\Vert I_{n}^{GL} \phi \Vert_{L^2(I_n;L^2)} \le c \Vert \phi \Vert_{L^2(I_n;L^2)} + c\tau_n \Vert \phi_t \Vert_{L^2(I_n;L^2)},
\end{align}
\begin{align}\label{3.16}
\Vert \phi \Vert_{L^2(I_n;L^2)} \le c\tau_n \Vert \psi \Vert_{L^2(I_n;L^2)}, \quad \text{where} \; \phi = \int_{t_{n-1}}^t \psi \ds. 
\end{align}
We have the representation 
\begin{align*}
W_0 - u = I_n^{GL} \int_{t_{n-1}}^t (W_1 -\omega_t)\ds + I_n^{GL}\omega - u.
\end{align*}
Using first \eqref{3.15} and then \eqref{3.16} one gets
\begin{align}\label{3.17}
\Vert I_n^{GL} \int_{t_{n-1}}^t(W_1-\omega_t)\ds \Vert_{L^2(I_n;L^2)} & \le 
c \Vert \int_{t_{n-1}}^t (W_1 -\omega_t) \ds \Vert_{L^2(I_n;L^2)} \\ & \quad \nonumber + c\tau_n \Vert W_1 -\omega_t \Vert_{L^2(I_n;L^2)} \\
& \nonumber \le 
c\tau_n  \Vert W_1 -\omega_t \Vert_{L^2(I_n;L^2)}.
\end{align}
Using~\eqref{W_1} and~\eqref{omega_eta} we write 
\begin{align*}
W_1 - \omega_t = I_n^{GL} \omega_t - \omega_t  = -I_{n}^{GL} \eta_t - (u_t-I_n^{GL}u_t)+\eta_t
\end{align*}
and applying~\eqref{3.15} and \eqref{ell_proj_err2} for $m=0$ we obtain
\begin{align}\label{3.18}
\begin{aligned}
\Vert I_n^{GL} \eta_t\Vert_{L^2(I_n;L^2)} + \Vert \eta_t \Vert_{L^2(I_n;L^2)} & \le c \Vert \eta_t \Vert_{L^2(I_n;L^2)} 
+ c\tau_n \Vert \eta_{tt}\Vert_{L^2(I_n;L^2)}  \\
& \le c h_n^4 \Vert u_t\Vert_{L^2(I_n;H^4)} + c \tau_n h_n^4 \Vert u_{tt}\Vert_{L^2(I_n;H^4)}. 
\end{aligned}
\end{align}
Moreover, for the approximation properties of $I_n^{GL}$ we have 
\begin{align}\label{3.19}
\Vert u_t - I_n^{GL} u_t\Vert_{L^2(I_n;L^2)} \le c \tau_n^{k+1}\Vert u^{(k+2)}\Vert_{L^2(I_n;L^2)}
\end{align}
and can estimate $I_n^{GL}\omega - u = -I_n^{GL} \eta + I_n^{GL} u - u$ as follows
\begin{align}\label{3.20}
\begin{aligned}
\Vert I_n^{GL}\omega - u\Vert_{L^2(I_n;L^2)} & \le \Vert I_n^{GL}\eta \Vert_{L^2(I_n;L^2)} + \Vert I_n^{GL}u - u\Vert_{L^2(I_n;L^2)} \\
& \le c\tau_n^{k+1} \Vert u^{(k+1)}\Vert_{L^2(I_n;L^2)} + c \Vert \eta\Vert_{L^2(I_n;L^2)} \\
& \quad +c \tau_n \Vert \eta_t \Vert_{L^2(I_n;L^2)} \\
& \le c\tau_n^{k+1}\Vert u^{(k+1)}\Vert_{L^2(I_n;L^2)} + c h_n^4 \Vert u\Vert_{L^2(I_n;H^4)} \\
& \quad + c\tau_n h_n^4 \Vert u_t \Vert_{L^2(I_n;H^4)}.
\end{aligned}
\end{align}
Collecting~\eqref{3.17}--\eqref{3.20} we obtain for $\Vert W_0-u\Vert_{L^2(I_n;L^2)}$ estimate~\eqref{3.10}.

Estimate~\eqref{3.11} follows directly from~\eqref{3.19} and \eqref{3.18} where we have used the representation 
$u_t-W_1 = u_t-I_n^{GL} u_t +I_n^{GL}\eta_t$.
\item[(ii)]
Using~\eqref{3.15}, \eqref{3.16} and \eqref{3.19} we obtain  
\begin{align*}
\Vert \Delta^2 A_I^n \Vert_{L^2(I_n;L^2)} & = \Vert I_n^{GL} \int_{t_{n-1}}^t (I-I_n^{GL})\Delta^2 u_t \ds \Vert_{L^2(I_n;L^2)} \\
& \le c \tau_n \Vert (I - I_n^{GL}) \Delta^2 u_t \Vert_{L^2(I_n;L^2)} 
 \le c \tau_n^{k+1} \Vert \Delta^2 u^{(k+1)}\Vert_{L^2(I_n;L^2)}
\end{align*}
which demonstrates~\eqref{3.12}.

Applying~\eqref{3.19} and the Lipschitz continuity of $f$, estimate~\eqref{3.14} is easily derived as follows
\begin{align*}
\Vert A_{I\!I\!I}^n \Vert_{L^2(I_n;L^2)} & \le \Vert f(W_0)-f(u)\Vert_{L^2(I_n;L^2)}+\Vert(I-I_n^{GL})\Delta^2 u\Vert_{L^2(I_n;L^2)} \\
& \le c \Vert W_0-u\Vert_{L^2(I_n;L^2)}+ c\tau_n^{k+1}\Vert \Delta^2 u^{(k+1)}\Vert_{L^2(I_n;L^2)}.
\end{align*}

Let $\varphi\in V_{k-1}$, then using the definition of $A_{I\!I}^n$, \eqref{W_1}, \eqref{omega_eta} and also that 
$\eta_t = u_t - \omega_t$ we obtain
\begin{align*}
\int_{I_n}(A_{I\!I}^n,\varphi)\dt = \int_{I_n}(u_{tt}-(I_n^{GL}u_t)_t,\varphi)\dt + \int_{I_n}((I_n^{GL}\eta_t)_t,\varphi)\dt = : \Gamma_1 + \Gamma_2.
\end{align*}

Integrating by parts and since the endpoints of $I_n$ are included in the Gauss-Lobatto points we find 
\begin{align*}
\Gamma_1 =  \int_{I_n}(u_{tt}-(I_n^{GL}u_t)_t,\varphi)\dt = - \int_{I_n}(u_{t}-I_n^{GL}u_t,\varphi_t)\dt.
\end{align*}

Next, introduce the Lagrange interpolation operator $I^{n,k+1}$ at the $k + 2$ points
of $[t_{n-1}, t_{n}]$ consisting of the $k + 1$ Gauss-Lobatto points $t_{n,0}^{GL},\ldots,t_{n,k}^{GL}$ and any number 
in $[t_{n-1}, t_{n}]$ that is distinct from these points. We have $\varphi_t \in \mathbb{P}_{k-2}$ in $t$ and 
$I^{n,{k+1}}u_t\in \mathbb{P}_{k+1}$ in $t$, and, therefore $(I^{n,k+1}u_t)\varphi_t$ is a polynomial of degree $2k-1$ in $t$ and 
it holds that 
\begin{align*}
\int_{I_n}(u_t-I_n^{GL}u_t,\varphi_t)\dt =  \int_{I_n} (u_t-I^{n,{k+1}}u_t,\varphi_t) \dt.
\end{align*}

Integrating by parts, applying the Cauchy-Schwarz inequality and using the approximation properties of the operator 
$I^{n,k+1}$, for $\Gamma_1$ we finally obtain 
\begin{align*}
\vert \Gamma_1\vert = \vert \int_{I_n} ([u_t-I^{n,k+1}u_t]_t, \varphi) \dt\vert \le c \tau_n^{k+1} \Vert u^{(k+3)} \Vert_{L^2(I_n;L^2)} 
\Vert \varphi\Vert_{L^2(I_n;L^2)}.
\end{align*}

Next, viewing $\eta_t(t_{n-1}^{+})$ as a constant in time function, $\Gamma_2$ takes the form
\begin{align*}
\Gamma_2 = \int_{I_n} ([I_n^{GL}(\eta_t-\eta_t(t_{n-1}^{+}))]_t,\varphi)\dt = \int_{I_n}([I_n^{GL}(\int_{t_{n-1}}^t \eta_{tt}\ds)]_t,\varphi) \dt
\end{align*}
and applying an $H^1-L^2$ inverse property furthermore gives us 
\begin{align*}
\vert \Gamma_2\vert \le c\tau_n^{-1}\Vert I_n^{GL}(\int_{t_{n-1}}^t \eta_{tt}\ds) \Vert_{L^2(I_n;L^2)} \Vert \varphi \Vert_{L^2(I_n;L^2)}.
\end{align*}
Finally, using \eqref{3.15}, \eqref{3.16} and \eqref{ell_proj_err2} we obtain the following estimate for $\Gamma_2$
\begin{align*}
\vert \Gamma_2\vert & \le c\tau_n^{-1}\left(\Vert \int_{t_{n-1}}^t \eta_{tt}\ds \Vert_{L^2(I_n;L^2)}+ 
\tau_n\Vert (\int_{t_{n-1}}^t \eta_{tt}\ds)_t \Vert_{L^2(I_n;L^2)}
\right) \Vert \varphi \Vert_{L^2(I_n;L^2)} \\
& \le 
c\tau_n^{-1}\left(\tau_n \Vert \eta_{tt} \Vert_{L^2(I_n;L^2)}
\right) \Vert \varphi \Vert_{L^2(I_n;L^2)} 
\le
\left(
c h_n^4 \Vert u_{tt} \Vert_{L^2(I_n;H^4)}
\right) \Vert \varphi \Vert_{L^2(I_n;L^2)}
\end{align*}
and with this the proof is complete.
\end{itemize}
\end{proof}




\textbf{Stability}. Our aim now is to estimate $\vertiii{\tilde{E}}=(\Vert \Delta \tilde{E}_0\Vert^2 + \Vert \tilde{E}_1\Vert^2)^{1/2}$. 
We consider~\eqref{3.9} with the test function $\Phi=\mathcal{P}_t^{n,k-1} \mathcal{A}_h^n \tilde{E}\in V_{k-1}\times V_{k-1}$, where 
\begin{align*}
\mathcal{P}_t^{n,k-1}=\begin{pmatrix}P_{t}^{n,k-1} & 0 \\ 0 & P_t^{n,k-1} \end{pmatrix}
\end{align*}
and 
\begin{align*}
\mathcal{A}_h^n \tilde{E}=\begin{pmatrix}A_h^n \tilde{E}_0  \\  \tilde{E}_1 \end{pmatrix}.
\end{align*}

It obviously holds 
\begin{align}\label{3.21}
\int_{I_n} (\tilde{E}_t,\Phi)\dt = \frac{1}{2}\vertiii{\tilde{E}(t_{n})}^2 - \frac{1}{2}\vertiii{\tilde{E}(t_{n-1}^{+})}^2.
\end{align}


From~\eqref{2.6} and the definition of $\mathcal{A}_h^n\tilde{E}$ it further follows 
\begin{equation}\label{3.22}
\int_{I_n} ( \mathcal{L}_h^n \tilde{E},\mathcal{P}_t^{n,k-1} \mathcal{A}_h^n \tilde{E} ) \dt 
= \sum_{j=1}^k w_{n,j} ( \mathcal{L}_h^n \tilde{E},\mathcal{A}_h^n \tilde{E} ) (t_{n,j})=0.
\end{equation}

For this particular choice of $\Phi$, in the right-hand side of~\eqref{3.9} only $\varphi_1=P_t^{n,k-1}\tilde{E}_1$ is present and because $f$ 
is a Lipschitz function the first term is bounded by the expression $c \Vert \tilde{E}_0\Vert_{L^2(I_n;L^2)}\Vert \tilde{E}_1\Vert_{L^2(I_n;L^2)}$ and 
moreover by $c \Vert \Delta \tilde{E}_0\Vert_{L^2(I_n;L^2)} \Vert \tilde{E}_1\Vert_{L^2(I_n;L^2)}$.
Estimates for the remaining terms are given in Lemma~\ref{lemma3_3} (see~\eqref{3.12}--\eqref{3.14}) and considering also~\eqref{3.21} 
and \eqref{3.22} for $n=1,\ldots,N$ we obtain
\begin{equation}\label{3.23}
\vertiii{\tilde{E}(t_{n})}^2 \leq \vertiii{\tilde{E}(t_{n-1}^{+})}^2 + c \vertiii{\tilde{E}}^2_{L^2(I_n;L^2)} + c\left(\tau_n^{k+1}\mathcal{E}_t^n 
+ c h_n^4 \mathcal{E}_x^n\right)^2.
\end{equation}

Here we have used the notation 
\begin{equation*}
\begin{array}{lclcl}
\mathcal{E}_t^n &=& \mathcal{E}_t^n(u,k) & = & \Vert  \vert u^{(k+1)}\vert 
+ \vert u^{(k+2)} \vert + \vert u^{(k+3)}\vert + \vert \Delta u^{(k+2)}\vert + \vert \Delta^2 u^{(k+1)}\vert \Vert_{L^2(I_n;L^2)} \\
\mathcal{E}_x^n & = & \mathcal{E}_x^n(u) &= &\Vert \vert u \vert + \vert u_t \vert + \vert u_{tt} \vert \Vert_{L^2(I_n;H^4)}.
\end{array}
\end{equation*}



\begin{lemma}
For any $n$, $1\le n \le N$ and $\tau_n$ sufficiently small the following estimate is fulfilled
\begin{align}\label{3.24}
\vertiii{\tilde{E}}^2_{L^2(I_n;L^2)} \le c\tau_n \vertiii{\tilde{E}(t_{n-1}^{+})}^2 + c\tau_n \left( \tau_n^{k+1} \mathcal{E}_t^n + c h_n^4 
\mathcal{E}_x^n \right)^2.
\end{align}
\end{lemma}

\begin{proof}
Let $\bar{E}^{n,j}=\bar{\tau}_j^{-1/2} \tilde{E}^{n,j}$, $j=1,\ldots,k$. Noting that $\tilde{E}^{n,0}=\tilde{E}(t_{n-1}^{+})$, we obtain 
\begin{equation*}
\tilde{E}(\vec{x},t)=\sum_{j=0}^k \hat{l}_{n,j}(t) \tilde{E}^{n,j}(\vec{x})=\sum_{j=1}^k \hat{l}_{n,j}(t)\bar{\tau}_j^{1/2} \bar{E}^{n,j}(\vec{x}) + \hat{l}_{n,0}(t)\tilde{E}^{n,0}(\vec{x}).
\end{equation*}

Let $\Phi=\Phi_{\tilde{E}}:=\sum_{i=1}^k l_{n,i}(t) \bar{\tau}_i^{-1/2}\mathcal{A}_h^n \bar{E}^{n,i}$ in~\eqref{3.9}, then 
\begin{align*}
\int_{I_n}(\mathcal{L}_h^n \tilde{E}, \Phi_{\tilde{E}}) \dt = \sum_{j=1}^k w_{n,j} \bar{\tau}_j^{-1} 
(\mathcal{L}_h^n \tilde{E}^{n,j},\mathcal{A}_h^n \tilde{E}^{n,j})=0.
\end{align*}

Next, using~\eqref{2.7} and also Lemma~\ref{lemma2_1} we obtain 
\begin{align}\label{3.25}
\begin{aligned}
\int_{I_n}(\tilde{E}_t,\Phi_{\tilde{E}})\dt & = \sum_{i,j=1}^k \tilde{m}_{ij}\left((\Delta \bar{E}_0^{n,j},\Delta \bar{E}_0^{n,i})
+(\bar{E}_1^{n,j},\bar{E}_1^{n,i})\right) 
\\
& \quad + \sum_{i=1}^k m_{i0} \bar{\tau}_i^{-1/2} \left((\Delta \tilde{E}_0^{n,0},\Delta \bar{E}_0^{n,i})
+(\tilde{E}_1^{n,0},\bar{E}_1^{n,i})\right)
\\
& \ge c \sum_{j=1}^k \vertiii{\bar{E}^{n,j}}^2 - c\left( \sum_{j=1}^k \vertiii{\bar{E}^{n,j}}^2 \right)^{1/2}\vertiii{\tilde{E}(t_{n-1}^{+})}.
\end{aligned}
\end{align}

We have that $\sum_{j=1}^k \vertiii{\bar{E}^{n,j}}^2$ and $\sum_{j=1}^k \vertiii{\tilde{E}^{n,j}}^2$ are equivalent modulo 
constants depending only on $\bar{\tau}_i$, $i=1,\ldots,k$, see also~\cite{Karakashian2005convergence}, 
which means that we can exchange them in the above estimate. Moreover, cf.~\cite{Karakashian1999aspacetime}, it holds 
that 
\begin{equation}\label{3.26}
c_1\tau_n \sum_{j=1}^k \vertiii{\tilde{E}^{n,j}}^2 \le  \vertiii{\tilde{E}}_{L^2(I_n;L^2)}^2 \le  c_2 \tau_n \sum_{j=1}^k \vertiii{\tilde{E}^{n,j}}^2.
\end{equation}  

Similarly, as before, we estimate the terms on the right side of~\eqref{3.9} with $\Phi=\Phi_{\tilde{E}}$ by 
$ c\left(\tau_n^{k+1}\mathcal{E}_t^n + c h_n^4 \mathcal{E}_x^n\right)^2$ which along with~\eqref{3.25} and \eqref{3.26} imply~\eqref{3.24}.


\end{proof}

For $ M_n \geq 2 $, we define the expressions
\begin{align*}
    \beta_n = \frac{\gamma_n}{M_n - 1}, \qquad
    \gamma_n =
    \begin{cases}
        0, & \text{if } V_h^{n+1} = V_h^{n}, \\
        1, & \text{otherwise},
    \end{cases}
\end{align*}
for $n = 1, \ldots, N-1$.

\begin{lemma} 
    The estimate
    \begin{align}\label{3.27}
        \begin{aligned}
            \triplenorm{\tilde{E}(t_{n-1}^+)}^2 & \leq (1 + \beta_{n-1} + \tau_{n-1}) \triplenorm{\tilde{E}(t_{n-1})}^2 + \gamma_{n-1} (M_{n-1} + \tau_{n-1}) \triplenorm{J^{n-1}}^2 \\
            & \quad + c B_{n-1}^2
        \end{aligned}
    \end{align}
    is valid for $ n = 2, \ldots, N $, where we use $ J^n = (\omega(t_n^+) - \omega(t_n), \omega_t(t_n^+) - \omega_t(t_n)) $
    and $ B_n = \norm{(I - I_n^{GL}) \laplace u_t}_{L^2(I_n; L^2)} $, which satisfies
    \begin{align}
        \label{eq:laplace_approx}
        B_{n} \leq c \tau_{n}^{q+1} \norm{\laplace u^{(q+2)}}_{L^2(I_{n}; L^2)}.
    \end{align}
\end{lemma}
\begin{proof}
    First of all, we recall $ \triplenorm{\tilde{E}}^2 = \norm{\laplace \tilde{E}_0}^2 + \norm{\tilde{E}_1}^2 $ and $ \tilde{E} = U_{\tau,h} - W $. Using the identities
    $ W_1(t_{n-1}^+) = I_{n}^{GL} \omega_t(t_{n-1}^+) = \omega_t(t_{n-1}^+) $ and $ u_{\tau, h}^1(t_{n-1}^+) = P_h^{n} u_{\tau, h}^1(t_{n-1}) $, we obtain
    \begin{align*}
        \norm{\tilde{E}_1(t_{n-1}^+)} & = \norm{u_{\tau, h}^1(t_{n-1}^+) - W_1(t_{n-1}^+)} \\
        & \leq \norm{P_h^{n} u_{\tau, h}^1(t_{n-1}) - P_h^{n} \omega_t(t_{n-1})} + \norm{P_h^{n} \omega_t(t_{n-1}) - \omega_t(t_{n-1}^+)} \\
        & = \norm{P_h^{n} [u_{\tau, h}^1(t_{n-1}) - W_1(t_{n-1})]} + \norm{P_h^{n} [\omega_t(t_{n-1}) - \omega_t(t_{n-1}^+)]} \\
        & \leq \norm{\tilde{E}_1(t_{n-1})} + \norm{\omega_t(t_{n-1}^+) - \omega_t(t_{n-1})}.
    \end{align*}
    With the definition $ u_{\tau, h}^0(t_{n-1}^+) = R_h^{n} u_{\tau, h}^0(t_{n-1}) $, we have
    \begin{align*}
        \norm{\laplace \tilde{E}_0(t_{n-1}^+)} & = \norm{\laplace [u_{\tau, h}^0(t_{n-1}^+) - W_0(t_{n-1}^+)]} \\
        & = \norm{\laplace [R_h^{n} u_{\tau, h}^0(t_{n-1}) - W_0(t_{n-1}^+)]} \\
        & \leq \norm{\laplace R_h^{n}[u_{\tau, h}^0(t_{n-1}) - W_0(t_{n-1})]} + \norm{\laplace R_h^{n}[W_0(t_{n-1}) - W_0(t_{n-1}^+)]} \\
        & \leq \norm{\laplace \tilde{E}_0(t_{n-1})} + \norm{\laplace [W_0(t_{n-1}^+) - W_0(t_{n-1})]}.
    \end{align*}
    By using the definition of $ W_0 $ and the identity $ R_h^{n-1} u(t_{n-1}) = \omega(t_{n-1}) $ along with the Cauchy-Schwarz inequality, it follows that
    \begin{align*}
        & \quad \norm{\laplace[W_0(t_{n-1}^+) - W_0(t_{n-1})]} \\
        & \leq \norm{\laplace[W_0(t_{n-1}^+) - R_h^{n-1} u(t_{n-1})]} + \norm{\laplace R_h^{n-1} u(t_{n-1}) - W_0(t_{n-1})} \\
        %
        %
        & = \norm{\laplace[W_0(t_{n-1}^+) - R_h^{n-1} u(t_{n-1})]} + \norm{\laplace R_h^{n-1}[u(t_{n-1}) - \int_{t_{n-2}}^{t_{n-1}} I_{n-1}^{GL} u_t \dt - u(t_{n-2}^+)]} \\
        & \leq \norm{\laplace[\omega(t_{n-1}^+) - \omega(t_{n-1})]} + \norm{\laplace [u(t_{n-1}) - u(t_{n-2}^+) - \int_{t_{n-2}}^{t_{n-1}} I_{n-1}^{GL} u_t \dt]} \\
        & = \norm{\laplace[\omega(t_{n-1}^+) - \omega(t_{n-1})]} + \norm{\int_{t_{n-2}}^{t_{n-1}} (I - I_{n-1}^{GL}) \laplace u_t \dt} \\
        %
        %
        & \leq \norm{\laplace[\omega(t_{n-1}^+) - \omega(t_{n-1})]} + c \tau_{n-1}^{1/2} \norm{(I - I_{n-1}^{GL}) \laplace u_t}_{L^2(I_{n-1}; L^2)}
    \end{align*}
    Combining the estimates along with the arithmetic-geometric mean inequality yields the assertion.
\end{proof}

\begin{theorem}
\label{thm:3.1}
Let $ u $ be the solution of \eqref{eq:instationary_plate} with right-hand side $ f(u) $ and $ U_{\tau, h} = (u_{\tau, h}^0, u_{\tau, h}^1) $ be the discrete solution of \eqref{eq:adaptive_plate}. Then, the estimate
\begin{align}\label{error_semi_linear}
\max_{t\in[0,T]} \Vert E\Vert 
\le c \sum_{n=0}^{N-1} e^{c(T-t_n)} \left\{
\tau_n^{k+1} \mathcal{E}_t^n + h_n^4 \mathcal{E}_x^n
\right\}
+c e^{CT} \sqrt{N}_C \max_n \Vert J^n\Vert
\end{align}
holds where $ \mathcal{N}_C $ denotes the number of times where $ S_h^j \neq S_h^{j-1}, \, j = 1, \ldots, N-1 $. Furthermore, we get
\begin{align}
    \begin{aligned}
        \max_{t \in [0, T]} (\norm{u(t) - u_{\tau, h}^0} + \norm{u_t(t) - u_{\tau, h}^1})
        & \leq C [\max_{n} \tau_n^{k+1} C_t(u) + \max_{n} h_n^4 C_x(u) \\
        & \quad + \sqrt{\mathcal{N}_C} \max_{n} \triplenorm{J^n}]
    \end{aligned}
\end{align}
and
\begin{align}
    \max_{t \in [0, T]} \norm{u(t) - u_{\tau, h}^0}_{L^\infty} \leq C L_h [\max_{n} \tau_n^{k+1} C_t(u) + \max_{n} h_n^4 C_x(u) + \sqrt{\mathcal{N}_C} \max_{n} \norm{J^n}].
\end{align} 
\end{theorem}
\begin{proof}
    First of all, we have $ \tilde{E}(t_0^+) = (0, (P_h^1 - R_h^1) u_1) $. With that follows $ \triplenorm{\tilde{E}(t_0^+)} \leq c h_0^4 \mathcal{E}_x^1 $. Combining the estimates \eqref{3.23}, \eqref{3.24} and \eqref{3.27} yields
    \begin{align}
        \label{eq:recursion}
        \begin{aligned}
            \triplenorm{\tilde{E}(t_{n})}^2 & \leq (1 + c \tau_{n}) \left[(1 + \beta_{n-1} + \tau_{n-1}) \triplenorm{\tilde{E}(t_{n-1})}^2 + \gamma_{n-1} (M_{n-1} + \tau_{n-1}) \triplenorm{J^{n-1}}^2\right] \\
            & \quad + (1 + c \tau_{n}) [c B_{n-1}^2 + c(\tau_{n}^{k + 1} \mathcal{E}_t^n + h_n^4 \mathcal{E}_x^n)^2]
        \end{aligned}
    \end{align}
    for $ n = 1, \ldots, N $. Recursive application of the inequality \eqref{eq:recursion} together with \eqref{eq:laplace_approx} leads for $ n = 1, \ldots, N $ to
    \begin{align}
        \label{eq:recursion_result}
        \triplenorm{\tilde{E}(t_n)}^2 & \leq c \sum_{m = 0}^{n-1} C_{m, n-1} \left[(\tau_m^{k+1} \mathcal{E}_t^m + h_m^2 \mathcal{E}_x^m)^2\right.
        \left. + \gamma_m (M_m + \tau_{m-1}) \triplenorm{J^m}^2 \right],
    \end{align}
    where $ C_{m, n-1} = \prod_{j=m}^{n-1} (1 + c \tau_j) (1 + \beta_j + \tau_{j-1}) $. For fixed $ n $, let $ \mathcal{N}_C(n-1) $ be the number of times where $ S_h^j \neq S_h^{j-1}, \, j = 1, \ldots, n-1 $ holds. Then, we define $ M_m = M = \mathcal{N}_C(n-1) + 1, \, m = 1, \ldots, n-1 $. So we obtain $ \beta_j = \beta = \frac{1}{M-1} $ if $ S_h^j \neq S_h^{j-1} $ and $ \beta_j = 0 $ otherwise. Then, we have
    \begin{align*}
        C_{m, n-1} & \leq \prod_{j = m}^{n-1} (1 + c \tau_j) \prod_{\stackrel{j=m}{\beta < \tau_{j-1}}}^{n-1} (1 + 2 \tau_{j-1}) \prod_{\stackrel{j=m}{\beta \geq \tau_{j-1}}}^{n-1} (1 + 2 \beta) \\
        & \leq e^{c (t_n - t_m)} \cdot e^{2(t_{n-1} - t_{m-1})} \cdot (1 + 2 \beta)^{M - 1}
        \leq e^{c (t_{n} - t_{m-1})} \cdot e^2
    \end{align*}
    for $ 0 \leq m \leq n-1 $. Inserting this into \eqref{eq:recursion_result} yields
    \begin{align}\label{3.33}
        \max_{1 \leq n \leq N} \triplenorm{\tilde{E}(t_n)}^2 \leq c \sum_{n = 0}^{N-1} e^{c (T - t_n)} [\tau_n^{k+1} \mathcal{E}_t^n + h_n^4 \mathcal{E}_x^n]^2 + c e^{cT} \mathcal{N}_C \max_{n} \triplenorm{J^n}^2.
    \end{align}
    Consequently applying~\eqref{3.24}, \eqref{3.27}, \eqref{3.33} and \eqref{2.5} provides the final result.
\end{proof}

\end{document}